\documentclass[11pt,a4paper]{article}
\usepackage[margin=1in]{geometry}
\usepackage{amsmath,amssymb}
\usepackage{amsthm}
\usepackage{booktabs}
\usepackage{multirow}
\usepackage{array}
\usepackage{graphicx}
\usepackage{float}
\usepackage{algorithm}
\usepackage{algpseudocode}
\usepackage[numbers,sort&compress]{natbib}
\usepackage{hyperref}
\usepackage{cleveref}
\hypersetup{colorlinks=true, linkcolor=blue, citecolor=blue, urlcolor=blue}

\graphicspath{{figs/}{./}}

\newtheorem{theorem}{Theorem}[section]
\newtheorem{lemma}[theorem]{Lemma}
\newtheorem{proposition}[theorem]{Proposition}
\newtheorem{corollary}[theorem]{Corollary}

\newtheorem{remark}[theorem]{Remark}

\newcommand{\Lap}{\Lambda}
\newcommand{\Ebase}{u_{\mathrm{base}}}
\newcommand{\Rres}{\mathcal R_\theta}
\newcommand{\Ea}{E_\alpha}
\newcommand{\CD}{{}^{C}\!D_t^\alpha}

\begin{document}
\title{FrFNO: Injecting the Analytic Mittag--Leffler Propagator into a
Resolution--Robust Neural Operator for Space--Time Fractional PDEs}

\author{Guofei Pang\\
School of Mathematics, Southeast University, Nanjing 210096, China\\
\textit{Corresponding author}: guofei\_pang@seu.edu.cn}
\date{}

\maketitle

\begin{abstract}
Fractional partial differential equations couple a memory-dependent time
derivative with a nonlocal fractional Laplacian, and their repeated solution under
varying fractional orders, initial data, or diffusivity fields is computationally expensive. Neural
operators offer a fast surrogate, but standard architectures must relearn the
dominant linear fractional evolution from data. We propose the \emph{Fractional Fourier Neural Operator} (FrFNO), a resolution-robust conditional operator. It injects the analytic Mittag--Leffler response of the frozen linear part as a once-precomputed, resolution-independent propagator table, and trains a spectral convolutional network only on the residual induced by variable coefficients and nonlinear advection, conditioned continuously on both fractional orders. The central theoretical result is that under zero-shot super-resolution the injected propagator fills the out-of-band modes that a standard Fourier neural operator sets to zero, replacing the out-of-band full-field tail by the smaller residual tail (a $K$-independent constant-factor reduction under weak perturbation). The residual and the full solution carry the same Sobolev order, so the advantage is amplitude reduction rather than a steeper tail. Mesh refinement drives the error to a residual-controlled floor. On a nonlinear two-dimensional space--time fractional Burgers problem FrFNO achieves the lowest
relative $L^2$ error among six baselines (FNO, PINO, PDNO, CNO, DeepONet,
U-Net) at the training resolution and under zero-shot super-resolution, yields the smallest spectral phase error, and remains the best at the integer-order limit and for long integration windows. The same construction extends to coupled fractional systems such as fractional Allen--Cahn and Navier--Stokes equations. Code, training scripts, and reference outputs are available at
\url{https://github.com/Derek2021Pang/FrFNO}.
\end{abstract}


\medskip
\section{Introduction}\label{sec:intro}

Fractional evolution equations arise as the macroscopic limit of stochastic
processes with long memory or long-range jumps. In viscoelastic media,
power-law stress relaxation $G(t)\sim t^{-\alpha}$ with $0<\alpha<1$
interpolates between elastic and viscous response and is naturally described by
a time-fractional derivative~\cite{bagley1983,mainardi2001}. In heterogeneous
and porous media, continuous-time random walks with heavy-tailed waiting times
yield sub-diffusion governed by a Caputo time
derivative~\cite{metzler2000,podlubny1999}. L\'evy-type transport
and long-range spatial interactions produce heavy-tailed jump distributions
whose continuum limit is the fractional Laplacian $(-\Delta)^s$, a nonlocal
operator~\cite{benson2000,caffarelli2007,bucurvaldinoci2016,lischke2020whatis}. Replacing the
integer-order time derivative and the classical Laplacian by their fractional
counterparts therefore captures memory and non-locality that the classical
diffusion equation cannot, and the resulting space--time fractional PDEs form
a standard modeling framework across physics, mechanics, and
hydrology~\cite{kilbas2006}. We mainly focus on the two-dimensional space--time
fractional nonlinear equation, the space- and time-fractional viscous Burgers
model~\cite{sugimoto1991,momani2006,mao2017fb}
\begin{equation}\label{eq:fpde-intro}
{}^{C}D_t^{\alpha} u(x,y,t)
   = -\nu(x,y)\,(-\Delta)^{s} u(x,y,t)
      - \beta\, u\,\partial_x u ,\qquad (x,y)\in\Omega=(0,1)^2,\ t\in(0,T],
\end{equation}
with homogeneous Dirichlet boundary conditions, Caputo order $0<\alpha\le 1$,
fractional-Laplacian order $0<s\le 1$, a spatially variable diffusivity
$\nu(x,y)>0$, and advective strength $\beta\ge 0$. The classical integer-order
variable-coefficient viscous Burgers equation is recovered at $\alpha=s=1$.
\Cref{sec:breadth} further applies the identical construction to linear
variable-coefficient fractional diffusion and to fractional Fisher--KPP and
Allen--Cahn reaction--diffusion equations~\cite{bucurvaldinoci2016,dipierrovaldinoci2018}, demonstrating that the gain does not rely on the advective structure.

Two structural features make solving \cref{eq:fpde-intro} expensive.
The Caputo derivative~\cite{podlubny1999,kilbas2006} is a history convolution
($O(n^2)$ past steps for an $L_1$ scheme at step $n$), and the fractional
Laplacian $(-\Delta)^s$ is a global operator coupling all grid points. Recent
finite-difference, finite-element, and rational-approximation discretizations of
the fractional Laplacian on general
domains~\cite{huang2024gridoverlay,salgado2026semianalytic,hao2026simplefd}
reduce per-solve constants but retain the global coupling and, for time-fractional
evolution, the history cost. When
$(\alpha,s)$, $u_0$, or $\nu$ are swept (uncertainty quantification, inverse
identification, parameter studies, design optimization), this becomes a many-query
burden of thousands of solves. A surrogate that maps $(u_0,\nu,\alpha,s)$ directly
to $u(T)$, evaluable on a finer grid than it was trained on, is therefore
attractive.

Neural operators learn maps between function spaces. The Fourier neural operator
(FNO)~\cite{li2021fno,kovachki2023no} parameterizes layer-wise integral kernels by
truncated Fourier multipliers and admits cross-resolution evaluation (subject to
aliasing at high wavenumbers), with approximation and
aliasing theory developed in~\cite{kovachki2021ua,lanthaler2024disc}. DeepONet~\cite{lu2021deeponet}
uses a branch--trunk decomposition (a branch network encodes the input function on a
fixed sensor grid and a trunk network encodes the evaluation coordinates) and
approximates general nonlinear operators by the universal approximation theorem, but
the fixed grid prevents resolution-independent evaluation. Representation variants include localized,
geometry-aware, attention, and convolutional operators~\cite{liuschaffini2024lno,li2023geofno,cao2023gt,raonic2023cno,zang2025dgenno}.
FNO-based surrogates have also been embedded in optimal control and inverse-design pipelines~\cite{margenberg2024fno}.
All of these methods learn the full response from data. Physics-informed variants (PINNs~\cite{raissi2019pinn,zou2025multihead},
PINO~\cite{li2024pino}, PDNO~\cite{shin2024pdno}, and fractional
PINNs~\cite{pang2019fpinn,guo2022mcfpinn}) remain instance-level or full-solution solvers.

These methods are powerful, yet for \cref{eq:fpde-intro} they all force the
network to learn something that is, in fact, known in closed form.
For constant diffusivity $c$ ($\nu\equiv c$), the \emph{linear} part
$\CD u=-c(-\Delta)^s u$ is diagonal in the sine basis (for the spectral
Dirichlet fractional Laplacian), and each mode with
eigenvalue $\lambda$ of $-\Delta$ evolves
through the one-parameter Mittag--Leffler function~\cite{podlubny1999}
$E_\alpha(-c\,t^\alpha\lambda^{s})$. This Mittag--Leffler propagator encodes exactly the long memory
and the algebraic fractional decay. It is the part of the solution that is
simultaneously the most characteristic of fractional dynamics and the most awkward
for a generic network to reproduce. The premise of this paper is that the
dominant linear fractional response should be injected exactly. Only the
genuinely unresolved residual needs to be learned. This residual captures
the effect of variable coefficients and nonlinear
advection.

We build the \emph{Fractional Fourier Neural Operator} (FrFNO) around an exact,
parameter-free analytic base field. An offline table stores the discrete modal
propagator on a grid of the two orders $(\alpha,s)$ and a log-spaced
resolution-independent coordinate $z=\lambda^s$; at run time the base field is
a single diagonal spectral multiplier in the discrete sine transform (DST) basis,
$\Ebase=\mathrm{iDST}\big[g^{N_t}_{\alpha,s,c}\cdot\mathrm{DST}(u_0)\big]$, and a
spectral-convolutional residual network, continuously conditioned on $(\alpha,s)$
through Feature-wise Linear Modulation (FiLM)~\cite{perez2018film}, approximates only $r=u-\Ebase$.
The propagator is alias-free and tabulated on a resolution-independent
coordinate, so under mesh refinement it fills the out-of-band modes that a
matched FNO sets to zero. The residual carries the same Sobolev order as the
full field. Its tail is smaller by a constant factor, not steeper, and
with $K$ fixed the error converges to a constant floor that FrFNO places below a
matched FNO (\cref{sec:theory}).
The closest work in spirit is the differentiable Mittag--Leffler spectral layer
of DFSC~\cite{hu2026dfsc}, which implements the same ``known propagator plus
learned correction'' split for a finite-dimensional Caputo linear system.
FrFNO extends it to the field setting with a spatial fractional Laplacian,
continuous $(\alpha,s)$ conditioning, a resolution-independent table, and the
error-floor analysis above.

Our contributions are as follows.
\begin{enumerate}
\item \textbf{Method and analysis.} FrFNO injects the analytic Mittag--Leffler
propagator as an exact, $(\alpha,s)$-conditioned base field through a
resolution-independent table and learns only the residual. The analysis gives a
propagator bound via fractional Duhamel--Dyson expansion, a
resolution-independent propagator-table approximation, a same-order
residual-regularity result (the residual and the full field share one Sobolev
order), and a band-by-band FrFNO--FNO floor comparison showing that the
tabulated base field fills exactly the modes beyond the training band
(\cref{sec:theory}).
\item \textbf{Controlled verification and generality.} Against six matched baselines,
FrFNO is the best at every resolution for nonlinear Burgers, a four-times longer
time horizon, and the integer-order limit, and passes five bootstrap-tested theory checks
(\cref{sec:numerics}). The construction transfers to periodic fractional diffusion equations,
the Fisher--KPP and Allen--Cahn reaction--diffusion equations, and coupled
systems (including fractional Navier--Stokes) via a matrix-valued
propagator (\cref{sec:breadth}).
\item \textbf{Reproducibility.} We release code, pre-trained weights, and full
pseudocode. A controlled ablation cleanly isolates the gain attributable to the
analytic propagator. Systematic sweeps over seeds, order grids, and learning
rates confirm robustness across all hyperparameters.
\end{enumerate}

\section{Problem formulation and spectral setting}\label{sec:setup}

\subsection{The space--time fractional equation}

Let $\Omega=(0,1)^2$ and write $\boldsymbol{x}=(x,y)$. We mainly study the scalar
space--time fractional nonlinear equation
\begin{align}
{}^{C}D_t^{\alpha} u(\boldsymbol{x},t)
   &= -\nu(\boldsymbol{x})\,(-\Delta)^{s}u(\boldsymbol{x},t)
      - \beta\,u(\boldsymbol{x},t)\,\partial_x u(\boldsymbol{x},t),
      && \boldsymbol{x}\in\Omega,\ 0<t\le T, \label{eq:pde1}\\
u(\boldsymbol{x},t)&=0, && \boldsymbol{x}\in\partial\Omega,\ t\in[0,T],\label{eq:pde-bc}\\
u(\boldsymbol{x},0)&=u_0(\boldsymbol{x}), && \boldsymbol{x}\in\Omega,\label{eq:pde-ic}
\end{align}
with $0<\alpha\le 1$, $0<s\le 1$, $\beta\ge 0$, and a smooth, strictly positive
coefficient field $0<\nu_{\min}\le\nu(\boldsymbol{x})\le\nu_{\max}<\infty$. The
operator ${}^{C}D_t^\alpha$ is the Caputo derivative of order $\alpha$ and
$(-\Delta)^s$ is the spectral fractional Laplacian with homogeneous Dirichlet
data. We use a scalar nonlinearity (\(u\partial_x u\)) to keep the exposition clear. The method extends unchanged to vector-valued systems.

For $0<\alpha<1$ and a sufficiently regular $v$,
\begin{equation}\label{eq:caputo}
{}^{C}D_t^\alpha v(t)
 =\frac{1}{\Gamma(1-\alpha)}\int_0^t
   (t-\tau)^{-\alpha}v'(\tau)\,d\tau ,
\end{equation}
with the classical derivative recovered as $\alpha\to1^{-}$. The Caputo derivative
uses the ordinary initial value $v(0)=u_0$, which is the physically natural
choice~\cite{podlubny1999,kilbas2006}.

\paragraph{Spectral fractional Laplacian}
With Dirichlet data the negative Laplacian $-\Delta$ has eigenpairs
\begin{equation}\label{eq:eig-cont}
-\Delta\,\phi_{\boldsymbol{k}}=\lambda_{\boldsymbol{k}}\phi_{\boldsymbol{k}},
\qquad
\phi_{\boldsymbol{k}}(\boldsymbol{x})=2\sin(k_1\pi x)\sin(k_2\pi y),\quad
\lambda_{\boldsymbol{k}}=\pi^2(k_1^2+k_2^2),
\end{equation}
for $\boldsymbol{k}=(k_1,k_2)\in\mathbb N^2$. The spectral fractional Laplacian is
defined by functional calculus~\cite{dinezza2012},
\begin{equation}\label{eq:frac-lap}
(-\Delta)^s v=\sum_{\boldsymbol{k}\ge1}\lambda_{\boldsymbol{k}}^{s}\,
\widehat v_{\boldsymbol{k}}\,\phi_{\boldsymbol{k}},
\qquad
\widehat v_{\boldsymbol{k}}=\langle v,\phi_{\boldsymbol{k}}\rangle_{L^2(\Omega)}.
\end{equation}
The choice of the \emph{spectral} (Dirichlet) fractional Laplacian makes the
analytic propagator a cheap, alias-free operation (two DSTs). On the periodic torus
the spectral and integral (Riesz) operators coincide, and the method applies
verbatim after replacing the DST by an FFT. On a bounded domain, the
\emph{restricted integral} fractional Laplacian is a different operator. It
requires $u=0$ a.e. outside $\Omega$ and its eigenbasis is not the sine
basis. The cheap DST implementation therefore does not transfer verbatim.
\Cref{sec:breadth} reports the periodic Riesz realization.

\subsection{The linear propagator and the Mittag--Leffler function}

For constant $\nu\equiv c$ and $\beta=0$, \cref{eq:pde1} decouples
completely on the sine basis. Each coefficient solves the scalar fractional ODE
\begin{equation}\label{eq:scalar-linear}
{}^{C}D_t^\alpha \widehat u_{\boldsymbol{k}}(t)
 =-c\,\lambda_{\boldsymbol{k}}^{s}\widehat u_{\boldsymbol{k}}(t),
 \qquad \widehat u_{\boldsymbol{k}}(0)=\widehat u_{0,\boldsymbol{k}}.
\end{equation}
Its solution is expressed through the one-parameter Mittag--Leffler function
$\Ea(z)=\sum_{m=0}^\infty z^m/\Gamma(1+m\alpha)$,
\begin{equation}\label{eq:ml-sol}
\widehat u_{\boldsymbol{k}}(t)
 = \Ea\!\left(-c\,t^\alpha\lambda_{\boldsymbol{k}}^{s}\right)
   \widehat u_{0,\boldsymbol{k}}.
\end{equation}
As $\alpha\to1$, $\Ea(z)\to e^z$ and \cref{eq:ml-sol} recovers the classical
heat-semigroup solution; for $\alpha<1$ and large $z$, $\Ea(-z)\sim
(\Gamma(1-\alpha)z)^{-1}$, i.e.\ \emph{algebraic} decay. This algebraic decay is
a characteristic feature that distinguishes memory-dependent fractional diffusion from
classical exponential relaxation, and is built exactly into the injected
propagator and never has to be learned.

Split the operator as ${}^{C}D_t^\alpha u=A_0u+\big(A_1u-\beta\mathcal N(u)\big)$
with $A_0=-c(-\Delta)^s$, $A_1=-(\nu-c)(-\Delta)^s$ and $\mathcal N(u)=u\partial_xu$.
The fractional variation-of-constants (Duhamel) formula gives the mild solution
\begin{equation}\label{eq:mild}
u(t)=E_0(t)u_0+\int_0^t (t-\tau)^{\alpha-1}
E_{\alpha,\alpha}\!\left(-c(t-\tau)^\alpha(-\Delta)^s\right)
\big[A_1u(\tau)-\beta\mathcal N(u(\tau))\big]\,d\tau ,
\end{equation}
where $E_0(t)=E_\alpha(-ct^\alpha(-\Delta)^s)$ and $E_{\alpha,\alpha}$ is the
two-parameter Mittag--Leffler function. The first term is exactly the injected base
field. The Duhamel integral collects the variable-coefficient perturbation, the
nonlinear advection and their interaction. This is precisely the residual that FrFNO
learns. The decomposition is valid for any equation written as a (fractional) linear
principal part plus a remainder. A positive $\nu$ keeps that principal part
dissipative and the mild form well posed.

\subsection{Discrete sine transform and the discrete propagator}

On a uniform grid $x_i=i/N$, $y_j=j/N$, $i,j=0,\dots,N$, with $h=1/N$, zero
boundary values and $n=N-1$ interior points per axis ($(N+1)^2$ nodes in total,
$(N-1)^2=n^2$ interior), the orthonormal type-I discrete sine transform (DST-I)
diagonalizes the Dirichlet spectral calculus.
\begin{equation}\label{eq:dst}
\widetilde v_{k_1,k_2}=\frac{2}{n+1}\sum_{i,j=1}^{n}v_{ij}
\sin\!\frac{k_1\pi i}{n+1}\sin\!\frac{k_2\pi j}{n+1},\qquad k_1,k_2=1,\dots,n,
\end{equation}
with inverse $v_{ij}=\sum_{k_1,k_2}\widetilde v_{k_1,k_2}
\sin(k_1\pi i/(n+1))\sin(k_2\pi j/(n+1))$. The DST is separable and computed in
$O(N^2\log N)$ operations by the FFT.

The reference solver uses the five-point finite-difference Laplacian, whose interior
symbol is
\begin{equation}\label{eq:eig-disc}
q_{\boldsymbol{k}}^{(N)}
 =\frac{4}{h^2}\left[\sin^2\!\frac{k_1\pi h}{2}
                 +\sin^2\!\frac{k_2\pi h}{2}\right],\qquad h=\frac1N ,
\end{equation}
so that $q_{\boldsymbol{k}}^{(N)}=\lambda_{\boldsymbol{k}}+O(h^2)$. Using the
discrete symbol keeps the analytic base field consistent with the reference solver to
the solver's own order.

Partition $[0,T]$ uniformly with step $\tau=T/N_t$, $t_n=n\tau$. The $L_1$
approximation of the Caputo derivative is
\begin{equation}\label{eq:L1}
{}^{C}D_t^\alpha v(t_n)\approx
\frac{1}{\tau^\alpha}\sum_{j=0}^{n-1}b_j\big[v^{n-j}-v^{n-j-1}\big],
\qquad
b_j=\frac{(j+1)^{1-\alpha}-j^{1-\alpha}}{\Gamma(2-\alpha)},
\end{equation}
where $b_0=1/\Gamma(2-\alpha)$, $b_j>0$, and $b_j\sim j^{-\alpha}/\Gamma(1-\alpha)$.
This is the standard $L_1$ form, whose local truncation error and global error are
both $O(\tau^{2-\alpha})$ (and therefore first order, $O(\tau)$, at $\alpha=1$,
where it reduces to backward Euler). For the scalar linear mode \cref{eq:scalar-linear}, with
modal diffusivity $\mu=c(q_{\boldsymbol{k}}^{(N)})^s$, an implicit treatment of the fractional
diffusion gives the recurrence
\begin{equation}\label{eq:g-rec}
g^n=\frac{\displaystyle b_{n-1}g^0+
\sum_{\ell=1}^{n-1}(b_{n-\ell-1}-b_{n-\ell})g^\ell}
{\,b_0+\tau^\alpha\mu\,},\quad g^0=1, \, \mbox{where} \, g^k = v^k/v^0 .
\end{equation}
The terminal value $g^{N_t}_{\alpha,s,c}(\lambda)$ is the \emph{discrete
modal propagator}; it converges to the continuum modal multiplier
$\Ea(-cT^\alpha\lambda^s)$ at the $L_1$ order
$O(\tau^{2-\alpha})$~\cite{linxu2007,sunwu2006}. Note that \cref{eq:g-rec} is a
\emph{scalar} recurrence in the single variable $\mu=c\lambda^s$, which enables
offline tabulation of the discrete propagator.

\subsection{The reference nonlinear solver}

For the full problem \cref{eq:pde1} we use an IMEX/Picard scheme~\cite{linxu2007} that treats the
stiff fractional diffusion implicitly and the nonlinear advection explicitly, with
two Picard inner iterations. The upwind choice is necessary for small $s$ where
high-frequency modes are weakly damped. The dominant cost is the $L_1$ history,
which scales as $O(BN^2N_t^2)$ for a batch of $B$ fields. This is precisely the cost the learned surrogate
avoids at inference. The full discretization formula is in \cref{app:ref}. Training
labels use resolution--time-step pairs chosen so that the reference solution is
converged to well below the surrogate error being measured. For the held-out test
set, a single frozen ground truth is generated on a $513^2$ grid with $N_t=6400$
using a GPU-accelerated implementation of the above scheme ($48$ independent
samples), and then locally averaged to each evaluation grid ($17^2,65^2,129^2,257^2$).

Because the variable-coefficient, nonlinear problem admits no closed-form
solution, \emph{all training and test labels are numerical reference solutions}.
Three Mittag--Leffler objects must be kept apart. The first is the \emph{continuum analytic}
response $E_\alpha(-ct^\alpha\lambda^s)$ of~\cref{eq:ml-sol}. The second is its \emph{discrete
modal propagator} $g^{N_t}$ of~\cref{eq:g-rec}, an $O(\tau^{2-\alpha})$
approximation injected as the parameter-free base field. The third is the
\emph{full-problem numerical reference}, which is the only object used as a label. In the
purely linear limit the third object agrees with the injected
propagator to a relative $L^2$ error of $10^{-5}$--$10^{-6}$. Reported accuracies are always measured
against the converged nonlinear reference. Solver convergence verification is in the
supplementary material (Section S1).

\subsection{Data, conditioning variables, and error metric}

Initial fields $u_0$ are multiscale random sine fields with a
$(i^2+j^2)^{-1/2}$ high-frequency decay on the Dirichlet eigenbasis; after synthesis
each field is rescaled by its interior root-mean-square (RMS) amplitude to a target
value of $0.5$, so that all initial fields share a common amplitude scale and relative
errors are comparable across samples. The coefficient fields $\nu$ are smooth \emph{log-normal} random fields generated by a truncated Karhunen--Lo\`eve expansion on a cosine basis ($20$ modes, spectral decay $\beta_{\rm KL}=1$); writing the KL expansion as $g$, the diffusivity is
$\nu=\exp(g)/\frac{1}{|\Omega|}\int_\Omega \exp(g)\,dx$, which enforces strict positivity and unit spatial
mean. Writing $\bar\nu := \frac{1}{|\Omega|}\int_\Omega \nu\,dx$ for the spatial average, the normalization gives $\bar\nu=1$, so that the analytic base field is always evaluated at the uniform coefficient $\bar\nu=1$ and the
variable-coefficient perturbation is measured relative to a uniform baseline. Full generation formulas, parameters, and the training-time pairing of $(u_0,\nu,\alpha,s)$ are
given in~\cref{app:data}. The operator is conditioned on the
pair $(\alpha,s)$ on a uniform grid $\mathcal A\times\mathcal S$ (the main
experiments use $9\times9$ order settings). Each training field is paired with every
order setting. The network receives $(\alpha,s)$ as continuous FiLM inputs, so
it can be queried at off-grid orders.

For a prediction $\hat u$ against a reference $u$, we report the relative
$L^2$ error (in per cent),
\begin{equation}\label{eq:metric}
\frac{\|\hat u-u\|_2}{\|u\|_2}\times100\%,
\end{equation}
averaged over a fixed held-out test set. ``Zero-shot super-resolution'' means
that the network is trained on $17^2$ fields and evaluated, without retraining
or fine-tuning, on $65^2$ and $129^2$ (and, in an additional test, $257^2$)
fields. At an evaluation grid with $N$ intervals per axis, the network produces
a prediction $\hat u^{(N)}$ directly on that grid (the spectral operator is
resolution-invariant), and the reference $u^{(N)}$ is obtained by locally
averaging a frozen $513^2$ high-resolution truth to the same $N$-grid. The
error in \cref{eq:metric} is then computed with both fields on the common
$N$-grid.

\section{The FrFNO method}\label{sec:method}

\subsection{Operator decomposition: exact base plus learned residual}

The central design is an additive decomposition of the solution operator. For a
query $(u_0,\nu,\alpha,s)$ at terminal time $T$, write
\begin{equation}\label{eq:decomp}
\mathcal G(u_0,\nu,\alpha,s)=u(T)
 =\underbrace{\mathcal P_{\alpha,s,T}u_0}_{u_{\rm base}:\ \text{exact linear propagator}}
  +\underbrace{\mathcal R(u_0,\nu,\alpha,s)}_{r:\ \text{unresolved residual}} .
\end{equation}
The base operator $\mathcal P_{\alpha,s,T}$ is the frozen-coefficient linear
Mittag--Leffler propagator of \cref{sec:setup}. It is known, has no
trainable parameters, and is applied as a diagonal DST multiplier. It is the exact
first term of the fractional Duhamel formula~\cref{eq:mild}, while the network
parameterizes the Duhamel integral. The residual $r$ collects everything the frozen
propagator does not contain, the effect of the variable coefficient, of the
nonlinear advection, and of their interaction. FrFNO parameterizes only $r$ by a
neural network $\Rres$ and trains the surrogate
\begin{equation}\label{eq:surrogate}
\hat u=\Ebase+\Rres(u_0,\nu;\alpha,s),\qquad
\Ebase=\mathrm{iDST}\big[g^{N_t}_{\alpha,s,\bar\nu}(q^{(N)}_{\boldsymbol{k}})\,\mathrm{DST}(u_0)\big],
\end{equation}
where $g^{N_t}_{\alpha,s,\bar\nu}(q^{(N)}_{\boldsymbol{k}})$ is the discrete modal multiplier
\cref{eq:g-rec} evaluated on the five-point eigenvalues \cref{eq:eig-disc}. The
base field is detached from the autograd graph and $\Rres$ is trained only. This
is fundamentally different from learning the full solution (an FNO). Since the long-memory, algebraic-tail, high-frequency response is removed from the learning
target, the network sees a smoother, smaller, and weakly nonlinear residual.
\Cref{sec:theory} makes this statement quantitative.

\subsection{Resolution-independent propagator table}\label{sec:table}

Evaluating \cref{eq:g-rec} for every mode and every query would be wasteful. We
precompute a lookup table once offline. The variable on which the multiplier depends
is the single scalar $\mu=\bar\nu\,\lambda^{s}$; equivalently, after fixing
$\bar\nu$, the multiplier is a smooth function of the coordinate $z=\lambda^{s}$. We
therefore discretize the three coordinates separately. The $z$-grid
$\{z_m\}_{m=1}^{M}$ is dense near $z=0$, where $s<1$ produces an algebraic
singularity, and uniform thereafter. The order grids are
$\{\alpha_i\}_{i=1}^{n_a}$ and $\{s_j\}_{j=1}^{n_s}$. We store
\begin{equation}\label{eq:table}
\mathcal T[i,j,m]=g^{N_t}_{\alpha_i,s_j,\bar\nu}(z_m),
\end{equation}
with each entry computed by the scalar recurrence \cref{eq:g-rec}. At run time, given the
two-dimensional eigenvalue field $q^{(N)}_{\boldsymbol{k}}$, we compute
$z_{\boldsymbol{k}}=(q^{(N)}_{\boldsymbol{k}})^{s}$ and obtain the multiplier field by trilinear interpolation in
$(\alpha,s,z)$.
\begin{equation}\label{eq:propat}
g^{N_t}_{\alpha,s,\bar\nu}(q^{(N)}_{\boldsymbol{k}})
 =\operatorname{Interp}_{\mathcal T}\big(\alpha,s,(q^{(N)}_{\boldsymbol{k}})^s\big).
\end{equation}

Storing one multiplier per grid mode would yield $O(N^2)$ distinct entries.
The five-point symbol \cref{eq:eig-disc} mixes two sine arguments, so the table
would be resolution-dependent. The smooth coordinate $z=\lambda^s$ removes this
dependence. We use a fixed one-dimensional log-spaced grid with $M=1600$ points
by default and $M=400$ in the fast setting. This grid is independent of $N$ and
resolves the multiplier $g_{\alpha,s,\bar\nu}^{N_t}(z)$ to machine precision.
A single cached table serves all evaluation grids from $17^2$ to $257^2$.
We validate the table against a direct recurrence. The
maximum absolute discrepancy between the tabulated propagator and a direct
recurrence is $2.2\times10^{-4}$. It is dominated by
interpolation and remains fully below the surrogate error. The one-time build
costs $O(n_an_sMN_t)$ scalar steps. On an RTX~3090 with $41\times41$ orders,
$M=1600$, and $N_t=400$, this takes about $5$\,s. The table is cached globally.
Inference at any resolution costs only $O(N^2)$ scalar trilinear
interpolations. One table thus serves every sample and every evaluation grid.

\subsection{Extension to coupled systems}\label{sec:sysmethod}

The same construction applies verbatim to systems. Consider an $m$-component
equation whose frozen linear principal part is
\begin{equation}\label{eq:syslin}
{}^{C}D_t^\alpha \mathbf u
=-\mathbf D(-\Delta)^s\mathbf u-\mathbf J\mathbf u,
\end{equation}
where $\mathbf D\in\mathbb R^{m\times m}$ is the diffusion matrix and
$\mathbf J\in\mathbb R^{m\times m}$ is the zero-order reaction or coupling
matrix. On the sine basis each mode $\widehat{\mathbf U}_{\boldsymbol k}$
satisfies the matrix ODE
\begin{equation}\label{eq:sysmode}
{}^{C}D_t^\alpha\widehat{\mathbf U}_{\boldsymbol k}
=-\mathbf L_{\boldsymbol k}\widehat{\mathbf U}_{\boldsymbol k}
+\widehat{\mathcal N}_{\boldsymbol k},
\qquad
\mathbf L_{\boldsymbol k}=\mathbf D\lambda_{\boldsymbol k}^s+\mathbf J.
\end{equation}
The frozen response is the matrix-valued Mittag--Leffler propagator
$E_\alpha(-t^\alpha\mathbf L_{\boldsymbol k})$. In all applications here
$m$ is small ($m=2$ or $3$), so we diagonalize $\mathbf L_{\boldsymbol k}$
once per mode, offline, and evaluate the matrix function on its eigenvalues.
The resulting propagator remains a diagonal multiplier in the DST basis and
therefore retains the alias-free, resolution-independent property. Variable
coefficients, inter-mode coupling beyond $\mathbf J$, and all nonlinear terms
stay in the residual. The propagator-error bound \cref{prop:prop} and the
resolution-consistency / error-floor result \cref{prop:floor} carry over to
$(L^2)^m$ whenever $\mathbf L_{\boldsymbol k}$ is uniformly dissipative, that
is, its eigenvalues have uniformly positive real parts. This covers
linearized reaction--diffusion with cross-diffusion, two-component vector
Burgers, non-symmetric coupled systems, and the vorticity formulation of
fractional Navier--Stokes (scalar in 2D, three-vector with vortex stretching
in 3D).

\subsection{The conditional residual network}\label{sec:arch}

\Cref{fig:arch-frfno} gives a schematic overview of the FrFNO architecture.
FrFNO maps a thirteen-channel input tensor to a single residual field, while being continuously
conditioned on the order pair $(\alpha,s)$ via FiLM modulation. The matched FNO baseline uses
the identical spectral backbone but takes only two explicit channels $(u_0,\nu)$ and regresses
the full solution directly rather than learning a residual on top of the analytic base field.
We describe each stage in detail.

\begin{figure}
\centering
\includegraphics[width=0.74\textwidth]{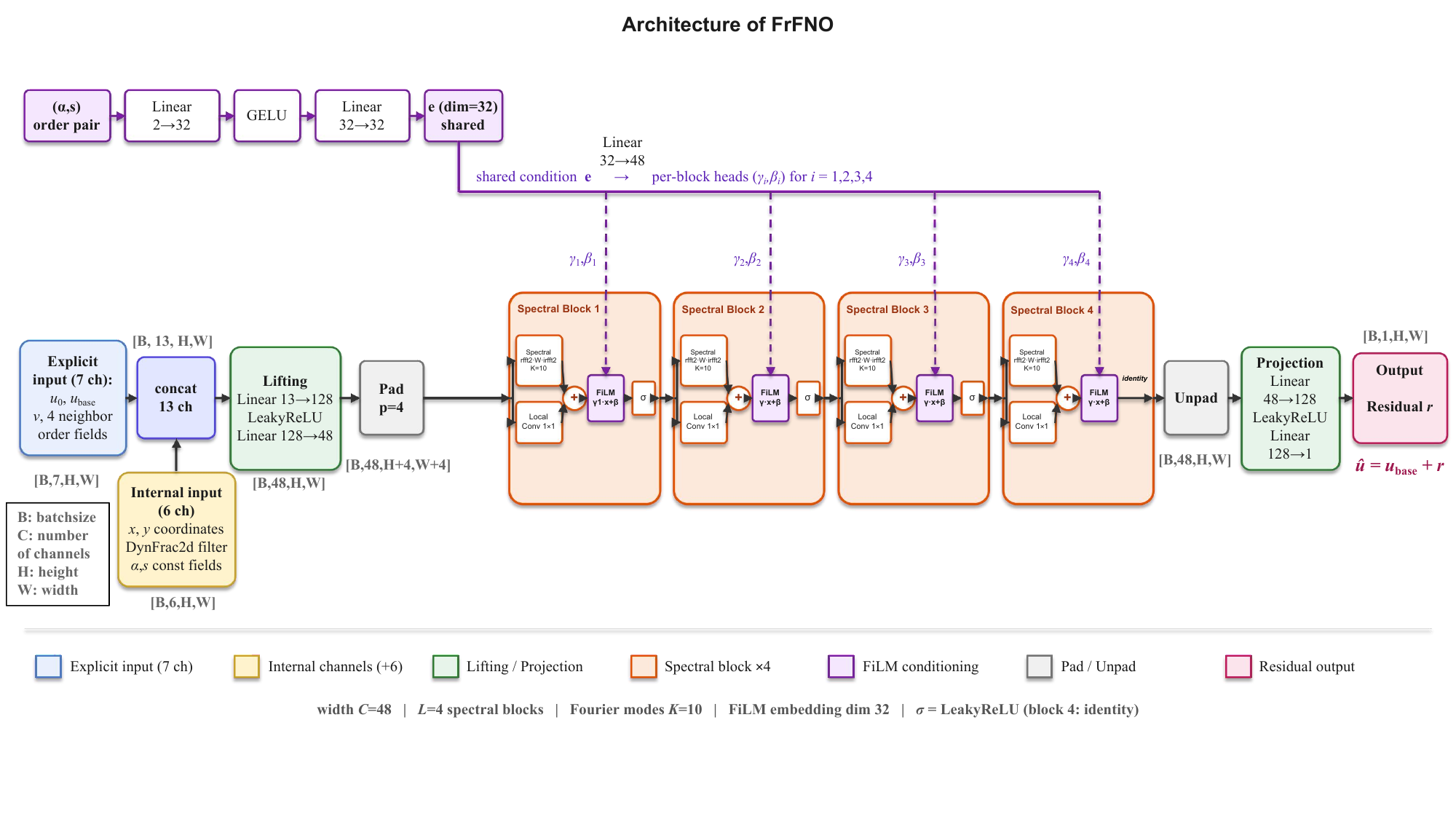}
\caption{FrFNO architecture. A thirteen-channel input (seven explicit fields:
$u_0$, $u_{\mathrm{base}}^{(\alpha,s)}$, $\nu$, and four neighbor-order bases; plus
six internal channels: coordinates, DynFrac2d features, and $\alpha/s$ constant
fields) is lifted to width $48$ and passed through four spectral blocks with FiLM
conditioning. The network outputs a normalized residual added to the analytic base
field.}
\label{fig:arch-frfno}
\end{figure}

\paragraph{Input and output}
The explicit input has seven channels,
\begin{equation}
  x_{\rm in} = [u_0,\; u_{\rm base}^{(\alpha,s)},\; \nu,\;
  \{u_{\rm base}^{(\alpha\pm\Delta\alpha,s)},\;
  u_{\rm base}^{(\alpha,s\pm\Delta s)}\}],
\end{equation}
namely the current base, $\nu$, and four neighbor-order bases used as local
context. Six internal channels (two coordinates, two DynFrac2d fractional features
(a fractional low-pass filter described in~\cref{app:arch}), and two
constant $\alpha,s$ fields) are appended before lifting, for $13$ channels
$(B,13, H,W)$. A pointwise two-layer perceptron ($13\to128\to C=48$, equivalent to
$1\times1$ convolutions) lifts the input; it is zero-padded by $p=4$ to avoid
spectral-convolution boundary artefacts, passed through four spectral blocks, and
un-padded. A symmetric $48\to128\to1$ perceptron then returns the standardised
residual $r_{\rm norm}=(u(T)-u_{\rm base})/\sigma$, where $\sigma$ is the
standard deviation of the training-set solutions $u(T)$ (a fixed global scalar).
This normalization keeps the regression target at $O(1)$, matching the output
scale of the default network initialization and avoiding the vanishing gradients
that would arise from directly predicting the much smaller residual. It also
makes the FrFNO loss directly comparable to the matched FNO, which regresses
$(u-\mu)/\sigma$ with the same $\sigma$. The prediction is recovered at
inference as
\begin{equation}
  \hat u = u_{\rm base} + \sigma\, r_{\rm norm}.
\end{equation}

\paragraph{Spectral block and FiLM conditioning}
Block $i$ computes
\begin{equation}
  v_i = \gamma_i\odot\bigl(\operatorname{SpecConv}_i(v_{i-1})
  + \operatorname{Conv}_{1\times1,i}(v_{i-1})\bigr) + \beta_i,
\end{equation}
followed by LeakyReLU for $i<L$, where $L=4$ is the number of spectral blocks.
The spectral path retains $K=10$ Fourier modes per
axis ($\mathrm{rfft2}$, learned complex weights
$W_i\in\mathbb C^{C\times C\times K\times K}$ with separate positive/negative
horizontal bands, $\mathrm{irfft2}$). The $1\times1$ path carries the discarded
high modes, and the two are added before FiLM. FiLM maps $(\alpha,s)$ through a
shared $2\to32\to32$ GELU encoder to an embedding $e$, with per-block heads
$\gamma_i=1+\mathrm{Linear}_{32\to C}(e)$ and
$\beta_i=\mathrm{Linear}_{32\to C}(e)$ broadcast over space, which supplies smooth
off-grid-order interpolation. Full layer dimensions are in~\cref{app:arch}.

The full FrFNO has $1.88\times10^6$ trainable parameters, dominated by the four
spectral-convolution weight tensors. The matched FNO baseline
shares the identical backbone (the same DynFrac2d fractional features, constant order
fields, FiLM conditioning, spectral blocks, width, depth, and optimizer) but takes
only two explicit channels $(u_0,\nu)$ (no analytic base or neighbor-order fields) and
predicts the mean-centred full solution $(u-\mu)/\sigma$ rather than the residual
$(u-u_{\rm base})/\sigma$, where $\mu$ is the training-set mean. Because the
lifting-layer input dimension differs (eight vs.~thirteen channels), the FNO
parameter count is $1.88\times10^6$ as well, differing by only $640$ weights,
negligible relative to the total. This matched control (identical capacity and
optimizer, with only the physical injection and output paradigm changed) allows
the gain to be attributed cleanly to the analytic propagator rather than to
architecture or capacity.

\subsection{Algorithms}\label{sec:alg}

The training procedure, zero-shot inference, and offline construction of the
resolution-independent propagator table are summarized in Section S4.5 of the
supplementary material. In brief, training generates reference pairs at the low
resolution $N_0=17$, computes the analytic base field from the cached propagator
table, and optimizes the residual network by Adam with cosine annealing. Inference at
any resolution $N$ applies the same base-field construction on the $N$-grid and
evaluates the fixed-$K$ residual network, with no retraining.

\subsection{Computational complexity and temporal super-resolution}\label{sec:complexity}

Let $N$ be the number of grid intervals per axis, $K$ the number of retained Fourier modes per axis, $L$ the
number of spectral blocks, $B$ the batch size, $N_t$ the number of reference
time steps, and $S_{\rm tr}$ the number of training steps. All timings below are measured on an NVIDIA RTX~3090 GPU. The reference
solver dominates because of the $L_1$ history,
which scales as $O(BN^2N_t^2)$. A single nonlinear reference field costs $1.73$~s at $17^2$,
$6.98$~s at $65^2$ and $13.92$~s at $129^2$ (\cref{tab:cost-perquery}). Standard operators require training labels at the evaluation resolution.
Generating all $10\,368$ training fields at $129^2$ would take about
$40$~GPU-hours. FrFNO trains only at $17^2$, so the entire training set is
generated in $22$~s, and zero-shot super-resolution removes the need for
high-resolution labels. The marginal price of the injected
analytic propagator is a one-time, seconds-scale propagator table ($\approx5$~s,
cached for all resolutions) and a base-field application that is two DSTs,
$O(N^2\log N)$ (milliseconds). The residual forward pass costs
$O(L(K^2+N^2\log N))$ ($<0.1$~s per field), and training costs
$O(S_{\rm tr}B\,L(K^2+N_0^2\log N_0))$ ($\approx260$~s
for $8000$ steps).

Because the table stores the multiplier as a function of $T$ (equivalently of
$N_t$), the base field is exact for \emph{any} horizon with no accumulated
time-stepping error. The correct way to predict a long horizon is therefore a single
one-shot evaluation at the target $T$, \emph{not} an autoregressive rollout of
short-horizon predictions. Rolling out compounds the residual network's error and,
for a memory equation, reintroduces the step-by-step history bookkeeping that the
one-shot propagator table eliminates. \Cref{sec:longT} confirms this empirically.

Full reproducibility detailCs (software versions, optimizer settings, random seeds, checkpointing, and CNO training prescription) are collected in Section S4 of the
supplementary material.

\section{Error analysis}\label{sec:theory}

The central theoretical claim of this paper is that the analytic
Mittag--Leffler propagator, tabulated as a resolution-independent spectral
multiplier, is defined on \emph{every} mode at \emph{every} grid. A standard
Fourier neural operator retains only $K$ spectral modes and sets every mode
outside its training band to zero. When the evaluation grid is refined from
$17^2$ to $129^2$, the reference solution acquires many high-wavenumber
components that the network never saw during training and cannot represent. An
FNO that predicts the full solution $u$ loses the entire out-of-band
content of $u$, and this content grows as the grid is refined. FrFNO instead
injects the analytic linear propagator, tabulated to a controlled defect
$\varepsilon_{\rm tab}$, on all modes including the modes that did
not exist on the training grid. It then trains a residual network only inside the
retained band. Its out-of-band error is therefore the tail of the residual
$r=u-u_{\rm base}$, not the tail of the full field $u$.

We make this precise through three results. First, the tabulated multiplier has
an interpolation defect that is independent of the evaluation resolution
(\cref{prop:tab}). Second, the residual $r$ and the full field $u$ carry the
\emph{same} Sobolev order. A variable-coefficient perturbation of the same order as
the fractional principal part consumes the $2s$ derivatives supplied by the
fractional heat semigroup, so the tail amplitude is reduced only by a constant
factor set by the perturbation strength (\cref{prop:reg}).
Third, the fixed-$K$ error floors nevertheless differ term by term, because on
the out-of-band modes the FNO predicts zero while FrFNO predicts the exact
propagator. This band filling, rather than residual smoothing, is what makes
FrFNO degrade much less sharply under zero-shot super-resolution
(\cref{prop:floor}). The rigorous comparison is on the out-of-band term. The
retained-band term is a conventional approximation/generalization error that we
compare numerically.

This section states the main error results. All auxiliary lemmas, propositions,
corollaries and their proofs are collected in~\cref{app:proofs}. We work on
$\mathbb T^d$ ($d=2$ in the numerics. The Dirichlet problem differs only by the
fixed factor $\pi^2$ in the eigenvalues, absorbed into the constants). We split
$\nu=c+\delta\nu$ with $\eta=\|\delta\nu\|_{L_\infty}/c$, and assume
\textbf{(H1)}~$0<\nu_{\min}\le\nu\le\nu_{\max}<\infty$, $\nu\in W^{2,\infty}$,
\textbf{(H2)}~$\alpha\in(0,1]$, $s\in(0,1]$, and \textbf{(H3)}~$u_0\in H^\gamma$,
$\gamma\ge2s$. The frozen propagator
$E_0(t)=\Ea(-ct^\alpha\Lap^{2s})$, where $\Lap=(-\Delta)^{1/2}$ so that
$\Lap^{2s}=(-\Delta)^s$, is contractive on every $H^\gamma$ and obeys the
standard Mittag--Leffler estimates (ML1)--(ML3) and the analytic-regularization
bound $\|\Lap^{2s}E_0(t)\|\le C_\alpha/(ct^\alpha)$ (\cref{app:ml}).

\subsection{Three structural results}\label{sec:th-main}

We build up to the main conclusion in three steps. Each step is a separate
proposition. The final comparison \cref{eq:compare} follows by combining them.

\emph{(i) Propagator error.} A fractional Duhamel--Dyson scattering expansion
$E=E_0+D_1+D_2+\cdots$ bounds the injected base-field error. Moving $\Lap^{2s}$ onto
the smooth data-carrying factor avoids the singular regularization estimate. For
generic data the only time singularity is the integrable kernel
$r^{\alpha-1}(T-r)^{-\alpha}$ ($B(\alpha,1-\alpha)=\pi/\sin\pi\alpha$), and the
Dirichlet-convolution identity rules out factorial growth of higher terms. Combining
\cref{lem:l1,lem:eig} with \cref{prop:prop} gives
\begin{equation}\label{eq:eprop-main}
\varepsilon_{\rm prop}\le
C_\alpha\eta cT^\alpha\|u_0\|_{H^\gamma}
+C_\alpha\beta T^\alpha M_1
+C_\tau\tau^{2-\alpha}+C_h h^2 ,
\end{equation}
with $M_1=C_S\|u_0\|_{H^\gamma}^2$ (nonlinear closure needs $\gamma>d/2$). Full statements
are \cref{eq:properr}--\cref{eq:eprop} in~\cref{app:proofs}. This error is
small in the weakly nonlinear regime.

\emph{(i$'$) The propagator table is resolution-independent.} The modal
propagator depends on the grid only through one scalar stiffness coordinate,
$z_{\boldsymbol k}^{(N)}=cT^\alpha(q_{\boldsymbol k}^{(N)})^s$, through
$g_\alpha(z)=\Ea(-z)$. A one-dimensional log-spaced table of $M$ values of
$g_\alpha$, reconstructed by $p$-th order interpolation, has uniform defect
$\varepsilon_{\rm tab}\le C M^{-p}$ with \emph{no} dependence on the evaluation
grid $N$ (\cref{prop:tab}). The base multiplier and the reference solver use the
same five-point symbol, so the resulting $O(h^2)$ eigenvalue-consistency error is
identical in the label and the base and cancels in the residual. It survives only
at the reference-solver level, see Section S1 of the supplementary material. One
cached table therefore
serves $17^2,65^2,129^2,257^2$ with the same defect. This is the property that
allows modes absent from the training grid to be filled at query time. It
contrasts with an FNO, whose spectral weights are stored by discrete mode index
on the training grid, address a different physical eigenvalue on another grid,
and vanish identically beyond $K$ modes.

\emph{(ii) The residual has the same Sobolev order as the full field.}
The elliptic estimate of the frozen semigroup,
$\|E_0(t)f\|_{H^{\gamma+2s}}\le C(1+(ct^\alpha)^{-1})\|f\|_{H^\gamma}$, supplies
$2s$ derivatives whenever $E_0$ acts directly on the data. This gain is
\emph{cancelled} in the scattering term because the perturbation
$A_1=-\mathcal M_{\delta\nu}\Lap^{2s}$ itself carries the factor $\Lap^{2s}$.
Along the integrand, $u_0\in H^\gamma\to E_0u_0\in H^{\gamma+2s}\to
\Lap^{2s}E_0u_0\in H^\gamma$, and $\Lap^{2s}E_0(t)$ is order zero on the high
modes (its symbol $\lambda^s\Ea(-ct^\alpha\lambda^s)\le C_\alpha/(ct^\alpha)$ is
independent of $\lambda$). Extracting the gain from the outer semigroup instead
produces the non-integrable weight $\rho^{-1}$. Hence, up to an arbitrarily
small endpoint loss, $r=u-u_{\rm base}\in H^\gamma$, the same order as $u$, and
\begin{equation}\label{eq:k2sratio-main}
\sigma_K(r)\le C_T(\eta+\beta T^\alpha M_1)\,K^{-\gamma}\|u_0\|_{H^\gamma},
\qquad
\frac{\sigma_K(r)}{\sigma_K(u)}\le C(\eta,\beta,T),
\end{equation}
a constant factor, independent of $K$ (\cref{prop:reg}). It is strictly below one
for weak perturbations and short horizons, but the ratio is a constant independent
of $K$. A genuine $2s$-order gain appears only for a perturbation strictly lower
order than the principal part (or for $\delta\nu=0,\beta=0$, where $r\equiv0$).
What the injection removes is therefore the dominant \emph{amplitude} of the
linear evolution, including its algebraic memory, not $2s$ derivatives of the
residual. For orientation, the propagator $\Ea(-ct^\alpha\lambda^s)$
transitions from near-unity to strong decay at $ct^\alpha\lambda^s\sim1$,
i.e.\ $\lambda\sim|\boldsymbol k|^2$, giving a dissipation cutoff
$k_c(t)\asymp(ct^\alpha)^{-1/(2s)}$. This is the wavenumber at which the modal multiplier
$\Ea(-ct^\alpha k^{2s})$ crosses from order unity to small, with $\asymp$ meaning
equality up to dimensionless constants. Modes beyond this cutoff are damped by the injected
propagator. This describes the base field, not a tail-rate result for $r$.

\emph{(iii) Band filling sets the zero-shot floor.}
Let $N_0$ denote the training resolution (interior points per axis), $K$ the
retained modes, and $N\ge N_0$ the query resolution. We split the spectrum into the
retained band $|\boldsymbol k|<K$ and the out-of-band set $K\le|\boldsymbol
k|\le N$, with tail energy $\sigma_{K,N}$. The injected base is a diagonal,
alias-free DST multiplier on \emph{every} mode (\cref{lem:alias,prop:tab}),
whereas the FNO spectral convolution is zero on $|\boldsymbol k|\ge K$. Parseval
then decomposes the squared error band by band (\cref{prop:floor}).
\begin{equation}\label{eq:floor-main}
\begin{aligned}
\varepsilon_{\rm Fr}^2(N)
&=\|\mathcal P_K(\Rres-r)\|^2+\sigma_{K,N}^2(r)
+\varepsilon_{\rm tab}^2\|u_0\|^2 ,\\
\varepsilon_{\rm FNO}^2(N)
&=\|\mathcal P_K(\widetilde{\Rres}-u)\|^2+\sigma_{K,N}^2(u)
+\varepsilon_{\rm pw}^2 .
\end{aligned}
\end{equation}
The first term on each line is the retained-band ($|\boldsymbol k|<K$)
learning error. The second is the out-of-band tail. On the FrFNO line the table
term $\varepsilon_{\rm tab}^2\|u_0\|^2$ is the upper bound from \cref{prop:tab}.
The cross term between the table error and the residual is $O(\varepsilon_{\rm tab})$
by Cauchy--Schwarz and does not affect the $N$-independent comparison (see
Section S2 of the supplementary material).
The decisive term is the out-of-band one. For $N>N_0$ the modes
$N_0\le|\boldsymbol k|\le N$ were absent at training time. The FNO predicts zero
there and pays the full energy $\sigma_{K,N}(u)$, whereas FrFNO supplies the
tabulated propagator on all of them and pays only the residual tail
$\sigma_{K,N}(r)$ plus the $N$-independent table defect. With $K$ fixed and
$N\to\infty$, $\sigma_{K,N}\to\sigma_K$ while the table defect and the learning
terms stay constant, so both errors converge to $N$-independent floors,
$\varepsilon_\infty^{\rm FrFNO}=\varepsilon_{\rm low}^{\rm Fr}+\sigma_K(r)+
\varepsilon_{\rm tab}\|u_0\|$ and
$\varepsilon_\infty^{\rm FNO}=\varepsilon_{\rm low}^{\rm FNO}+\sigma_K(u)+
\varepsilon_{\rm pw}$. The two tails have the same Sobolev order but
$\sigma_K(r)<\sigma_K(u)$ in the weakly perturbed regime (\cref{prop:reg}). The
realisable bandwidth is $\min(K,N_0)$ (\cref{cor:plateau}), so zero-shot
super-resolution means non-degradation under mesh refinement, not convergence to
zero error. Crucially, \cref{prop:reg} shows the residual is \emph{not} smoother
than the full field. It carries the same Sobolev order. The advantage therefore
comes not from a steeper residual tail but from the table filling the out-of-band
modes, so the tail energy paid is $\sigma_K(r)$ (smaller amplitude) rather than
$\sigma_K(u)$ (full amplitude). This is band filling, not residual smoothing.

\subsection{End-to-end bound and the comparison with an FNO}\label{sec:th-total}

\begin{theorem}[End-to-end FrFNO error]\label{thm:total}
Under (H1)--(H3) and \cref{prop:prop,prop:tab,prop:floor}, for a residual
network of width $w$, depth $L$, $K$ modes trained on $n_{\rm tr}$ independent samples, in
expectation,
\begin{equation}\label{eq:total}
\mathbb E\|u(T)-\mathcal G_N(u_0)\|_{L_2}
\le\varepsilon_{\rm prop}+\varepsilon_K
+\varepsilon_{\rm approx}+\varepsilon_{\rm gen}+\varepsilon_{\rm alias},
\end{equation}
where $\varepsilon_{\rm prop}$ is bounded by \cref{eq:eprop-main} and the table
defect by \cref{prop:tab}. $\varepsilon_K=\sigma_K(r)\lesssim K^{-\gamma}$ is the
residual truncation tail (the same Sobolev order as the FNO tail $K^{-\gamma}$).
$\varepsilon_{\rm gen}=O(n_{\rm tr}^{-\beta_{\rm gen}})$ for a learning exponent
$\beta_{\rm gen}>0$. $\varepsilon_{\rm alias}$ is zero for the injected
propagator and bounded by $\sigma_K(r)$ for the residual network. At matched
$(w,L,K,n_{\rm tr},N)$, the band-by-band comparison is
\begin{equation}\label{eq:compare}
\begin{aligned}
\varepsilon_\infty^{\rm FrFNO}
&\le \varepsilon_{\rm prop}
+\underbrace{\sigma_K(r)}_{\substack{\text{out-of-band tail,}\\ \text{filled by the table}}}
+\varepsilon_{\rm low}^{\rm Fr},\\
\varepsilon_\infty^{\rm FNO}
&=\underbrace{\sigma_K(u)}_{\substack{\text{out-of-band tail,}\\ \text{set to zero}}}
+\varepsilon_{\rm low}^{\rm FNO},\\[-1pt]
&\hspace{2.2em}\frac{\sigma_K(r)}{\sigma_K(u)}\le C(\eta,\beta,T)<1 .
\end{aligned}
\end{equation}
The rigorous, resolution-independent advantage is on the out-of-band term. Under
super-resolution the FNO spectral convolution vanishes on every mode beyond the
training band, while the tabulated propagator supplies it with an $N$-independent
defect $\varepsilon_{\rm tab}$, replacing $\sigma_K(u)$ by the smaller
$\sigma_K(r)$. The retained-band terms $\varepsilon_{\rm low}$ are ordinary
approximation and generalization errors of the matched backbones. The FrFNO target
has smaller amplitude and need not learn the Mittag--Leffler multiplier, but this
comparison is reported numerically rather than asserted as a theorem. When
$\varepsilon_{\rm tab}$ and $\varepsilon_{\rm prop}$ are small and the
out-of-band energy is appreciable (the super-resolution regime of the numerics),
the FrFNO floor lies below the FNO floor.
Quantitative learning bounds follow from the standard neural operator
approximation theory~\cite{kovachki2021ua,deryck2022generic}.
\end{theorem}

\subsection{Scope and testable predictions}\label{sec:th-scope}

The analysis delimits the method in four respects.
(i)~short horizon, since the Duhamel--Dyson expansion requires
$\beta T^\alpha M_1\ll1$.
(ii)~Sobolev admissibility $s>1/2$ for the advective bilinear form to be bounded.
(iii)~smooth solutions $u_0\in H^\gamma$ with $\gamma\ge2s$ (across a discontinuity
the constant-factor reduction no longer holds).
(iv)~bandwidth bottleneck $\min(K,N_0)$.

The theory makes three quantitative predictions tested in
\cref{sec:theoryverify}.
(a)~the out-of-band gap widens with the super-resolution factor $N/N_0$
(\cref{prop:floor}).
(b)~$\varepsilon_{\rm prop}$ is first order in $\eta$, $\beta T^\alpha M_1$, and
$T^\alpha$ (\cref{prop:prop}).
(c)~the residual is dominated by the first Born term, consistent with the
single-correction architecture.
The $s$-dependence of the advantage and the $N$-independent plateau are
reported numerically in \cref{sec:theoryverify}.

\section{Numerical experiments}\label{sec:numerics}

\subsection{Experimental setup}\label{sec:exp-setup}

Unless stated otherwise we solve \cref{eq:pde1} on $\Omega=(0,1)^2$ with
homogeneous Dirichlet boundary conditions. The main nonlinear test (\cref{sec:b1}) uses
$T=0.015$, $\beta=3$, a smooth weakly nonlinear regime where the
residual-regularity bound applies. The long-horizon test (\cref{sec:longT}) quadruples
$T$ to $0.06$. At $(\alpha,s)=(0.75,0.75)$ the lowest mode $(1,1)$ retains $0.663$ of its initial
amplitude at $T=0.015$ while the $(8,8)$ mode loses $97\%$ of its initial
amplitude, so the damping is substantial despite the short dimensional time
(the relevant fractional time scale is $T^\alpha$, not $T$). The horizon is
deliberately short to stay in the weakly nonlinear regime where the
residual-regularity bound applies. Larger $T$ strengthens nonlinear
steepening, and in the supercritical regime $s<1/2$ can produce finite-time
gradient blowup~\cite{kiselev2008blowup,alibaud2007occurrence}, consistent with
the Sobolev-admissibility condition \cref{sec:th-scope}(ii) that $s>1/2$ is
required for the advective bilinear form to be bounded. In all cases
it degrades the $H^\gamma$ regularity ($\gamma\ge 2s$) on which the bound relies.

The reference solver is the IMEX/upwind scheme \cref{eq:imex} (upwinding is
mandatory for stability at small $s$), with resolution--step pairs
$(N,N_t)=(16,400),\allowbreak(64,800),\allowbreak(128,1600)$ (\cref{app:ref}). Training uses $128$
multiscale random fields paired with a $9\times9$ order grid ($10\,368$ pairs at
$17^2$), $8000$ iterations, batch $64$, Adam$(1.5\times10^{-3},w_d=10^{-5})$ cosine
decayed to $5\times10^{-5}$ (\cref{app:train}). The held-out set has $48$ independent
samples with continuously drawn orders off the training grid (\cref{app:data}).
Reported errors are means of \cref{eq:metric} at $17^2,65^2,129^2$, of which only
$17^2$ was seen during training.

All spectral models share the same FiLM/spectral-convolution backbone with
$K=10$ retained Fourier modes per axis and effectively identical capacity ($1.88\times10^6$ parameters for
FrFNO and FNO; $0.38$--$2.1\times10^6$ for the other baselines). FrFNO injects the
analytic base field and learns only the residual. FNO regresses the full solution.
PINO adds a high-resolution PDE residual. PDNO learns a pseudo-differential symbol.
CNO uses alias-controlled convolutions with AdamW/$L_1$. DeepONet uses a
$17$-sensor branch--trunk. U-Net is a fully convolutional baseline.

\subsection{Main nonlinear test }\label{sec:b1}

\Cref{tab:b1} reports the comparison for 2D fractional Burgers equation. FrFNO attains the lowest error at
\emph{all three} resolutions. It reaches $1.72\%$ at the training resolution and
$8.30\%/10.23\%$ under zero-shot super-resolution. The matched FNO degrades much more
sharply ($2.66\%\to23.46\%$), which is exactly the contrast predicted by
\cref{eq:compare}. Both models share the same backbone and retain $K$ modes, but
on the modes beyond the training band the FNO predicts zero and pays the full
out-of-band energy $\sigma_K(u)$, whereas FrFNO fills those modes from the
resolution-independent propagator table and pays only the smaller residual tail
$\sigma_K(r)\le C\,\sigma_K(u)$ with $C<1$ in this weakly nonlinear regime
(\cref{eq:k2sratio-main}). On the retained band FrFNO also avoids relearning the
Mittag--Leffler multiplier. U-Net, lacking any
fixed spectral symbol, is the least resolution robust and reaches $84.4\%$ at $129^2$.
DeepONet, limited by its fixed $17^2$ sensor grid and an MLP trunk without
spectral or convolutional inductive bias, saturates near $44\%$ even at the
training resolution. CNO is
competitive at $17^2$ but its super-resolution degradation is larger than FrFNO's.
PINO's Caputo residual term is a coarse, mis-specified regularizer. Removing it
($\lambda_{\rm res}=0$) returns error to the FNO level, while the analytic
propagator embeds the linear memory exactly with no temporal discretization error.

\begin{table}[htbp]
\centering\small
\caption{Nonlinear fractional Burgers' equation ($T=0.015,\beta=3$): relative $L^2$
test error (\%) on the 48-sample held-out set at the training resolution
$17^2$ and under zero-shot super-resolution to $65^2,129^2$, with training
wall-clock. CNO uses its official training recipe. Best in bold.}
\label{tab:b1}
\begin{tabular}{lcccc}
\toprule
Model & $17^2$ & $65^2$ & $129^2$ & Train (s) \\
\midrule
\textbf{FrFNO (ours)} & $\mathbf{1.720}$ & $\mathbf{8.304}$ & $\mathbf{10.227}$ & 243.2 \\
FNO & 2.662 & 20.221 & 23.463 & 180.4 \\
PINO & 2.710 & 12.883 & 15.057 & 338.7 \\
PDNO & 2.246 & 10.019 & 20.649 & 172.1 \\
CNO & 2.686 & 14.851 & 17.076 & 206.0 \\
DeepONet & 44.102 & 44.242 & 44.515 & 54.8 \\
U-Net & 4.565 & 69.813 & 84.403 & 61.2 \\
\bottomrule
\end{tabular}
\end{table}

\Cref{fig:snapshot} shows held-out field snapshots, an $x$-cross-section, and
two diagnostics beyond relative $L^2$ error. The first is the radial phase error
\begin{equation}\label{eq:phase-err}
\theta(k)=\frac{1}{\#\{\boldsymbol j:|\boldsymbol j|=k\}}
\sum_{|\boldsymbol j|=k}
\left|\arg\!\big(\hat u_{{\rm pred},\boldsymbol j}\,\overline{\hat u_{{\rm ref},\boldsymbol j}}\big)\right|,
\end{equation}
where $\boldsymbol k\in\mathbb Z^2$ is the FFT wavenumber vector and $k=|\boldsymbol k|$
its radial magnitude (on the $129^2$ grid the Nyquist cutoff is $k=64$).
and the normalized cross-correlation
($\mathrm{NCC}\in[-1,1]$, with $1$ denoting identical fields). Phase error is
reported separately because two fields can match in amplitude yet be spatially
misaligned, which is invisible to an $L^2$ norm but degrades convective
accuracy. FrFNO has the smallest phase error ($0.807$ rad, averaged over $k=1\sim 32$) and highest NCC
($0.9954$), versus $1.062$ rad/$0.9837$ for FNO. The analytic propagator is a
real-valued diagonal multiplier in Fourier space, so it scales each mode's
amplitude without shifting its phase. This locks the spectral phase and avoids
the phase misplacement and Gibbs ringing that data-driven operators exhibit at
high modes.

\begin{figure}
\centering
\includegraphics[width=0.80\textwidth]{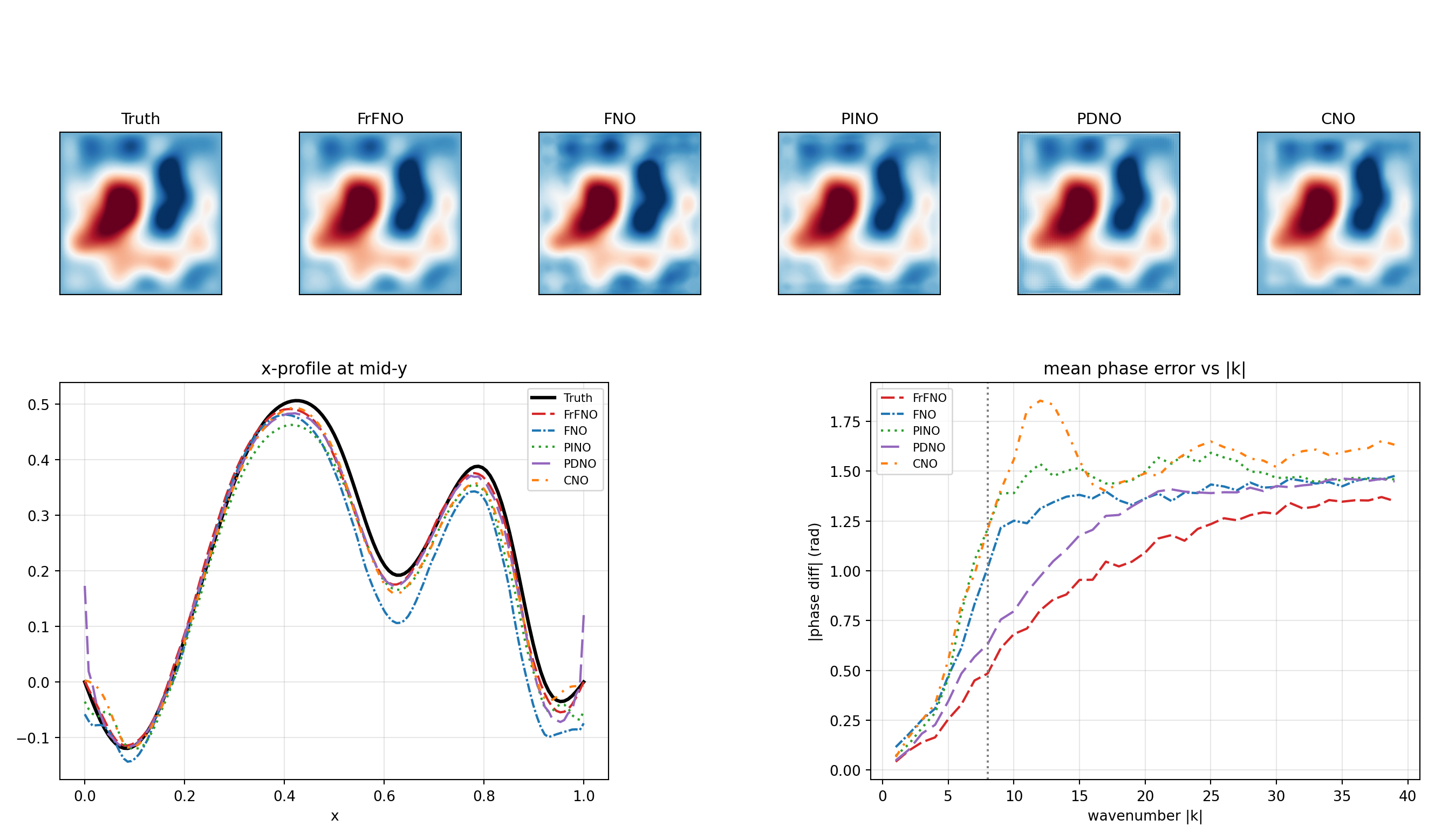}
\caption{Held-out $129^2$ prediction for the nonlinear fractional Burgers' equation: field snapshots (top),
an $x$-cross-section (bottom left), and the radial phase error (bottom right). FrFNO
is visually indistinguishable from the reference and attains the smallest phase error
and highest NCC.}
\label{fig:snapshot}
\end{figure}

\subsection{Long integration window: one-shot, not autoregressive}
\label{sec:longT}

\Cref{tab:b2b} confirms for $T=0.06$ (four times the main horizon). FrFNO
remains the best at every resolution ($2.65/9.16/10.93\%$), while the FNO error at $129^2$
grows to $31.1\%$ and U-Net to $166.6\%$. CNO, trained with its official
AdamW+L1 recipe, remains stable and is in fact the most resolution-robust
baseline after FrFNO ($3.86/13.97/16.07\%$). In a separate experiment we take
each network trained for the single-step horizon $T=0.015$ and roll it out
autoregressively (four steps to reach $T=0.06$, feeding each prediction back as
the next initial condition). Every operator, including FrFNO, degrades to
$87$--$112\%$ error, because rolling out compounds the per-step residual bias
and, for the Caputo memory equation, reintroduces the history accumulation that
the one-shot propagator avoids. The propagator is exact for any $T$,
so time super-resolution is free in the one-shot setting.

\begin{table}[htbp]
\centering\small
\caption{One-shot long-horizon prediction ($T=0.06$): relative $L^2$ test
error (\%) on the held-out set. FrFNO remains the best at all resolutions.}
\label{tab:b2b}
\begin{tabular}{lccc}
\toprule
Model & $17^2$ & $65^2$ & $129^2$ \\
\midrule
\textbf{FrFNO (ours)} & $\mathbf{2.646}$ & $\mathbf{9.164}$ & $\mathbf{10.933}$ \\
FNO & 3.577 & 26.468 & 31.124 \\
PINO & 3.864 & 18.026 & 23.006 \\
PDNO & 3.150 & 12.930 & 27.633 \\
CNO & 3.860 & 13.965 & 16.066 \\
DeepONet & 30.366 & 30.939 & 31.410 \\
U-Net & 5.850 & 129.487 & 166.589 \\
\bottomrule
\end{tabular}
\end{table}

\subsection{Space--time variant and the PINO residual-weight sweep}
\label{sec:b3}

The main Burgers test treats the fractional time derivative as a conditioning
input and applies one-shot prediction. To further test whether the
analytic-propagator advantage survives a tighter physics-loss coupling, we
consider a space--time variant in which PINO additionally penalises a
collocation residual on the Caputo time derivative. \Cref{tab:b3} compares
FrFNO, the matched FNO, and PINO with two residual weights $\lambda_{\rm res}$
on the same data. At the training resolution $17^2$ all three are comparable
($3.4$--$4.1\%$), but under zero-shot super-resolution FrFNO remains at
$8.1$--$9.6\%$ while FNO degrades to $21$--$24\%$ and PINO to $25$--$38\%$.

The PINO sweep itself is informative. Setting $\lambda_{\rm res}=0$ recovers
the data-driven PINO, which at $17^2$ is the best of the three ($3.36\%$).
Turning the space--time Caputo residual on ($\lambda_{\rm res}=0.5$) makes
PINO substantially worse at every resolution ($22.7\%$ at $17^2$, $38.1\%$ at
$129^2$). The collocation residual is evaluated on the coarse $17^2$ mesh,
where the fractional Laplacian and the Caputo history are themselves
under-resolved, so the physics term acts as a mis-specified regulariser that
forces the network toward the wrong smoothness. The analytic propagator avoids
this trap entirely. It is exact on the linear part and leaves no residual for a
coarse collocation to mis-measure. This isolates the FNO/PINO difficulty as
the algebraic, order-dependent linear response itself, consistent with the
linear-diffusion row of \cref{tab:breadth}.

\begin{table}[htbp]
\centering\small
\caption{Space--time variant: relative $L^2$ test error (\%) under the
matched backbone, with the PINO Caputo-residual weight $\lambda_{\rm res}$
swept. At $17^2$ the three are comparable; under zero-shot super-resolution
FrFNO stays flat while FNO and PINO degrade sharply, and the coarse
space--time collocation residual hurts PINO rather than helping it.}
\label{tab:b3}
\begin{tabular}{lccc}
\toprule
Model & $17^2$ & $65^2$ & $129^2$ \\
\midrule
\textbf{FrFNO (ours)} & $\mathbf{4.03}$ & $\mathbf{8.07}$ & $\mathbf{9.47}$ \\
FNO & $3.56$ & $21.05$ & $23.52$ \\
PINO, $\lambda_{\rm res}=0$ & $3.36$ & $25.19$ & $28.62$ \\
PINO, $\lambda_{\rm res}=0.5$ & $22.73$ & $30.31$ & $38.07$ \\
\bottomrule
\end{tabular}
\end{table}

\subsection{Integer-order limit }\label{sec:integer}

\Cref{cor:int} predicts the largest advantage at $\alpha=s=1$. Expanding the
training grid to $\alpha\in[0.6,1]$, $s\in[0.4,1]$, FrFNO achieves
$3.53/6.40/7.56\%$ at the exact integer point (smallest super-resolution degradation,
$2.1\times$), versus $48.6\%$ for FNO at $129^2$. The same ordering holds on a
near-integer band $U[0.9,1]$ (\cref{tab:b4}). FrFNO thus specializes correctly to the classical
integer-order limit.

\begin{table}[htbp]
\centering\small
\caption{Integer-order generalization: relative $L^2$ test error (\%) on the
held-out set. Top block: exact $(\alpha,s)=(1,1)$;
bottom block: near-integer $U[0.9,1]$.}
\label{tab:b4}
\begin{tabular}{lccc|ccc}
\toprule
& \multicolumn{3}{c}{$(1,1)$} & \multicolumn{3}{c}{near-integer $U[0.9,1]$}\\
\cmidrule(lr){2-4}\cmidrule(lr){5-7}
Model & $17^2$ & $65^2$ & $129^2$ & $17^2$ & $65^2$ & $129^2$\\
\midrule
\textbf{FrFNO} & $\mathbf{3.502}$ & $\mathbf{6.318}$ & $\mathbf{7.352}$
 & $\mathbf{2.875}$ & $\mathbf{5.779}$ & $\mathbf{6.784}$\\
FNO & 6.596 & 41.518 & 48.595 & 5.620 & 40.673 & 47.458\\
PINO & 5.648 & 20.153 & 27.107 & 4.972 & 19.312 & 25.891\\
PDNO & 5.132 & 10.593 & 11.787 & 4.331 & 9.633 & 10.927\\
CNO & 5.860 & 11.689 & 12.919 & 4.503 & 10.651 & 11.910\\
DeepONet & 16.026 & 17.660 & 18.335 & 17.906 & 18.601 & 19.516\\
U-Net & 12.761 & 177.441 & 219.517 & 10.660 & 175.276 & 216.441\\
\bottomrule
\end{tabular}
\end{table}

\subsection{Computational cost and break-even analysis}\label{sec:cost}

We benchmark on a single RTX~3090 (float32) using the nonlinear problem of \cref{sec:b1} at four
resolutions, with $N_t$ scaled as $400,800,1600,3200$.

\paragraph{Per-query cost.}
\Cref{tab:cost-perquery} reports the median wall-clock on a single RTX~3090.
The columns are as follows. (i)~$B=1$, the cost of solving one field at a time.
(ii)~$B=64$ amortized, the cost of processing a batch of $64$ fields divided by
$64$. When multiple fields are stacked into one tensor, the GPU executes the
same kernel on all of them simultaneously, so the per-field cost drops below the
$B=1$ number. The ``Linear solver'' column solves the
homogeneous linear fractional diffusion (no advection). The ``Nonlinear solver''
includes the variable-coefficient perturbation and the $u\partial_x u$
nonlinearity, and is the reference used to generate training labels.

The nonlinear direct solver grows from $1.73$\,s at
$17^2$ to $13.9$\,s at $129^2$, driven by the $O(N_t^2)$ history convolution.
Even in bulk ($B=64$ amortized), the nonlinear solver costs $115.6$\,ms at
$65^2$ and $1844.7$\,ms at $129^2$. FrFNO
inference (five analytic base fields via the cached table plus one residual-network
forward pass) stays at $\approx13$\,ms across all resolutions, independent of $N_t$.
The per-query speedup grows from $129\times$ at $17^2$ to $999\times$ at $129^2$.

\begin{table}[htbp]
\centering\footnotesize
\caption{Per-query wall-clock (median, GPU). ``$B=64$ amortized'' divides the batch time by
$64$ (by $8$ at $129^2$ for memory).}
\label{tab:cost-perquery}
\begin{tabular}{cc|cc|cc|c}
\toprule
& & \multicolumn{2}{c|}{Linear solver} & \multicolumn{2}{c|}{Nonlinear solver} & FrFNO\\
\cmidrule(lr){3-4}\cmidrule(lr){5-6}
Res. & $N_t$ & $B=1$ & $B=64$ amortized & $B=1$ & $B=64$ amortized & $B=1$\\
\midrule
$17^2$  & 400  & 0.56\,s & 8.7\,ms  & 1.73\,s  & 28.8\,ms  & 13.4\,ms\\
$65^2$  & 1600 & 2.21\,s & 42.5\,ms & 6.98\,s  & 115.6\,ms & 12.0\,ms\\
$129^2$ & 3200 & 4.33\,s & 679.2\,ms& 13.92\,s & 1844.7\,ms& 13.9\,ms\\
\bottomrule
\end{tabular}
\end{table}

\Cref{fig:speedup} visualises the per-query and total cost comparison.

\begin{figure}
\centering
\includegraphics[width=0.80\textwidth]{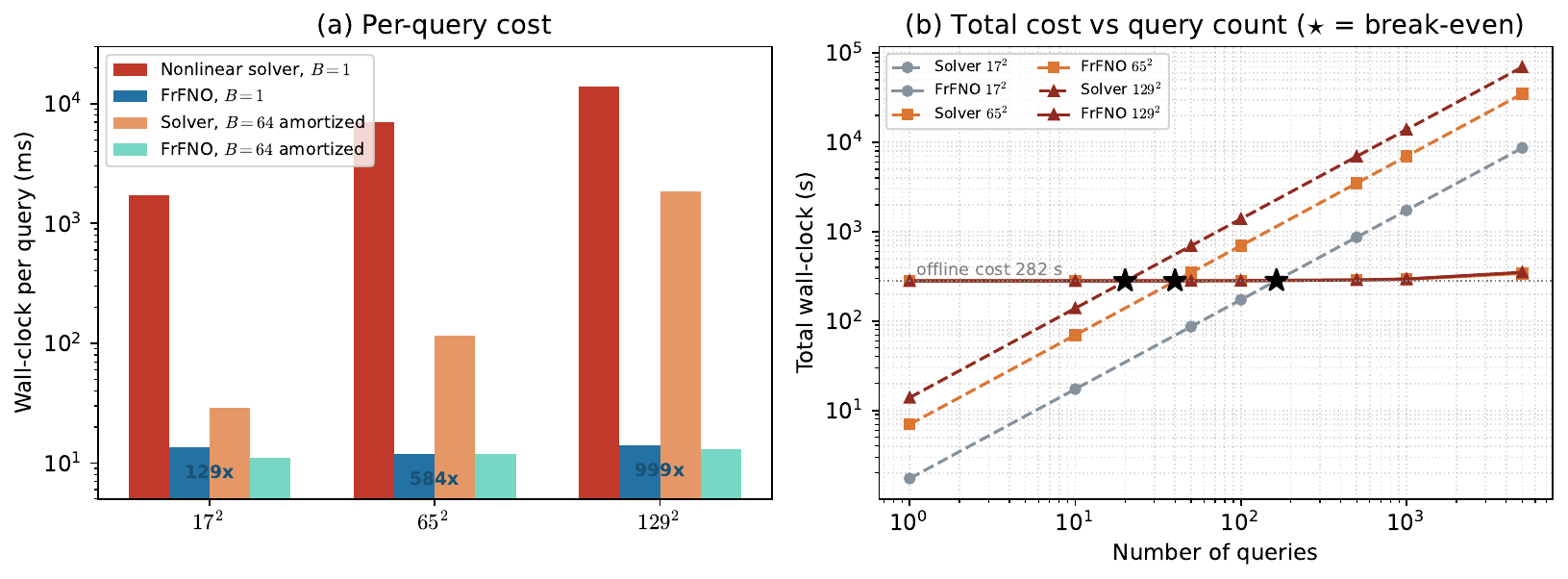}
\caption{(a) Per-query wall-clock on a log scale. FrFNO (blue/cyan) is nearly flat at
$\approx 13$\,ms across resolutions, while the nonlinear direct solver (red/orange) grows
with $N_t^2$. Numbers above the blue bars are the per-query speedup. (b) Total
wall-clock versus number of queries (log--log). Stars mark the break-even point for each
resolution. The dotted horizontal line is the one-time offline cost ($282$\,s).}
\label{fig:speedup}
\end{figure}

\paragraph{Offline cost and break-even.} The one-time offline cost is $281.6$\,s
($22.0$\,s data generation plus $259.6$\,s training). The propagator table ($\approx5$\,s)
is built once and serves every resolution. The break-even query count is the
number of solves at which the offline cost is recovered by the per-query savings,
$N_{\rm be}=281.6/(t_{\rm solver}-t_{\rm FrFNO})$. It reaches after $20$ queries
at $129^2$, $40$ at $65^2$, and $164$ at $17^2$. The higher the resolution,
the faster the break-even, because the direct solver cost grows with $N_t^2$
while FrFNO cost is flat.

\paragraph{Total cost.} At $129^2$, $1000$ queries cost $296$\,s for FrFNO versus
$13920$\,s for the solver ($47\times$), and $5000$ queries $351$\,s versus $69601$\,s
($198\times$, $19.2$\,hours saved), the typical regime of uncertainty
quantification, parameter sweeps, and inverse optimization.

\subsection{Direct verification of the theory}\label{sec:theoryverify}

\Cref{fig:band-fill} illustrates the central mechanism. Trained at $17^2$ with
$K=10$ retained modes, under zero-shot super-resolution to $129^2$, the radially
averaged phase error \eqref{eq:phase-err} of FrFNO is consistently smaller than
FNO across the entire wavenumber range, with the gap widening in the out-of-band
region $k>K$. The reason is band filling. On modes that the spectral
convolution cannot represent, FNO relies on the $1\times1$ skip path which
produces phase-distorted high-wavenumber content, while FrFNO supplies the exact
analytic propagator from the cached table on every mode.

\begin{figure}
\centering
\includegraphics[width=0.9\textwidth]{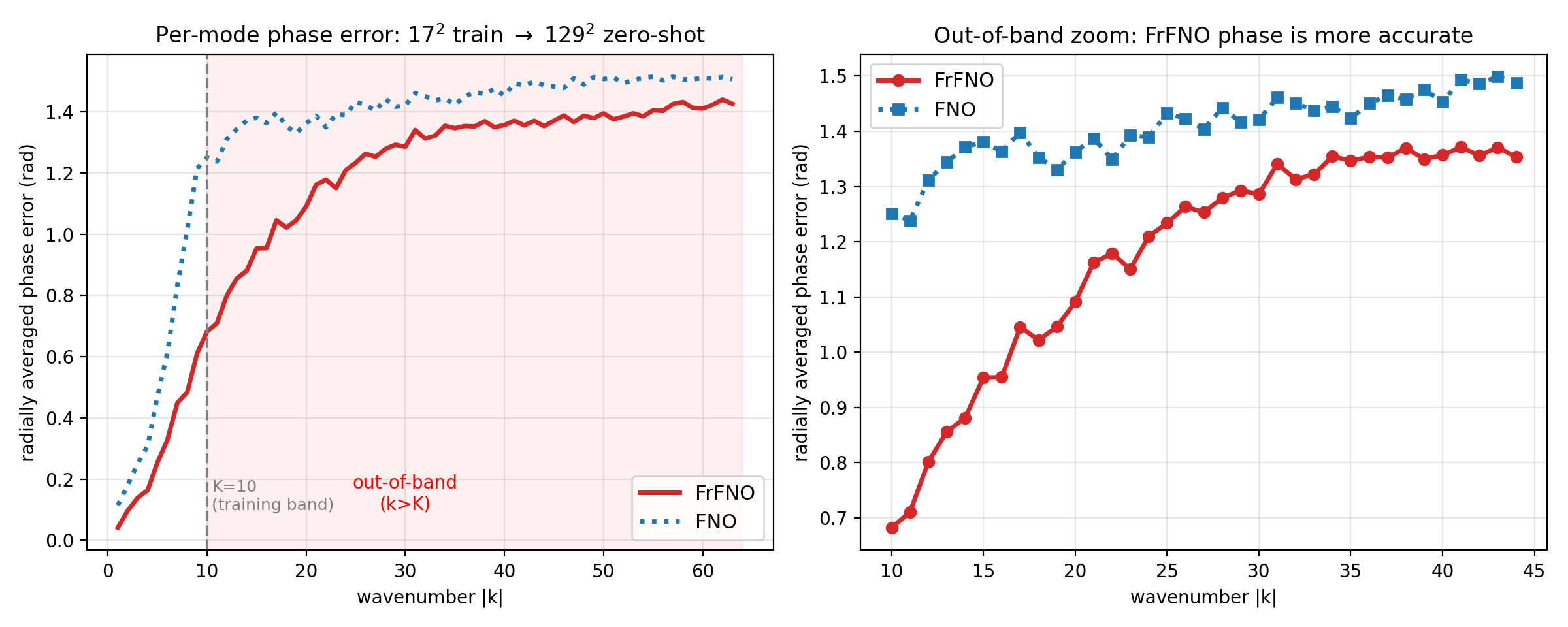}
\caption{Per-mode phase error (B1, trained at $17^2$, evaluated at $129^2$).
Left: radially averaged absolute phase error versus wavenumber. The vertical
dotted line marks $K=10$, the training-band cutoff. Right: zoom on the out-of-band
region $K<k<45$. FrFNO's phase error is consistently smaller, because it fills
out-of-band modes with the exact analytic propagator rather than relying on the
$1\times1$ skip path.}
\label{fig:band-fill}
\end{figure}

We run five tests on a common set of $24$ fields of the predictions in
\cref{sec:th-scope}. Exp.~1 and Exp.~5 share one frozen high-resolution truth
($513^2$, solved at $N_t=6400$, locally averaged to each $17/65/129/257^2$ evaluation grid). The full prediction/observation table, per-test
discussion, and complete Exp.~1 and Exp.~5 tables are in Section S3 of the supplementary material (Table S2). In brief,
all five agree with the theory. Past $129^2$ both networks plateau, yet the FrFNO
floor ($10.8\%$) is below half the FNO floor ($22.9\%$, Exp.~1). The advantage
ratio grows $1.38\to7.69$ as $s:0.4\to0.8$ (Exp.~2). The propagator error is first
order in $\eta,T^\alpha,\beta$ with measured slopes $1.016,0.828,0.931$ (Exp.~3).
A single Born term accounts for $92.5$--$96.2\%$ of the residual (Exp.~4). The
floor is set by the bottleneck $\min(K,N_0)$. An $N\to\infty$
convergence curve confirms a plateau rather than convergence to zero. A residual
spectral-tail measurement confirms the same Sobolev order as the full field
(the ratio saturates near one for $K\ge12$, Exp.~5b). A dedicated
$K$-as-bottleneck scan at $N_0=65$ shows the floor decreasing with $K$ before
saturating (Exp.~5), with the analytic-propagator advantage preserved on every axis.
The mechanism is band filling. On modes $k>K$ that the spectral convolution cannot
represent, the FNO predicts zero while FrFNO supplies the exact propagator from the
cached table, so the out-of-band error is $\sigma_K(u)$ for FNO but the smaller
$\sigma_K(r)$ for FrFNO.

\subsection{Ablation and robustness checks}
\label{sec:rigor}

To ensure the main results are not artifacts of a favorable seed or budget, we run
a unified suite at matched backbone ($1.88\times10^6$ parameters), matched data, matched
shuffling, a common $3000$-step budget, and $24$ test fields; absolute values are
slightly above the $8000$-step main table, but relative conclusions are unchanged.

\paragraph{Single-knockout ablation.} Removing the analytic base is by far the dominant
change ($9.49\%\to22.53\%$ at $T=0.015$ and $22.69\%\to34.99\%$ at $T=0.06$). The
secondary components (DynFrac2d features, neighbor-order bases) are near-neutral on the
short horizon but matter at the four-times-longer
horizon, where each costs $2.6$--$3.7$ points at $129^2$. The neighbor fields leave
training-resolution error unchanged but degrade zero-shot super-resolution, so they
carry cross-resolution order context and are retained. Full ablation tables are in the
supplementary material (Section S3).

\paragraph{Robustness.} Across three seeds FrFNO reaches $9.49/9.70/9.84\%$ at
$129^2$ (worst $9.84\%$), while the best FNO seed is at least $21.7\%$. The worst
FrFNO seed still beats the best FNO seed by a wide margin. Error decreases with
training fields per order ($11.29\%$ at $16$, $9.49\%$ at $128$), and a compact
$0.70\times10^6$-parameter network ($K=6$) already attains $9.36\%$, on par with the
$1.88\times10^6$-parameter model, because the exact propagator carries the
modewise evolution and the residual network only represents a small-amplitude
correction inside the retained band, while the out-of-band modes are supplied by
the table rather than learned (\cref{prop:floor}). The
order grid $3\times3/5\times5/9\times9$ gives $10.26/9.89/9.49\%$ (saturated by
$9\times9$). Learning rates $10^{-3}/3\times10^{-3}$ move the result by less than one
point. The method is thus insensitive to secondary hyperparameters once the
propagator is present. Full sweep tables are in the supplementary material (Section S3).

\subsection{Equation-class generality}\label{sec:breadth}

To verify that the advantage is not specific to the Burgers advective structure, we
apply the identical construction, with no change to the injected propagator, to eight
further equation settings, all sharing the same fractional dissipative principal part
$-\nu(x)(-\Delta)^su$ and differing only in the zero-order or system structure. FrFNO
outperforms the matched FNO in every setting. The advantage ratio ranges from
$1.8\times$ (periodic Riesz Burgers) to $14.5\times$ (linear fractional diffusion),
with zero-shot drift below $+1.2$ points in all cases. The construction thus transfers
unchanged to periodic diffusion, Fisher--KPP and Allen--Cahn reaction--diffusion,
two-component and non-symmetric coupled systems, and 2D/3D incompressible fractional
Navier--Stokes. Full equation definitions, boundary conditions, and parameter settings for all eight settings are collected in Section~S3.5 of the supplementary material. \Cref{tab:breadth} summarizes the results.

\begin{table}[H]
\centering\footnotesize
\caption{Equation-class generality. Relative $L_2$ error (\%) of FrFNO versus matched
FNO at the highest zero-shot resolution, with the advantage ratio and the zero-shot
drift (train $\to$ highest resolution). All use identical backbone and budget.}
\label{tab:breadth}
\setlength\tabcolsep{3pt}
\begin{tabular}{p{0.30\textwidth}lcccc}
\toprule
Equation class & Domain & FrFNO & FNO & Ratio & ZS drift \\
\midrule
Periodic Riesz Burgers & $\mathbb T^2$, $64^2$ & $3.55$ & $6.40$ & $1.80\times$ & $+0.54$ \\
Linear fractional diffusion & $(0,1)^2$, $64^2$ & $1.82$ & $26.49$ & $14.5\times$ & $+0.11$ \\
Fisher--KPP & $(0,1)^2$, $64^2$ & $4.08$ & $26.37$ & $6.46\times$ & $+0.93$ \\
Allen--Cahn & $(0,1)^2$, $64^2$ & $3.80$ & $26.61$ & $7.00\times$ & $+0.58$ \\
Two-component vector Burgers & $(0,1)^2$, $64^2$ & $8.54$ & $27.92$ & $3.27\times$ & $+0.05$ \\
Non-symmetric coupled system & $(0,1)^2$, $64^2$ & $4.51$ & $36.49$ & $8.09\times$ & $-0.02$ \\
Incompressible NS (2D vorticity) & $\mathbb T^2$, $64^2$ & $2.91/1.41$ & $6.74/3.46$ & $2.32/2.45\times$ & $+1.20/+0.63$ \\
3D vector vorticity NS & $\mathbb T^3$, $24^3$ & $7.83$ & $18.01$ & $2.30\times$ & $+1.08$ \\
\bottomrule
\end{tabular}
\end{table}

The most striking result is the linear fractional diffusion case: with no
nonlinearity, the matched FNO still has an error of $26.5\%$ while FrFNO
reaches $1.8\%$, isolating the FNO's difficulty as the algebraic, long-memory,
order-dependent \emph{linear} response itself. The linear/Fisher/Allen--Cahn FNO
rows are nearly identical ($26.4$--$26.6\%$). Discarding the off-diagonal coupling
in the non-symmetric system raises the FrFNO error from $4.51\%$ (full
matrix-valued propagator) to $5.28\%$ (diagonal propagator, a $17\%$ relative
degradation), confirming that the full eigen-rotated matrix propagator is
required. The 2D/3D Navier--Stokes rows show the propagator, dictated by the
componentwise identical linear operator, is blind to spatial dimension while
dimension-specific physics (Biot--Savart, vortex stretching) stays in the
residual. The construction is thus equation-class agnostic in the sense of
\cref{sec:sysmethod}. It applies whenever the frozen linear principal part
is a uniformly dissipative, diagonalizable finite matrix, scalar or coupled and in
any dimension. The boundary is equally clear. A reaction strong enough to nucleate
fronts or an advection strong enough to form shocks lowers the $H^\gamma$ residual
regularity ($\gamma\ge 2s$) and falls outside the present smooth-regime surrogate.

\section{Concluding remarks}\label{sec:conclusion}

We proposed FrFNO, a resolution-robust conditional neural operator for
space--time fractional PDEs that injects the exact Mittag--Leffler propagator as a
parameter-free base field and learns only the residual. The theory shows this change
is structural. A fractional Duhamel--Dyson expansion bounds the propagator error, the tabulated multiplier is
resolution-independent and fills the out-of-band modes that an FNO sets to zero, and the residual shares the Sobolev order of the full field (a constant-factor
smaller tail rather than a steeper decay rate). The alias-free diagonal propagator
makes zero-shot super-resolution converge to a residual-controlled floor. The
band-by-band comparison explains every numerical observation. Under an identical
backbone, FrFNO is the best at all resolutions on nonlinear Burgers, a
four-times-longer horizon, the integer-order limit, and eight further equation
settings. Five tests confirm the plateau, $s$-scaling, propagator scaling, Born
residual, and $\min(K,N_0)$ floor.

The method has following limitations. It targets smooth-regime solutions, so
shock-dominated or front-nucleating problems where the residual loses
$H^\gamma$ regularity are out of scope. The weak-nonlinearity assumption
$\beta T^\alpha M_1\ll1$ means very long horizons should be predicted one-shot
at the target $T$ rather than by autoregressive rollout. For small $s\le1/2$
the advective coupling is weakly damped and requires additional upwinding or
stabilisation beyond the present recipe. Finally, the error floor under mesh
refinement is set by the retained bandwidth $K$ and the training resolution
$N_0$. Refining only the query grid cannot lower it, and a smaller floor
requires enlarging $K$ together with $N_0$.

Future directions include adaptive bandwidth selection driven by residual-tail
estimates, alias-controlled residual blocks on complex geometries via mapped
bases, and shock-dominated compressible flows that require a different linear
principal part than the present fractional diffusion operator.

\section*{Acknowledgments}
This work was financially supported by Jiangsu Provincial Scientific Research Center of Applied Mathematics (BK20233002).

\appendix
\section{Theoretical statements}\label{app:proofs}

This appendix collects the statements of the auxiliary lemmas and propositions
summarized in \cref{sec:theory}. Full proofs with all intermediate inequalities are
in Section S2 of the supplementary material. We work on $\mathbb T^d$ with
$\Lap=(-\Delta)^{1/2}$, $\lambda_{\boldsymbol k}=|\boldsymbol k|^2$, split
$\nu=c+\delta\nu$ with $\eta=\|\delta\nu\|_{L_\infty}/c$, and write
$A_0=-c\Lap^{2s}$, $A_1=-\delta\nu\,\Lap^{2s}$, $\mathcal N(u)=u\partial_xu$,
under (H1) $\nu\in W^{2,\infty}$, (H2) $\alpha\in(0,1],s\in(0,1]$, (H3)
$u_0\in H^\gamma$, $\gamma\ge2s$.

\subsection{Spectral facts}\label{app:ml}

The Mittag--Leffler function $\Ea(z)=\sum_{n\ge0}z^n/\Gamma(1+\alpha n)$ satisfies
for $0<\alpha\le1$, $x\ge0$,
\begin{equation}\label{eq:ML}
\text{(ML1)}\ 0<\Ea(-x)\le1,\quad
\text{(ML2)}\ \Ea(-x)\le\frac{C_\alpha}{1+x},\quad
\text{(ML3)}\ \sup_{x\ge0}x\Ea(-x)\le C_\alpha,
\end{equation}
where $C_\alpha>0$ depends only on $\alpha$~\cite{pollard1948,podlubny1999,kilbas2006}.
The frozen propagator $E_0(t)=\Ea(-ct^\alpha\Lap^{2s})$ is self-adjoint, contractive
on every $H^\gamma$, and commutes with $\Lap$. Its analytic-regularization estimate is
\begin{equation}\label{eq:anreg}
\|\Lap^{2s}E_0(t)\|_{H^\gamma\to H^\gamma}
=\sup_{\boldsymbol k}\lambda_{\boldsymbol k}^s
 \Ea(-ct^\alpha\lambda_{\boldsymbol k}^s)
\le\frac{C_\alpha}{ct^\alpha}.
\end{equation}

\subsection{Auxiliary lemmas}

\begin{lemma}[$L_1$ temporal consistency, modewise]\label{lem:l1}
For the scalar test equation ${}^{C}D_t^\alpha v=-\mu v$ with $\mu=c\lambda^s$,
\begin{equation}\label{eq:l1bound}
\sup_{\lambda\ge0}\big|g^{N_t}_{\alpha,s,c}(\lambda)-\Ea(-cT^\alpha\lambda^s)\big|
\le C_\alpha\,(\tau^{2-\alpha})\,cT^\alpha\lambda^s,
\end{equation}
for smooth data. For $t^\sigma$-type initial singularities the order degrades to
$O(\tau^{\alpha\sigma})$, $\sigma\in(0,1)$.
\end{lemma}

\begin{lemma}[Spatial eigenvalue consistency]\label{lem:eig}
For each fixed mode $\boldsymbol k$, the five-point realisation satisfies
\begin{equation}\label{eq:eigbound}
\big|g^{N_t}(q_{\boldsymbol k}^{(N)})-g^{N_t}(\lambda_{\boldsymbol k})\big|
\le C_\alpha\,cT^\alpha\,h^2\lambda_{\boldsymbol k}^2,
\end{equation}
whereas using the continuum symbol $\lambda_{\boldsymbol k}$ would make the left
side identically zero. The implementation instead uses the same five-point
symbol $q_{\boldsymbol k}^{(N)}$ for both the reference solver and the injected
base, so this $O(h^2)$ error is common to the label and the base and cancels in
the residual $r=u-u_{\rm base}$.
\end{lemma}

\begin{remark}[Diagonal spectral multipliers are alias-free]\label{lem:alias}
The injected base field is a diagonal DST multiplier, hence commutes with the
spectral projection and incurs zero aliasing at every $N$. The remaining
$O(h^2)$ eigenvalue-consistency term of \cref{lem:eig} is common to the base
and the reference solver and cancels in the residual.
\end{remark}

\subsection{Propositions and corollaries}

Let $E(t)$ denote the evolution operator of the full nonlinear equation (with variable diffusivity $\nu=c+\delta\nu$ and advection $\beta u\partial_xu$), so that $E(t)u_0$ is the exact solution at time $t$. The fractional Duhamel identity gives the scattering series $E=E_0+D_1+D_2+\cdots$,
\begin{equation}\label{eq:duhamel}
\begin{aligned}
&E(t)u_0=E_0(t)u_0+\int_0^t
K_0(t-\tau)\big(A_1E(\tau)u_0-\beta\mathcal N(E(\tau)u_0)\big)\,d\tau,\\
&K_0(r):=\frac{r^{\alpha-1}}{\Gamma(\alpha)}
E_{\alpha,\alpha}(-cr^\alpha\Lap^{2s}).
\end{aligned}
\end{equation}

\begin{proposition}[Propagator error]\label{prop:prop}
Under (H1)--(H3), for data $u_0\in H^\gamma$ with $\gamma\ge 2s$ (and $\gamma>d/2$ in the nonlinear case, automatic when $s>1/2$),
\begin{equation}\label{eq:properr}
\|u(T)-E_0(T)u_0\|_{L_2}
\le C_\alpha\,\eta\,(cT^\alpha)\,\|u_0\|_{H^\gamma}\,
\frac1{1-C_\alpha'\eta},
\end{equation}
convergent whenever $C_\alpha'\eta\,cT^\alpha<1$. This bound includes only the linear variable-diffusivity perturbation. Including the nonlinear advection term for
$s>1/2$,
\begin{equation}\label{eq:properr-nl}
\|u(T)-E_0(T)u_0\|_{L_2}\le
C_\alpha\eta cT^\alpha\|u_0\|_{H^\gamma}
+C_\alpha\beta T^\alpha M_1+O(T^{2\alpha}),
\end{equation}
where $M_1=C_S\|u_0\|_{H^\gamma}^2$ by the Sobolev embedding $H^\gamma\hookrightarrow L_\infty$ for $\gamma>d/2$.
\end{proposition}

\begin{remark}[Why there is no $\log$ divergence]\label{rmk:beta}
For smooth data, $\Lap^{2s}$ is commuted onto the data-carrying inner factor and
bounded by contractivity, giving the integrable kernel $r^{\alpha-1}$. For generic
$H^\gamma$ data with $\gamma<2s$, the inner factor uses \cref{eq:anreg}, producing
the integrable kernel $r^{\alpha-1}(T-r)^{-\alpha}$
($B(\alpha,1-\alpha)=\pi/\sin\pi\alpha$).
\end{remark}

Combining \cref{lem:l1,lem:eig} with \cref{prop:prop}, the injected base-field error is
\begin{equation}\label{eq:eprop}
\varepsilon_{\rm prop}\le
C_\alpha\eta cT^\alpha\|u_0\|_{H^\gamma}
+C_\alpha\beta T^\alpha M_1
+C_\tau\tau^{2-\alpha}+C_h h^2.
\end{equation}

\begin{proposition}[Residual regularity, same order]\label{prop:reg}
The elliptic estimate $\|E_0(t)f\|_{H^{\gamma+2s}}\le
C(1+(ct^\alpha)^{-1})\|f\|_{H^\gamma}$ supplies $2s$ derivatives when $E_0$ acts
directly on the data, but this gain is cancelled in the scattering term because
$A_1=-\mathcal M_{\delta\nu}\Lap^{2s}$ carries the factor $\Lap^{2s}$ itself.
$\Lap^{2s}E_0(t)$ is order zero on the high modes. Hence, for every
$\varepsilon>0$, $\|r(T)\|_{H^{\gamma-\varepsilon}}\le
C_{T,\varepsilon}(\eta+\beta T^\alpha M_1)\|u_0\|_{H^\gamma}$ with
$r=u-E_0(T)u_0$, and
\begin{equation}\label{eq:k2s}
\sigma_K(r)\le C_T(\eta+\beta T^\alpha M_1)K^{-\gamma}\|u_0\|_{H^\gamma},
\qquad
\sigma_K(u)\le K^{-\gamma}\|u\|_{H^\gamma};
\end{equation}
under non-degeneracy of the full-field spectrum, i.e.\ the tail satisfies
$\sigma_K(u)\asymp K^{-\gamma}\|u\|_{H^\gamma}$ uniformly for large $K$, which
holds when the initial data has non-negligible high-frequency content and the
propagator decays algebraically (fractional order $s<1$). For integer-order
$($exponential decay$)$ the tail ratio may be larger and the bound is qualitative only.
\begin{equation}\label{eq:k2sratio}
\frac{\sigma_K(r)}{\sigma_K(u)}\le C(\eta,\beta,T),
\end{equation}
a $K$-independent constant (below one for weak perturbations). A genuine
$2s$-order gain appears only for perturbations strictly lower order than the
principal part (or $\delta\nu=0,\beta=0$, where $r\equiv0$).
\end{proposition}

\begin{remark}[Resolution-independent propagator table]\label{prop:tab}
The modal propagator depends on the grid only through
$z_{\boldsymbol k}^{(N)}=cT^\alpha(q_{\boldsymbol k}^{(N)})^s$, via
$g_\alpha(z)=\Ea(-z)$. $p$-th order interpolation on a fixed one-dimensional
log-spaced table of $M$ nodes has uniform defect
$\varepsilon_{\rm tab}\le C_pM^{-p}$ independent of $N$, and the assembled base
field satisfies $\|u_{\rm base}^N-u_{\rm base}^{\star,N}\|_{L_2}\le
\varepsilon_{\rm tab}\|u_0\|_{L_2}$. One cached table serves every resolution.
By contrast, the FNO spectral weights are indexed by the discrete training-band
mode number, address a different physical eigenvalue on another grid, and vanish
beyond $K$ modes.
\end{remark}

\begin{proposition}[Band filling and the resolution-independent floor]\label{prop:floor}
With training resolution $N_0$ and modes $K$ fixed, Parseval splits the error
into a retained band ($|\boldsymbol k|<K$) and an out-of-band set
($K\le|\boldsymbol k|\le N$).
\begin{equation}\label{eq:floor-fr}
\varepsilon_{\rm Fr}^2(N)
=\|\mathcal P_K^N(\Rres-r)\|^2+\sigma_{K,N}^2(r)
+\varepsilon_{\rm tab}^2\|u_0\|^2,
\end{equation}
\begin{equation}\label{eq:floor-fno}
\varepsilon_{\rm FNO}^2(N)
=\|\mathcal P_K^N(\widetilde\Rres-u)\|^2+\sigma_{K,N}^2(u)+\varepsilon_{\rm pw}^2 .
\end{equation}
The remainder $\varepsilon_{\rm pw}^2$ is the out-of-band pointwise-path error,
defined as a Parseval identity in Section~S2 of the supplementary material. It has
no closed bound and is reported numerically. On the FrFNO line the table term is an
upper bound from \cref{prop:tab}, and the cross term between the table error and
the residual is $O(\varepsilon_{\rm tab})$ by Cauchy--Schwarz and is $N$-independent.
For $N>N_0$ the modes $N_0\le|\boldsymbol k|\le N$ were absent at training. The
FNO spectral convolution is zero there and pays the full energy
$\sigma_{K,N}(u)$, whereas the tabulated base field supplies them and leaves
only $\sigma_{K,N}(r)$ plus the $N$-independent $\varepsilon_{\rm tab}$. As
$N\to\infty$ with $K$ fixed both errors reach constant floors,
$\varepsilon_\infty^{\rm FrFNO}=\varepsilon_{\rm low}^{\rm Fr}+\sigma_K(r)+
\varepsilon_{\rm tab}\|u_0\|$ and
$\varepsilon_\infty^{\rm FNO}=\varepsilon_{\rm low}^{\rm FNO}+\sigma_K(u)+
\varepsilon_{\rm pw}$, with $\sigma_K(r)/\sigma_K(u)\le C<1$ in the weakly
perturbed regime by Proposition~\ref{prop:reg}. The out-of-band comparison is
the rigorous resolution-independent result. The retained-band comparison is
numerical.
\end{proposition}

\begin{corollary}[Plateau, not convergence to zero]\label{cor:plateau}
Because $K$ is fixed, $\varepsilon_\infty>0$. Refining $N$ drives the error to a
constant floor rather than to zero. The realisable spectral bandwidth is
$\min(K,N_0)$, so the floor is controlled by the bottleneck.
\end{corollary}

\begin{corollary}[Integer-order limit]\label{cor:int}
At $\alpha=s=1$ the propagator is the heat semigroup $e^{-ct(-\Delta)}$ with
exponential modal factor $e^{-cT\lambda}$, represented exactly by the
resolution-independent table, so the band-filling mechanism of
\cref{prop:floor} persists and the linear out-of-band content is captured with
no learned multiplier. At $\alpha=1$ the $L_1$ weights collapse to backward
Euler with first-order consistency $O(\tau)=O(\tau^{2-\alpha})$.
\end{corollary}

\subsection{Proof of the end-to-end theorem}

The statement of \cref{thm:total} and its band-by-band comparison
\cref{eq:compare} are in \cref{sec:theory}. The proof splits the error by two
triangle inequalities.
(a)~propagator/model error $\varepsilon_{\rm prop}$ (\cref{eq:eprop}),
(b)~$L_1$ plus five-point consistency and the table defect $\varepsilon_{\rm tab}$
(\cref{lem:eig,prop:tab}),
(c)~fixed-$K$ residual spectral tail $\sigma_K(r)$ (\cref{prop:reg}),
(d)~approximation, generalization ($O(n^{-\beta_{\rm gen}})$), and alias errors,
the last bounded by $\sigma_K(r)$ (\cref{lem:alias}). For the matched FNO the
identical decomposition holds with $r$ replaced by $u$; the two tails have the
same Sobolev order, and under super-resolution the out-of-band modes are filled
exactly by the table for FrFNO but set to zero for the FNO (\cref{prop:floor}).
Full details are in Section~S2 of the supplementary material.

\section{Additional numerical details}\label{app:numerics}

The five tests of the theoretical predictions in \cref{sec:theoryverify} are
summarized in the main text (Section~5.7). The complete prediction/observation
table (Table S2), per-test discussion, and bootstrap data for Exp.~1 and Exp.~5
are in Section S3 of the supplementary material. Raw data for Exp.~2--4 are in
the code repository.

\section{Implementation details}\label{app:impl}

This appendix summarizes key reproducibility details. Full data-generation formulas,
network layer specifications, baseline configurations, training hyperparameters, and
the reference-solver setup are in the supplementary material (Section S4).

\subsection{Data generation}\label{app:data}

Initial fields are multiscale random sine fields on the Dirichlet eigenbasis
($K_0=8$ modes per axis, $(i^2+j^2)^{-1/2}$ spectral decay, RMS normalized to $0.5$).
Diffusivities are smooth log-normal Karhunen--Lo\`eve fields ($20$ cosine modes,
$\beta_{\rm KL}=1$, contrast $\eta\approx 0.35$, unit spatial mean). Orders are sampled
on a $9\times9$ grid, $\alpha\in[0.55,0.95]$, $s\in[0.35,0.78]$. Training uses
$128$ fields paired with every order ($10\,368$ pairs). The test set has $48$ held-out
fields with continuously drawn orders. Full generation formulas are in Section S4.1
of the supplementary material.

\subsection{Network architecture}\label{app:arch}

The residual network has $13$ input channels. There are 7 explicit channels,
$u_0$, base field, $\nu$, and four neighbor-order bases. There are 6 internal channels,
coordinates, DynFrac2d fractional features, and constant $\alpha/s$ fields. The network
is lifted to width $48$, passed through $4$ spectral blocks
with $K=10$ Fourier modes and FiLM conditioning on $(\alpha,s)$, and projected to a
single residual channel. Total trainable parameters are $1.88\times10^6$. Full layer
details and the DynFrac2d feature definition are in Section S4.2.

\subsection{Baseline surrogates}\label{app:baselines}

Five baselines (PINO, PDNO, CNO, DeepONet, U-Net) are matched in capacity
($0.38$--$2.1\times10^6$ parameters) and trained on identical data with a common
$8000$-step protocol. Full configurations and a summary table are in Section S4.3.

\subsection{Training protocol}\label{app:train}

Adam optimizer (learning rate $1.5\times10^{-3}$, weight decay $10^{-5}$), cosine
annealing to $5\times10^{-5}$, $8000$ gradient steps, mini-batch $64$ (one shared
order pair per batch), training resolution $17^2$ with $N_t=400$ reference steps,
single precision. Full details are in Section S4.4.

\subsection{Reference solver}\label{app:ref}

All training and test labels are generated by a GPU-accelerated IMEX scheme. Let
$L_\nu=\nu(\boldsymbol{x})(-\Delta)^s$ be represented spectrally, and let $D_{\rm
up}$ be the first-order upwind discretization of $\partial_x$. At step $n$, with two
Picard inner iterations, solve
\begin{equation}\label{eq:imex}
\big(b_0 I+\tau^\alpha L_\nu\big)u^n
 = \sum_{\ell=0}^{n-1}(b_{n-\ell-1}-b_{n-\ell})u^\ell
   +b_{n-1}u^0
   -\tau^\alpha\beta\,D_{\rm up}(u^{n-1})\,u^{n-1}.
\end{equation}
The Caputo derivative is discretized by the $L_1$ scheme, fractional diffusion is
treated implicitly and diagonalized by the discrete sine transform, and nonlinear
advection is treated explicitly by first-order upwinding. Reference solutions are
computed at three resolution pairs $(N,N_t)=(16,400),(64,800),(128,1600)$.
Convergence verification via manufactured solutions is in Section S1 of the
supplementary material.

\subsection{Training and inference}

Training generates reference pairs on the $17^2$ grid via \cref{eq:imex}, computes
the analytic base field $\Ebase$ (exact, detached from the graph), and optimizes the
residual network on $\|\Ebase+r-u(T)\|_2^2$ for $8000$ Adam steps. At inference, the
base field is recomputed on the query grid from the same cached table, and the
residual network (fixed $K$, $1\times1$ path, FiLM conditioning) is evaluated on that
grid. The output is $\hat u=\Ebase+r$. Pseudocode for both stages is in Section S4.5
of the supplementary material.

\subsection{Propagator table}\label{app:table}

The analytic propagator is precomputed as a resolution-independent lookup table in
the smooth scalar coordinate $z=\lambda^s$. On a $41\times41$ order grid and a
log-spaced grid of $M=1600$ points over $z\in[1,2\times10^5]$, the $L_1$ recurrence
\cref{eq:g-rec} is run for $N_t=400$ steps at each $z_m$. Since $g$ depends only on
$\alpha$ and $z$, one stored curve serves all $s$. The table is built once and cached
globally, and rendering at resolution $N$ is only an $O(N^2)$ evaluation of
$z=q_{\boldsymbol k}^{s}$ on the interior modes. A query uses trilinear interpolation
(bilinear in $(\alpha,s)$, linear along $z$). Building the table takes about $5$~s and
one query (plus DST/iDST on $129^2$) about $0.3$~ms, negligible against the network
forward pass. The offline construction algorithm is in Section S4.5.

\clearpage
\addcontentsline{toc}{section}{Supplementary Material}
\section*{Supplementary Material}
\markboth{Supplementary Material}{Supplementary Material}

This supplementary material contains (i) code verification of the reference solver,
(ii) complete statements and proofs of all theoretical results,
(iii) additional numerical tables and figures,
and (iv) full implementation details.

\renewcommand{\thesection}{S\arabic{section}}
\setcounter{section}{0}
\renewcommand{\thetable}{S\arabic{table}}
\setcounter{table}{0}
\renewcommand{\thefigure}{S\arabic{figure}}
\setcounter{figure}{0}
\renewcommand{\theequation}{S\arabic{equation}}
\setcounter{equation}{0}
\renewcommand{\thealgorithm}{S\arabic{algorithm}}
\setcounter{algorithm}{0}
\section{Code verification of the reference solver}\label{app:mms}

Mesh refinement alone only establishes \emph{self-convergence}. It does not prove that
the limiting discrete solution equals the solution of the continuous problem. This
section documents the method-of-manufactured-solutions (MMS) verification that the
reference IMEX/Picard solver solves the equation it claims to solve. We use a smooth
Dirichlet eigenfunction $u_e=\sin(\pi x)\sin(\pi y)(\tfrac12+t^2)$ with constant
diffusivity $c=1$, and verify component by component. For the temporal order the
diffusion is injected with the discrete five-point eigenvalue (removing the spatial
floor), and for the spatial order the Caputo term is replaced by the discrete $L_1$
operator acting on the exact trajectory (removing the temporal floor).

\begin{table}[H]
\centering\small
\caption{MMS convergence orders. \textbf{Left:} temporal order of the $L_1$ scheme
($N_x=32$, $\beta=0$, $T=0.20$), theoretical order $2-\alpha$. \textbf{Right:}
spatial order ($N_t=200$, $\alpha=s=0.75$, $T=0.02$). Nonlinear advection
(first-order upwind) gives order $\to1$, diffusion only (five-point symbol) gives
order $2.000$. The second-order spatial rate arises because the five-point
Dirichlet eigenvalue $\lambda_k^{\rm diff}=4\sin^2(k\pi h/2)/h^2$ approximates the
exact Laplacian eigenvalue $(k\pi)^2$ with an $O(h^2)$ truncation error. Taking the
$s$-th power to form the fractional propagator preserves the $O(h^2)$ order in $h$,
only changing the error coefficient (and not its dependence on $h$).}
\label{tab:mms}
\begin{tabular}{cc|cc}
\toprule
$\alpha$ & observed temporal order & $\beta$ & observed spatial order \\
\midrule
$0.60$ & $1.397$ (theory $1.40$) & $3$ (upwind) & $0.992$ (theory $1$) \\
$0.75$ & $1.249$ (theory $1.25$) & $0$ (diffusion only) & $2.000$ (theory $2$) \\
$1.00$ & $0.999$ (theory $1.00$) & & \\
\bottomrule
\end{tabular}
\end{table}

At the production resolution $N_x=128$, $N_t=1600$, $\alpha=s=0.75$, $\beta=3$, the
fully continuous manufactured-solution error is $1.51\times10^{-3}$ at $T=0.015$ and
$1.76\times10^{-3}$ at $T=0.020$ ($0.15$--$0.18\%$), one to two orders of magnitude
below the surrogate errors measured in Section~5 of the main paper ($2\%$--$84\%$), so
the reference uncertainty cannot alter any reported comparison. An independently written
NumPy solver agrees with the GPU IMEX implementation to $1.65\times10^{-14}$ (machine
precision), and explicit/IMEX updates agree to their common $L_1$ temporal difference.
Two independently coded discretizations converging to the same limit, together with the
analytic-order MMS table, rule out a stable convergence to a wrong solution.

\section{Complete statements and proofs}\label{sm:proofs}

This section collects the complete statements of the auxiliary lemmas and
propositions summarized in Section~4 of the main paper, together with their full
proofs and the proof of the main theorem (Theorem~4.1). Standard facts (the
$O(\tau^{2-\alpha})$ $L_1$ rate \cite{sunwu2006,linxu2007}, the five-point
finite-difference symbol expansion, the fractional Duhamel formula
\cite{podlubny1999,kilbas2006}, and the fractional heat-semigroup elliptic
estimate \cite{dinezza2012}) are cited rather than re-derived. We record only
the stiffness-dependent or problem-specific extensions needed below. We work on
$\mathbb T^d$ with $\Lap=(-\Delta)^{1/2}$, $\lambda_{\boldsymbol k}=|\boldsymbol
k|^2$, split $\nu=c+\delta\nu$ with $\eta=\|\delta\nu\|_{L_\infty}/c$, and write
$A_0=-c\Lap^{2s}$, $A_1=-\delta\nu\,\Lap^{2s}$, $\mathcal N(u)=u\partial_xu$,
under (H1) $\nu\in W^{2,\infty}$ bounded above and below, (H2)
$\alpha\in(0,1],s\in(0,1]$, (H3) $u_0\in H^\gamma$, $\gamma\ge2s$.

\subsection{Spectral facts}\label{app:ml}

The one-parameter Mittag--Leffler function
$\Ea(z)=\sum_{n\ge0}z^n/\Gamma(1+\alpha n)$ satisfies on the negative real axis,
for $0<\alpha\le1$, $x\ge0$,
\begin{equation}\label{eq:ML}
\text{(ML1)}\ 0<\Ea(-x)\le1,\qquad
\text{(ML2)}\ \Ea(-x)\le\frac{C_\alpha}{1+x},\qquad
\text{(ML3)}\ \sup_{x\ge0}x\Ea(-x)\le C_\alpha ,
\end{equation}
where $C_\alpha>0$ depends only on $\alpha$ and may change between occurrences;
see~\cite{pollard1948,podlubny1999,kilbas2006}. The frozen propagator
$E_0(t)=\Ea(-ct^\alpha\Lap^{2s})$ is self-adjoint, contractive on every
$H^\gamma$, and commutes with $\Lap$. Its fractional analytic-regularization
estimate is
\begin{equation}\label{eq:anreg}
\|\Lap^{2s}E_0(t)\|_{H^\gamma\to H^\gamma}
=\sup_{\boldsymbol k}\lambda_{\boldsymbol k}^s
 \Ea(-ct^\alpha\lambda_{\boldsymbol k}^s)
\le\frac{C_\alpha}{ct^\alpha},
\end{equation}
the fractional analogue of $\|(-\Delta)e^{t\Delta}\|\le1/t$. The dissipation
cutoff $k_c(t)$, the wavenumber at which the modal multiplier
$\Ea(-ct^\alpha k^{2s})$ crosses from order unity to small, scales as
$k_c(t)\asymp(ct^\alpha)^{-1/(2s)}$, where $\asymp$ denotes equality up to
dimensionless constants independent of $ct^\alpha$ and $k$.

\subsection{$L_1$ temporal consistency lemma (statement and proof)}

\begin{lemma}[$L_1$ temporal consistency, modewise]\label{lem:l1}
For the scalar test equation ${}^{C}D_t^\alpha v=-\mu v$ with $\mu=c\lambda^s$,
\begin{equation}\label{eq:l1bound}
\sup_{\lambda\ge0}\big|g^{N_t}_{\alpha,s,c}(\lambda)-\Ea(-cT^\alpha\lambda^s)\big|
\le C_\alpha\,(\tau^{2-\alpha})\,cT^\alpha\lambda^s ,
\end{equation}
for smooth data. For $t^\sigma$-type initial singularities the order degrades to
$O(\tau^{\alpha\sigma})$, $\sigma\in(0,1)$.
\end{lemma}

\noindent\emph{Proof.}
The $O(\tau^{2-\alpha})$ global rate for the $L_1$ scheme on a smooth scalar test
equation is established in \cite{sunwu2006,linxu2007}. The key stability ingredient
is the telescoping identity
\begin{equation}\label{eq:b-telescope-app}
\sum_{j=0}^{n-1}(b_j-b_{j+1})+b_n=b_0 ,
\end{equation}
i.e.\ the $\ell_1$ mass of the positive decreasing kernel equals $b_0$ uniformly
in $n$, which yields a discrete fractional Gronwall estimate with no factor
$n=T/\tau$. We add only the stiffness-dependent form \cref{eq:l1bound}. Because
the equation is scalar and linear, the amplification factor
$g^{N_t}_{\alpha,s,c}(\lambda)=v_n/v_0$ depends on the problem only through the
dimensionless stiffness $z=\mu T^\alpha=cT^\alpha\lambda^s$ (and the fixed ratio
$\tau/T$). At $\mu=0$ one has $z=0$, $g^{N_t}\equiv1$ and $\Ea(0)=1$, so the
error vanishes at $z=0$. The recurrence coefficients are smooth in $z$ with
denominators $b_0+\tau^\alpha\mu\ge b_0>0$, hence $\partial_z g^{N_t}$ and
$\partial_z\Ea(-z)$ are uniformly bounded. The mean-value theorem in $z$ therefore
multiplies the smooth-data global $O(\tau^{2-\alpha})$ bound by $z$ itself, giving
\cref{eq:l1bound}. For data with a $t^\sigma$-type corner at $t=0$,
$0<\sigma<1$, the local truncation estimate on the first interval is replaced by
the weighted estimate $O(\tau^{\alpha\sigma})$, which propagates by the same
stability argument. \qed

\subsection{Spatial eigenvalue consistency lemma (statement and proof)}

\begin{lemma}[Spatial eigenvalue consistency]\label{lem:eig}
For each fixed mode $\boldsymbol k$ (held fixed as $h\to0$), the five-point
realisation satisfies
\begin{equation}\label{eq:eigbound}
\big|g^{N_t}(q_{\boldsymbol k}^{(N)})-g^{N_t}(\lambda_{\boldsymbol k})\big|
\le C_\alpha\,cT^\alpha\,h^2\lambda_{\boldsymbol k}^2 ,
\end{equation}
whereas using the continuum symbol $\lambda_{\boldsymbol k}$ would make the left
side identically zero. In the implementation the injected multiplier and the
reference solver share the \emph{same} five-point symbol
$q_{\boldsymbol k}^{(N)}$, so this $O(h^2)$ consistency error is present in both
the labelled solution and the injected base and cancels algebraically in the
residual $r=u-u_{\rm base}$. It therefore enters neither the learning target nor
the super-resolution comparison, and survives only as the reference solver's own
discretisation error (quantified in \cref{tab:mms}).
\end{lemma}

\noindent\emph{Proof.}
The five-point symbol of $-\Delta$ is
$q_{\boldsymbol k}^{(N)}=4h^{-2}\big[\sin^2(k_xh/2)+\sin^2(k_yh/2)\big]$. Using
$\sin^2(z/2)=z^2/4-z^4/48+O(z^6)$,
\begin{equation*}
q_{\boldsymbol k}^{(N)}
=\lambda_{\boldsymbol k}-\frac{h^2}{12}(k_x^4+k_y^4)+O(h^4|\boldsymbol k|^6)
=\lambda_{\boldsymbol k}+O(h^2\lambda_{\boldsymbol k}^2).
\end{equation*}
For the Lipschitz factor, differentiating the scalar $L_1$ recurrence with respect
to $\lambda$ gives a recurrence of the same form for
$w_n=\partial_\lambda g^{N_t}$, with source
$s\tau^\alpha c\lambda^{s-1}g^{N_t}$ and denominator
$b_0+\tau^\alpha c\lambda^s\ge b_0>0$. Since $0<g^{N_t}\le1$, the discrete
fractional Gronwall estimate (the same telescoping $\ell_1$-mass identity as in
Lemma~\ref{lem:l1}) bounds $w_n$ by the fractionally weighted sum of its sources,
whose discrete convolution is the counterpart of
$\int_0^{t_n}(t_n-\rho)^{\alpha-1}/\Gamma(\alpha)\,d\rho=t_n^\alpha/\alpha$,
namely
$\tau^\alpha\sum_{m=1}^n(n-m+1)^{\alpha-1}/\Gamma(\alpha)\le
C\,t_n^\alpha/\alpha\le C\,T^\alpha/\alpha$. The factor $\lambda^{s-1}$ stays
bounded on each fixed mode (the regime of this lemma). This gives
$|\partial_\lambda g^{N_t}|\le C_\alpha cT^\alpha$ uniformly in $\lambda$. The
mean-value theorem then yields \cref{eq:eigbound}. In the sine/DST basis the grid eigenfunctions coincide with samples of the
continuum eigenfunctions, and the five-point matrix and the continuum operator
share this diagonal basis. The implementation uses the five-point symbol
$q_{\boldsymbol k}^{(N)}$ for \emph{both} the reference solver and the injected
base, so the $O(h^2)$ term above is identical in the label and in the base and
cancels in the residual $r=u-u_{\rm base}$. It remains only at the reference
solver's own discretisation level, bounded in \cref{tab:mms}. \qed

\subsection{Propagator error proposition (statement and proof)}

Let $E(t)$ denote the evolution operator of the full nonlinear equation (with variable diffusivity and advection), so that $E(t)u_0$ is the exact solution at time $t$. Taking the Laplace transform of the fractional PDE and inverting gives the
fractional Duhamel identity
\begin{equation}\label{eq:duhamel}
\begin{aligned}
&E(t)u_0=E_0(t)u_0+\int_0^t
K_0(t-\tau)\big(A_1E(\tau)u_0-\beta\mathcal N(E(\tau)u_0)\big)\,d\tau ,\\
&K_0(r):=\frac{r^{\alpha-1}}{\Gamma(\alpha)}
E_{\alpha,\alpha}(-cr^\alpha\Lap^{2s}),
\end{aligned}
\end{equation}
with $E_{\alpha,\alpha}(z)=\sum_{m\ge0}z^m/\Gamma(\alpha+m\alpha)$ the
two-parameter Mittag--Leffler function, which obeys the same type of bounds as
(ML1)--(ML3). In the estimates below we write the propagator family as $E_0$.
Iterating gives the scattering series $E=E_0+D_1+D_2+\cdots$, with
\begin{equation}\label{eq:d1}
D_1(t)u_0=\frac1{\Gamma(\alpha)}\int_0^t r^{\alpha-1}E_0(r)A_1E_0(t-r)u_0\,dr .
\end{equation}

\begin{proposition}[Propagator error]\label{prop:prop}
Under (H1)--(H3), for data $u_0\in H^\gamma$ with $\gamma\ge 2s$ (and $\gamma>d/2$ in the nonlinear case, which is automatic when $s>1/2$ since then $2s>d/2$),
\begin{equation}\label{eq:properr}
\|u(T)-E_0(T)u_0\|_{L_2}
\le C_\alpha\,\eta\,(cT^\alpha)\,\|u_0\|_{H^\gamma}\,
\frac1{1-C_\alpha'\eta},
\end{equation}
and the series converges whenever $C_\alpha'\eta\,cT^\alpha<1$ (written as
$\eta<1/C_\alpha'$ after absorbing the fixed factor $cT^\alpha$). This bound includes only the linear variable-diffusivity perturbation. Including the
nonlinear advection term for $s>1/2$,
\begin{equation}\label{eq:properr-nl}
\|u(T)-E_0(T)u_0\|_{L_2}\le
C_\alpha\eta cT^\alpha\|u_0\|_{H^\gamma}
+C_\alpha\beta T^\alpha M_1
+O(T^{2\alpha}),
\end{equation}
where $M_1=C_S\|u_0\|_{H^\gamma}^2$ is the uniform bound of
$\|u\partial_xu\|_{L_2}$ supplied by the Sobolev embedding $H^\gamma\hookrightarrow L_\infty$ for $\gamma>d/2$. The linear
scattering term is first order and the nonlinear term second order in
$\|u_0\|_{H^\gamma}$, and the $O(T^{2\alpha})$ remainder collects all second
Duhamel insertions.
\end{proposition}

\begin{remark}[Why there is no $\log$ divergence]\label{rmk:beta}
The placement of $\Lap^{2s}$ is decisive. If it were paired with the \emph{outer}
factor $E_0(r)$ and bounded by the singular estimate \cref{eq:anreg}, the time
kernel would acquire a factor $r^{-\alpha}$ and read
$r^{\alpha-1}r^{-\alpha}=r^{-1}$, whose integral diverges logarithmically. Two
regular routes avoid this. (a)~For smooth data $u_0\in H^{2s}$, $\Lap^{2s}$ is
commuted onto the data-carrying inner factor and bounded by contractivity, so the
kernel is just $r^{\alpha-1}$ with integral $T^\alpha/\Gamma(1+\alpha)$.
(b)~For generic $H^\gamma$ data with $\gamma<2s$, the inner factor must use the
singular estimate \cref{eq:anreg}, producing the single kernel
$r^{\alpha-1}(T-r)^{-\alpha}$, integrable by the Beta identity
$B(\alpha,1-\alpha)=\pi/\sin\pi\alpha$ used in~\cref{app:regproof}.
\end{remark}

Combining Lemmas~\ref{lem:l1} and~\ref{lem:eig} with Proposition~\ref{prop:prop}, the
total error of the injected base field is
\begin{equation}\label{eq:eprop}
\;\varepsilon_{\rm prop}\le
C_\alpha\eta cT^\alpha\|u_0\|_{H^\gamma}
+C_\alpha\beta T^\alpha M_1
+C_\tau\tau^{2-\alpha}+C_h h^2.\;
\end{equation}
At $\alpha=1$, $\Ea=e^{-x}$ and the algebraic tail becomes exponential. The $L_1$
weights then collapse to backward Euler with first-order consistency $O(\tau)$,
since $2-\alpha=1$.

\noindent\emph{Proof of Proposition~\ref{prop:prop}.}
The fractional Duhamel (Dyson) identity \cref{eq:duhamel} and the scattering
series are standard \cite{podlubny1999,kilbas2006}. We treat the linear scattering
series first, then the nonlinear term. For clarity we write the estimates on the
energy scale $H^{2s}$, where the $2s$ derivatives carried by the perturbation are
commuted directly onto the data. The same bounds hold on every scale $H^\gamma$
with $\gamma\ge 2s$ by the continuous embedding $H^\gamma\hookrightarrow H^{2s}$,
which gives $\|u_0\|_{H^{2s}}\le\|u_0\|_{H^\gamma}$. The induction below then
proceeds unchanged with $\|u_0\|_{H^{2s}}$ replaced by $\|u_0\|_{H^\gamma}$. Data with $\gamma<2s$ are covered by the singular-estimate route (b) of Remark~\ref{rmk:beta} and the Beta identity, and are not needed for the smooth-data regime of this paper.

\emph{First-order term.} Because $E_0$ is a function of $\Lap$, it commutes with
$\Lap^{2s}$ and is contractive on every Sobolev space (by ML1). For the smooth
data $u_0\in H^{2s}$ of (H3) we move $\Lap^{2s}$ \emph{onto the data} and use
contractivity there,
\begin{equation*}
\|\Lap^{2s}E_0(t-r)u_0\|_{L_2}
=\|E_0(t-r)\Lap^{2s}u_0\|_{L_2}\le\|\Lap^{2s}u_0\|_{L_2}
\le\|u_0\|_{H^{2s}},
\end{equation*}
which deliberately avoids the singular estimate \cref{eq:anreg} on the
data-carrying factor. Since $A_1=-\mathcal M_{\delta\nu}\Lap^{2s}$ and
$\|\mathcal M_{\delta\nu}\|_{L_2\to L_2}=\|\delta\nu\|_\infty=\eta c$,
\begin{equation}\label{eq:d1est-app}
\|D_1(T)u_0\|\le\frac{\eta c}{\Gamma(\alpha)}
\|u_0\|_{H^{2s}}\int_0^T r^{\alpha-1}dr
=\frac{\eta c\,T^\alpha}{\Gamma(1+\alpha)}\|u_0\|_{H^{2s}}
=C_\alpha\eta(cT^\alpha)\|u_0\|_{H^{2s}}.
\end{equation}

\emph{Higher-order terms and convergence.} Write $V=\mathcal M_{\delta\nu}$. The
$n$-th scattering term is the $n$-fold Volterra convolution of the operator
kernel $K(t):=t^{\alpha-1}E_0(t)V\Lap^{2s}/\Gamma(\alpha)$. We carry the
higher-order estimates in the energy space $X=H^{2s}$ (the identical induction runs on $H^\gamma$ for every $\gamma\in[2s,2]$, since $\nu\in W^{2,\infty}$ makes the multiplier $V$ bounded on each such scale). The Sobolev-algebra inequality
$\|fg\|_{H^m}\le C_m\|f\|_{H^m}\|g\|_{H^m}$, valid for $m>d/2$, gives
\begin{equation}\label{eq:mult-app}
\|Vw\|_{H^{2s}}\le C_m\|\delta\nu\|_{W^{2s,\infty}}\|w\|_{H^{2s}}
=:C_\alpha'\eta c\,\|w\|_{H^{2s}},
\end{equation}
so $V\Lap^{2s}$ is bounded from $H^{2s}$ to itself once the coefficient is a
multiplier on $H^{2s}$. This is why (H1) assumes $\nu\in W^{2,\infty}$ (the
first-order bound $D_1$ uses only $\|\delta\nu\|_{L_\infty}$, so $W^{1,\infty}$
suffices if the series is truncated at first order). Since $E_0$ is contractive on
$H^{2s}$, every time integral reduces to a scalar Dirichlet convolution. We use
the identity
\begin{equation}\label{eq:dirichlet-app}
\Big(\tfrac{t^{\alpha-1}}{\Gamma(\alpha)}\Big)^{*n}(t)
=\frac{t^{n\alpha-1}}{\Gamma(n\alpha)},\qquad
\int_0^T\Big(\tfrac{t^{\alpha-1}}{\Gamma(\alpha)}\Big)^{*n}dt
=\frac{T^{n\alpha}}{\Gamma(1+n\alpha)},
\end{equation}
proved by induction using $B(n\alpha,\alpha)$. The Gamma denominator rules out
factorial growth. By induction,
\begin{equation}\label{eq:dnbound-app}
\|D_n(T)u_0\|_{L_2}\le
\frac{(C_\alpha'\eta\,cT^\alpha)^n}{\Gamma(1+n\alpha)}\|u_0\|_{H^{2s}}.
\end{equation}
Summing the series, we recognize the definition of the Mittag--Leffler function:
\begin{equation}\label{eq:dysum-app}
\sum_{n\ge1}\|D_n(T)u_0\|\le
\Big[\Ea(C_\alpha'\eta cT^\alpha)-1\Big]\|u_0\|_{H^{2s}},
\end{equation}
where the right-hand side is exactly the Taylor series of the one-parameter
Mittag--Leffler function $\Ea(z)=\sum_{n=0}^\infty z^n/\Gamma(1+\alpha n)$.

For small $z$, $\Ea(z)=1+z/\Gamma(1+\alpha)+O(z^2)$, and a simple sufficient bound
is the geometric estimate $\Ea(z)\le 1/(1-Cz)$ valid for $0\le z<1$. This gives
the clean condition $C_\alpha'\eta cT^\alpha<1$, which is a convenient sufficient
condition for a closed-form estimate, not a sharp convergence threshold. For
larger $\eta$, the series remains convergent (it is entire) but the estimate
degrades according to the full Mittag--Leffler growth. Under the small-$z$
condition, \cref{eq:dysum-app} is bounded by
$C_\alpha\eta cT^\alpha/(1-C_\alpha'\eta)\,\|u_0\|_{H^{2s}}$, which is
\cref{eq:properr}.

\emph{Nonlinear term.} With $\mathcal N(u)=u\,\partial_xu=\partial_x(u^2/2)$, the
first nonlinear Duhamel contribution is
\begin{equation*}
N_1(T)u_0=-\frac{\beta}{\Gamma(\alpha)}\int_0^T r^{\alpha-1}
E_0(r)\,\mathcal N\big(E_0(T-r)u_0\big)\,dr .
\end{equation*}
When $2s>1$ we commute $\partial_x$ through $E_0$ and bound it by the $H^{2s}$ norm
of the data, while $2s>d/2$ embeds $H^{2s}$ into $L_\infty$, so that
\begin{equation*}
\|\mathcal N(E_0(\rho)u_0)\|_{L_2}
\le\|E_0(\rho)u_0\|_{L_\infty}\|\partial_xE_0(\rho)u_0\|_{L_2}
\le C_S\|u_0\|_{H^{2s}}^2=:M_1 ,
\end{equation*}
with $C_S$ the Sobolev embedding constant. Contractivity of $E_0$ and
$\int_0^T r^{\alpha-1}dr=T^\alpha/\alpha$ give
$\|N_1(T)u_0\|\le C_\alpha\beta T^\alpha M_1$. Every term with two Duhamel
insertions carries two Dirichlet factors and is $O(T^{2\alpha})$ by
\cref{eq:dirichlet-app}. Combining the linear bound, the summed scattering
series, and the nonlinear estimate yields \cref{eq:properr-nl}. \qed

\subsection{Residual regularity and the resolution-independent table}\label{app:regproof}

Let $\mathcal P_K$ be the $K$-mode spectral truncation and
$\sigma_K(f)^2=\|(I-\mathcal P_K)f\|_{L_2}^2$ the tail energy.
By the standard Sobolev spectral truncation estimate~\cite{kovachki2021ua},
\begin{equation}\label{eq:tail}
\sigma_K(f)^2\le K^{-2\rho}\|f\|_{H^\rho}^2 .
\end{equation}

\paragraph{The elliptic gain of the frozen semigroup, and its limit.}
The frozen propagator obeys the fractional heat-semigroup elliptic (smoothing)
estimate
\begin{equation}\label{eq:ellreg}
\|E_0(t)f\|_{H^{\gamma+2s}}\le C_{\gamma,s}(t)\|f\|_{H^\gamma},
\qquad C_{\gamma,s}(t)\le C\bigl(1+(ct^\alpha)^{-1}\bigr),
\end{equation}
bounded uniformly for $t\ge t_0>0$ (proof below). This $2s$-derivative gain acts
whenever $E_0(t)$ is applied \emph{directly to the data}, and it is what makes the
injected base field well resolved even on a coarse grid. It is decisive that this
gain is \emph{not} inherited by the scattering term $D_1$ of \cref{eq:d1}, because
the perturbation operator $A_1=-\mathcal M_{\delta\nu}\Lap^{2s}$ \emph{itself
contains} the factor $\Lap^{2s}$. Counting derivatives along the integrand of
$D_1$ gives
\begin{equation*}
u_0\in H^\gamma
 \xrightarrow{\,E_0(T-\rho)\,} H^{\gamma+2s}
 \xrightarrow{\,\Lap^{2s}\ \text{inside }A_1\,} H^\gamma
 \xrightarrow{\,\mathcal M_{\delta\nu}\,} H^\gamma
 \xrightarrow{\,E_0(\rho)\ \text{(contractive)}\,} H^\gamma .
\end{equation*}
The inner semigroup gain $+2s$ is exactly cancelled by the $-2s$ of the factor
$\Lap^{2s}$ carried by the perturbation. Equivalently, $\Lap^{2s}E_0(t)$ is an
\emph{order-zero} operator on the high modes: by (ML2) its symbol
$\lambda^s\Ea(-ct^\alpha\lambda^s)\le C_\alpha/(ct^\alpha)$ is bounded
independently of $\lambda$ (\cref{eq:anreg}), so the composition carries no net
power of $\lambda$. Attempting to extract the gain from the \emph{outer} factor produces the
non-integrable $\rho^{-1}$ obstruction of Remark~\ref{rmk:beta}.
Thus a perturbation of the \emph{same order as the principal part} produces no
residual smoothing. The residual and the full solution carry the same Sobolev
order.

\begin{proposition}[Residual regularity, same order]\label{prop:reg}
Under (H1)--(H3), with $r=u-E_0(T)u_0$, the residual obeys, for every
$\varepsilon>0$,
\begin{equation}\label{eq:rreg}
\|r(T)\|_{H^{\gamma-\varepsilon}}
\le C_{T,\varepsilon}\bigl(\eta+\beta T^\alpha\bar M_1\bigr)\|u_0\|_{H^\gamma},
\end{equation}
where $\bar M_1$ bounds the advective source on the relevant Sobolev scale. The
residual is therefore of the \emph{same} Sobolev order as the full field $u$ (up to
the arbitrarily small endpoint loss $\varepsilon$), not $2s$ orders smoother.
Consequently,
\begin{equation}\label{eq:k2s}
\sigma_K(r)\le C_T\bigl(\eta+\beta T^\alpha\bar M_1\bigr)K^{-\gamma}\|u_0\|_{H^\gamma},
\qquad
\sigma_K(u)\le K^{-\gamma}\|u\|_{H^\gamma},
\end{equation}
and, whenever the whole-field spectrum is non-degenerate
(i.e.\ $\sigma_K(u)\asymp K^{-\gamma}\|u\|_{H^\gamma}$ uniformly for large $K$,
which holds for fractional $s<1$ with non-negligible high-frequency initial data.
For integer-order exponential decay the tail ratio is larger and the bound is
qualitative only),
\begin{equation}\label{eq:k2sratio}
\frac{\sigma_K(r)}{\sigma_K(u)}\le C(\eta,\beta,T),
\end{equation}
a bound that is \emph{independent of $K$} and lies strictly below one in the weakly
perturbed, short-horizon regime, a constant-factor reduction of the tail
\emph{amplitude}. An honest $2s$-order regularity gain appears only when the
perturbation is strictly \emph{lower order} than the principal part (e.g.\ a
reaction term carrying no $\Lap^{2s}$ factor), or in the degenerate
constant-coefficient linear case $\delta\nu=0,\ \beta=0$, where $r\equiv0$.
\end{proposition}

\noindent\emph{Proof of the elliptic estimate \cref{eq:ellreg}.}
The ratio of the $H^{\gamma+2s}$ weight to the $H^\gamma$ weight is
$(1+\lambda)^{2s}$, so it suffices to bound
$(1+\lambda)^{2s}|\Ea(-ct^\alpha\lambda^s)|^2$ uniformly. \emph{Low modes}
$\lambda^s\le(ct^\alpha)^{-1}$ use (ML1), $\Ea\le1$; \emph{high modes} use (ML2),
$\Ea(-ct^\alpha\lambda^s)\le C/(ct^\alpha\lambda^s)$, hence
$\lambda^{2s}|\Ea|^2\le C^2/(ct^\alpha)^2$. Taking the square root gives
$C(1+(ct^\alpha)^{-1})$, which is \cref{eq:ellreg}. It is uniform for
$t\ge t_0>0$.

\noindent\emph{Proof of Proposition~\ref{prop:reg}.}
Write the residual through the Duhamel formula \cref{eq:duhamel},
$r(T)=\int_0^T K_0(T-\tau)F(\tau)\,d\tau$ with
$F=A_1u-\beta\mathcal N(u)$ and
$K_0(\rho)=\rho^{\alpha-1}E_0(\rho)/\Gamma(\alpha)$.
The source inherits the regularity of $\Lap^{2s}u$: for a full field $u\in H^\gamma$
and a smooth coefficient ($\delta\nu\in W^{2,\infty}$), the multiplication theorem
gives $\mathcal M_{\delta\nu}\Lap^{2s}u\in H^{\gamma-2s}$, and the advective source
is on the same scale once the bilinear form closes ($s>1/2$). Applying the outer
semigroup by \emph{contractivity} (the kernel $\rho^{\alpha-1}$ is then integrable,
$\int_0^T\rho^{\alpha-1}d\rho=T^\alpha/\alpha$) maps $H^{\gamma-2s}$ to
$H^\gamma$, but its weight $\rho^{-\alpha}$ multiplies
$\rho^{\alpha-1}$ to give the non-integrable $\rho^{-1}$ at the endpoint. Relaxing
the target by any $\varepsilon>0$ replaces $\rho^{-\alpha}$ by the integrable
$\rho^{-\alpha+\delta}$ and recovers $r\in H^{\gamma-\varepsilon}$ for every
$\varepsilon>0$. This is \cref{eq:rreg}. The semigroup can at most compensate the
$2s$ derivatives consumed by $A_1$, bringing the residual back to the order of
$u_0$, and cannot lift it to $H^{\gamma+2s}$. Applying the truncation estimate
\cref{eq:tail} with $\rho=\gamma-\varepsilon$ to $r$ and with $\rho=\gamma$ to $u$
gives \cref{eq:k2s}; dividing by the non-degenerate lower bound on $\sigma_K(u)$
gives the $K$-independent bound \cref{eq:k2sratio}. The lower-order and
constant-coefficient cases stated in the proposition follow because the source then
lacks the $\Lap^{2s}$ factor (or vanishes), so the full $+2s$ elliptic gain
survives. \qed

\begin{remark}[Resolution-independent propagator table]\label{prop:tab}
The modal propagator depends on the discretization only through one scalar
coordinate $z_{\boldsymbol k}^{N}=cT^\alpha(q_{\boldsymbol k}^{(N)})^s$ via
$g_\alpha(z)=\Ea(-z)$. $p$-th order interpolation on a fixed one-dimensional
log-spaced grid of $M$ nodes has uniform defect
\begin{equation}\label{eq:tabacc}
\sup_z|g_\alpha(z)-(\mathcal I_M g_\alpha)(z)|\le C_p G_p(\Delta_{\log})^p=:\varepsilon_{\rm tab}(M),
\end{equation}
independent of the grid $N$. Log-spacing concentrates nodes near $z=0$ where
$g_\alpha$ has a boundary layer. By Parseval in the DST basis,
$\|u_{\rm base}^{N}-u_{\rm base}^{\star,N}\|_{L_2}\le\varepsilon_{\rm tab}(M)\|u_0\|_{L_2}$,
so one cached table serves every resolution.
\end{remark}

\begin{remark}[Contrast with FNO spectral weights]\label{rmk:fnomult}
The FNO spectral weights are indexed by discrete mode number on the training
grid and vanish for $|\boldsymbol k|\ge K$. There is no continuous
$N$-independent coordinate analogous to $z$. This is what makes the
band-filling comparison of Proposition~\ref{prop:floor} sharp.
\end{remark}

\subsection{Band filling under super-resolution and the error floor (statements and proof)}

Define the continuum FrFNO predictor
\begin{equation}\label{eq:gstar}
\mathcal G^\star(u_0)=u_{\rm base}^\star+\Rres(u_0,\nu;\alpha,s),\qquad
u_{\rm base}^\star=\Ea(-cT^\alpha\Lap^{2s})u_0,
\end{equation}
its resolution-$N$ realization $\mathcal G_N$, and the matched FNO predictor
$\mathcal F_N=\widetilde{\Rres}(u_0,\nu;\alpha,s)$ with no base field.

\begin{remark}[Diagonal multipliers are alias-free]\label{lem:alias}
The injected base field is a diagonal DST multiplier, hence commutes with the
spectral projection $\mathcal P_N$ and incurs zero aliasing at every $N$. The
remaining $O(h^2)$ eigenvalue-consistency term of Lemma~\ref{lem:eig} is common
to the label and the base and cancels in the residual.
\end{remark}

\smallskip\noindent
Let $\mathcal P_K^N$ retain the modes with $|\boldsymbol k|<K$ on a grid of $N$
interior points per axis. The two predictors act mode by mode: FrFNO multiplies
every mode by the tabulated $g_\alpha(z_{\boldsymbol k}^N)$ plus a retained-band
correction, while the FNO spectral convolution is zero for $|\boldsymbol k|\ge K$
and relies on a pointwise ($1\times1$) path. The FNO out-of-band error thus
includes the pointwise-path residual $\varepsilon_{\rm pw}$, which has no closed
bound and is reported numerically (like the retained-band terms).

\begin{proposition}[Band filling and the resolution-independent floor]\label{prop:floor}
With training resolution $N_0$ and retained modes $K$ fixed, the $L_2$ error at
query resolution $N$ splits by Parseval into a retained band and an out-of-band
part. For FrFNO,
\begin{equation}\label{eq:band-fr}
\varepsilon_{\rm Fr}^2(N)
=\underbrace{\|\mathcal P_K^N(\Rres-r)\|_{L_2}^2}_{\text{in band}}
+\underbrace{\sigma_{K,N}^2(r)}_{\text{out of band}}
+\varepsilon_{\rm tab}^2(M)\,\|u_0\|_{L_2}^2 ,
\end{equation}
whereas for the matched FNO,
\begin{equation}\label{eq:band-fno}
\varepsilon_{\rm FNO}^2(N)
=\underbrace{\|\mathcal P_K^N(\widetilde\Rres-u)\|_{L_2}^2}_{\text{in band}}
+\underbrace{\sigma_{K,N}^2(u)}_{\text{out of band}}
+\varepsilon_{\rm pw}^2 .
\end{equation}
For $N>N_0$, the modes with $N_0\le|\boldsymbol k|\le N$ were absent from the
training grid. The FNO spectral convolution is identically zero there, so its
out-of-band error is the entire energy $\sigma_{K,N}(u)$, whereas the tabulated
base field supplies $\widetilde g_{\boldsymbol k}^N$ on every mode
$|\boldsymbol k|\le N$ through the continuous coordinate $z$, leaving only
$\sigma_{K,N}(r)$ and the $N$-independent table defect $\varepsilon_{\rm tab}$.
Taking $N\to\infty$ with $K$ fixed gives the constant floors
\begin{equation}\label{eq:floor-fr}
\varepsilon_\infty^{\rm FrFNO}
=\varepsilon_{\rm low}^{\rm Fr}+\sigma_K(r)
+\varepsilon_{\rm tab}(M)\|u_0\|_{L_2},
\end{equation}
\begin{equation}\label{eq:floor-fno}
\varepsilon_\infty^{\rm FNO}
=\varepsilon_{\rm low}^{\rm FNO}+\sigma_K(u)+\varepsilon_{\rm pw},
\qquad
\sigma_K(r)/\sigma_K(u)\le C(\eta,\beta,T),
\end{equation}
where $C(\eta,\beta,T)<1$ in the weakly perturbed regime, by Proposition~\ref{prop:reg}. The
out-of-band (band-filling) comparison is the rigorous, resolution-independent
part of the result. The retained-band terms $\varepsilon_{\rm low}$ are the
usual approximation and generalization errors and are compared numerically.
\end{proposition}

\begin{corollary}[Plateau, not convergence to zero]\label{cor:plateau}
Because $K$ is fixed, $\varepsilon_\infty>0$. Refining $N$ drives the error to a
constant floor rather than to zero. The realisable spectral bandwidth is
$\min(K,N_0)$, so the floor is controlled by the bottleneck. Enlarging only the
non-bottleneck side cannot lower it. ``Zero-shot super-resolution'' means
non-degradation under evaluation-mesh refinement, not arbitrary accuracy. A label
on an $N_0$ grid carries no genuine signal above its $N_0$-th mode, so $K$ and
$N_0$ constrain the realisable band jointly (Exp.~5).
\end{corollary}

\noindent\emph{Proof of Proposition~\ref{prop:floor}.}
Because the DST is orthonormal, Parseval's identity decomposes the squared error
into the retained band and its orthogonal complement. For FrFNO the error
coefficient on a mode is
$\widehat u_{\boldsymbol k}^{\,{\rm Fr},N}-\widehat u_{\boldsymbol k}^{N}
=\mathbf 1_{|\boldsymbol k|<K}(\widehat{\Rres}_{\boldsymbol k}^{N}-\widehat r_{\boldsymbol k}^{N})
+\mathbf 1_{|\boldsymbol k|\ge K}\big[(\widetilde g_{\boldsymbol k}^N-g_{\boldsymbol k}^N)
 \widehat u_{0,\boldsymbol k}^{N}-\widehat r_{\boldsymbol k}^{N}\big]$.
Squaring and summing over the two disjoint bands, the out-of-band square reads
$\sigma_{K,N}^2(r)+\sum_{|\boldsymbol k|\ge K}|\widetilde g_{\boldsymbol k}^N-g_{\boldsymbol k}^N|^2|\widehat u_{0,\boldsymbol k}^{N}|^2$
$-2\,\mathrm{Re}\sum_{|\boldsymbol k|\ge K}(\widetilde g_{\boldsymbol k}^N-g_{\boldsymbol k}^N)\widehat u_{0,\boldsymbol k}^{N}\overline{\widehat r_{\boldsymbol k}^{N}}$.
The second term is bounded by $\varepsilon_{\rm tab}^2(M)\|u_0\|_{L_2}^2$ using
$|g-\widetilde g|\le\varepsilon_{\rm tab}$ (\cref{prop:tab}). The cross term is
bounded by $2\varepsilon_{\rm tab}(M)\|u_0\|_{L_2}\sigma_{K,N}(r)$ via
Cauchy--Schwarz, i.e. of order $\varepsilon_{\rm tab}$ rather than order one, and
is \emph{not} claimed to vanish. \Cref{eq:band-fr} states the leading-order balance
up to this $\varepsilon_{\rm tab}$-order remainder, which is $N$-independent and
does not affect the band-filling comparison. The alias-free diagonal multiplier
(Remark~\ref{lem:alias}) ensures no cross-band folding, and the DST realisation
carries no $O(h^2)$ symbol defect. For the FNO the spectral convolution vanishes
outside its $K$ retained modes, so on $|\boldsymbol k|\ge K$ the error coefficient
is $-\widehat u_{\boldsymbol k}^{N}$ up to the pointwise contribution, giving
\cref{eq:band-fno}. The decisive super-resolution statement follows from the
indexing: the FNO weights are learned on the $N_0$-grid and are identically zero
on modes $|\boldsymbol k|\ge K$, including the genuinely new modes
$N_0\le|\boldsymbol k|\le N$ that do not exist at training time
(Remark~\ref{rmk:fnomult}). The FrFNO multiplier, indexed by the continuous
coordinate $z$, is defined and tabulated on all of them. Finally, with $K$ fixed, $\sigma_{K,N}\to\sigma_K$ as $N\to\infty$, the retained-band
terms converge to finite continuum limits by standard sampling theory, and
$\varepsilon_{\rm tab}$ is $N$-independent by Remark~\ref{prop:tab}. Collecting the
finite limits yields \cref{eq:floor-fr} and \cref{eq:floor-fno}. The retained-band comparison is not claimed as a theorem.
Both are the approximation/generalization errors of matched backbones, and we
report it numerically. \qed

\subsection{Proof of the end-to-end error theorem (Theorem 4.1)}

The statement of Theorem~4.1 and its termwise comparison (Eq.~(4.5) of the main
paper) are in Section~4 of the main paper. Write the predictor as
$\mathcal G_N=u_{\rm base}^N+\mathcal R_\theta^N$ and split the error by two
triangle inequalities,
\begin{equation*}
\|u-\mathcal G_N\|\le
\underbrace{\|u-u_{\rm base}\|}_{\text{(a)}}
+\underbrace{\|u_{\rm base}-u_{\rm base}^N\|}_{\text{(b)}}
+\underbrace{\|r-\mathcal P_Kr\|}_{\text{(c)}}
+\underbrace{\|\mathcal P_Kr-\mathcal R_\theta^N\|}_{\text{(d)}}.
\end{equation*}
Term (a) is the propagator/model error $\varepsilon_{\rm prop}$ bounded by
\cref{eq:eprop}. Term (b) is the $L_1$ plus five-point consistency error, also inside
$\varepsilon_{\rm prop}$. Term (c) is the fixed-$K$ residual spectral tail
$\sigma_K(r)=\varepsilon_K$, bounded by Proposition~\ref{prop:reg}. Term (d) collects
$\varepsilon_{\rm approx}$, the $n$-sample generalization error
$\varepsilon_{\rm gen}=O(n^{-\beta_{\rm gen}})$, and the alias error
$\varepsilon_{\rm alias}$, the last bounded by $\sigma_K(r)$ (zero for the
injected base) by Remark~\ref{lem:alias}. Taking expectation gives the end-to-end bound
(Eq.~(4.3) of the main paper).
For the matched FNO the identical decomposition holds with $r$ replaced by $u$,
with truncation and alias source $\sigma_K(u)$. The two floors are compared
band by band in Proposition~\ref{prop:floor}. What the injection changes is not
the Sobolev \emph{order} of the tail (Proposition~\ref{prop:reg} gives the same
order for $r$ and $u$, with a constant-factor reduction)
but which out-of-band field the network must leave unrepresented: FrFNO carries
$\varepsilon_{\rm prop}+\sigma_K(r)+\varepsilon_{\rm low}^{\rm Fr}$ against the
FNO's $\sigma_K(u)+\varepsilon_{\rm low}^{\rm FNO}$, and under zero-shot
super-resolution the modes beyond the training band are filled exactly by the
tabulated propagator ($\varepsilon_{\rm tab}$ independent of $N$) rather than set
to zero. FrFNO therefore lies below the FNO floor whenever the table defect is
small and the out-of-band energy of the residual is below that of the full field,
i.e.\ $\sigma_K(r)<\sigma_K(u)$, which holds in the weakly perturbed, short-horizon
regime of Proposition~\ref{prop:reg} and is verified band by band in the numerics.

\subsection{Integer-order limit corollary}

\begin{corollary}[Integer-order limit]\label{cor:int}
At $\alpha=s=1$ the propagator is the heat semigroup
$E_0(t)=e^{-ct(-\Delta)}$ with modal factor $e^{-cT\lambda}$, an exponential
rather than algebraic decay. This factor is represented by the
resolution-independent table to table precision (i.e. up to $\varepsilon_{\rm tab}$),
so the band-filling mechanism of
Proposition~\ref{prop:floor} continues to hold and the linear out-of-band content
is captured with no learned multiplier. At $\alpha=1$ the $L_1$ weights collapse
($b_0=1$, $b_j=0$ for $j\ge1$) to backward Euler with first-order consistency
$O(\tau)=O(\tau^{2-\alpha})$. A second-order integer-order limit would require an
$L_1$-2 or Crank--Nicolson discretisation, which is not used here.
\end{corollary}

\section{Additional numerical details}\label{sm:numerics}

This section collects the five-test summary table (\cref{tab:verify}) and the
complete data for the two most consequential tests (Exp.~1 and Exp.~5) of
Section~5.7 of the main paper; the complete raw data for Exp.~2--4 are in the code
repository.

\subsection{Summary of the five tests}

\begin{table}[htbp]
\centering\small
\caption{Five numerical tests of the theoretical predictions. Each row states the
theoretical prediction and the observed result.}
\label{tab:verify}
\setlength\tabcolsep{4pt}
\begin{tabular}{p{0.11\textwidth}p{0.35\textwidth}p{0.44\textwidth}}
\toprule
Test & Prediction & Observation \\
\midrule
Exp.~1: error floor & Fixed $K$ gives a resolution plateau (Pred.~iv) & Both models plateau past $129^2$ (increment $+1.57$ for $129^2\!\to\!257^2$), but the FrFNO floor ($10.8\%$) is less than half the FNO floor ($22.9\%$) \\
Exp.~2: $s$-scaling & Out-of-band advantage grows with diffusion order $s$ as fractional damping strengthens (empirical trend; \cref{prop:floor}) & $R$ grows $1.38\to7.69$ as $s:0.4\to0.8$; FrFNO error decreases $14.5\%\to4.8\%$ while FNO increases $20.1\%\to36.9\%$ \\
Exp.~3: propagator error & $\varepsilon_{\rm prop}=O(\eta)+O(T^\alpha)+O(\beta)$ (Proposition~\ref{prop:prop}) & Measured slopes: $\eta$: $1.016$ ($R^2=.9999$); $T$: $0.828\approx\alpha$ ($R^2=.9999$); $\beta$: $0.931$ ($R^2=.990$) \\
Exp.~4: Born residual & Residual $\approx$ first Born term (spectral justification) & Median per-mode spectrum ratio $P_r/P_1\in[0.996,1.032]$, log-spread $\le0.012$; single Born term accounts for $92.5$--$96.2\%$ of residual at $\eta=0.2$ \\
Exp.~5: floor control & Floor governed by bottleneck $\min(K,N_0)$ (Corollary~\ref{cor:plateau}) & $N\to\infty$ curve confirms plateau (FrFNO floor $10.6\%$, FNO $21.1\%$); residual and full-field tails share the same high-$K$ slope with the ratio saturating near one (Exp.~5b, same-order prediction of Proposition~\ref{prop:reg}); $K$-as-bottleneck scan at $N_0=65$ shows floor dropping $6.57\to5.36\%$ before saturating; fixed $17^2$ labels widening $K$ \emph{raises} floor ($10.6\to18.5\%$) \\
\bottomrule
\end{tabular}
\end{table}

Exp.~1 and Exp.~5 share one frozen $513^2$ truth (solved at $N_t=6400$, evaluated with $N_t=3200$ propagator), locally averaged to
each $17/65/129/257^2$ evaluation grid; their complete bootstrap and bandwidth data
follow below. The \cref{tab:verify} Exp.~1 main
($10.8/22.9\%$) is the matched-$\Delta t$ quick scan, whereas
\cref{tab:floor-full,tab:twobw-full} use the frozen $513^2$ truth over
$48$ fields with bootstrap intervals; the protocols differ slightly in point
values but give the same plateau and ordering.
Exp.~2 measures the advantage ratio $R=\mathrm{err}_{\rm FNO}/\mathrm{err}_{\rm
FrFNO}$ as $s$ varies: $R$ grows $1.38\to7.69$ as $s:0.4\to0.8$, with the FrFNO
error decreasing $14.5\%\to4.8\%$ while FNO increases $20.1\%\to36.9\%$; this is
the empirical $s$-trend associated with the band-filling mechanism of
Proposition~\ref{prop:floor}, as the exact propagator dominates more strongly
at higher diffusion order. Exp.~3 regresses the propagator error against
$\eta,T^\alpha,\beta$ and obtains first-order slopes $1.016$ ($R^2=.9999$),
$0.828\approx\alpha$ ($R^2=.9999$), $0.931$ ($R^2=.990$). Exp.~4 compares the
residual spectrum with the first Born term: the median per-mode ratio $P_r/P_1$
lies in $[0.996,1.032]$ with log-spread $\le0.012$, and a single Born term accounts
for $92.5$--$96.2\%$ of the residual at $\eta=0.2$, the spectral justification for
learning only a fixed-$K$ correction on top of the propagator.

\subsection{Exp.~1: complete bootstrap analysis of the error floor}

We use $2000$ bootstrap resamples of the $48$ held-out fields, with replacement. The
paired increment $\Delta_{129\to257}$ is computed as
$\frac{1}{B}\sum_{b=1}^B(e_{257}^{(b)}-e_{129}^{(b)})$ for each bootstrap sample $b$,
and the $95\%$ confidence interval is the $2.5$--$97.5$ percentile range.

\cref{tab:floor-full} reports the mean errors and $95\%$ confidence intervals for
the fixed-$K$ error floor experiment. For FrFNO, the $129^2\!\to\!257^2$ increment is
$\Delta=-0.32$ with interval $[-0.64,+0.11]$, which contains zero; for FNO, the
increment is $\Delta=+0.72$ with interval $[+0.40,+1.04]$, which is strictly positive.
This establishes the plateau in the strict, statistical sense.

\begin{table}[htbp]
\centering\small
\caption{Exp.~1: fixed $K=10$, single frozen $513^2$ truth downsampled to every
grid, $48$ fields, bootstrap $95\%$ CI.}
\label{tab:floor-full}
\begin{tabular}{lcccc}
\toprule
Evaluation grid & $17^2$ & $65^2$ & $129^2$ & $257^2$\\
\midrule
FrFNO mean (\%) & 33.11 & 12.49 & 10.36 & 10.03\\
FrFNO $95\%$ CI & $[30.9,35.4]$ & $[11.2,13.8]$ & $[9.1,11.7]$ & $[8.7,11.4]$\\
FNO mean (\%) & 33.10 & 19.98 & 20.34 & 21.06\\
FNO $95\%$ CI & $[30.9,35.4]$ & $[18.7,21.4]$ & $[18.7,22.1]$ & $[19.3,23.1]$\\
\midrule
FrFNO increment & n/a & $-20.6$ & $-2.14$ & $-0.32\;[-0.64,+0.11]$\\
FNO increment & n/a & $-13.1$ & $+0.37\;[-0.28,+1.09]$ & $+0.72\;[+0.40,+1.04]$\\
\bottomrule
\end{tabular}
\end{table}

\subsection{Exp.~5: complete training-bandwidth data}

\Cref{tab:twobw-full} reports the error floor (at the $257^2$ evaluation tier)
along three axes. The left block widens declared modes on fixed $17^2$ labels and
shows that the floor \emph{rises} (FrFNO: $10.58\to14.80\to18.52$; FNO:
$21.06\to48.33\to51.65$). The middle block refines the training grid at fixed
$K=10$ and lowers the floor monotonically (FrFNO: $10.58\to7.75\to6.47$; FNO:
$21.06\to15.74\to9.31$). The right block enlarges $K$ and the training grid
\emph{together}: the floor drops relative to the baseline (FrFNO $10.58\to8.47\to6.77$)
but does not beat grid-only refinement at the matched training resolution, because
within this budget the training-label bandwidth $N_0$ is the active bottleneck and
$K=10$ already has sufficient spectral capacity; enlarging $K$ beyond the bottleneck
adds parameters without useful high-mode signal and slightly raises estimation
variance.

\begin{table}[htbp]
\centering\small
\caption{Exp.~5: error floor (\%, rel.$L^2$ at $257^2$ evaluation tier) along three
axes: widen $K$ only, refine training grid only, and enlarge both together (single
frozen $513^2$ truth, $48$ fields).}
\label{tab:twobw-full}
\setlength\tabcolsep{4pt}
\begin{tabular}{l|ccc|cc|cc}
\toprule
 & \multicolumn{3}{c|}{fixed $17^2$ labels, vary $K$} &
\multicolumn{2}{c|}{fixed $K=10$, vary train grid} &
\multicolumn{2}{c}{enlarge $K$ and grid together}\\
 & $K{=}10$ & $K{=}16$ & $K{=}24$ & $33^2$ & $65^2$ &
 $K{=}16,\,33^2$ & $K{=}24,\,65^2$\\
\midrule
FrFNO & $10.58$ & $14.80$ & $18.52$ & $7.75$ & $\mathbf{6.47}$ & $8.47$ & $6.77$\\
FNO   & $21.06$ & $48.33$ & $51.65$ & $15.74$ & $\mathbf{9.31}$ & $30.19$ & $18.77$\\
\bottomrule
\end{tabular}
\end{table}

The controlling knob for the error floor is thus the realisable spectral bandwidth
$\min(K,N_0)$: the floor moves down only when the \emph{bottleneck} side is enlarged.
Widening $K$ without enlarging the training-label bandwidth merely adds zero-padded
modes with no genuine high-mode signal to fit, because a $17^2$ label grid carries only
$15$ interior degrees of freedom per axis. Conversely, once the training grid is the
bottleneck, refining it at fixed $K{=}10$ is sufficient, and enlarging $K$ in the
same step does not further lower the floor. To directly verify the complementary
regime in which $K$ itself is the bottleneck, we train at $N_0=65$ ($63$ interior
spectral degrees of freedom) and scan $K\in\{4,6,8,12\}$. Three additional
diagnostic experiments complete the Exp.~5 picture. The analytic-propagator advantage
of FrFNO is preserved on every axis.

\subsubsection{Exp.~5a: $N\to\infty$ convergence to a constant floor}

Corollary~\ref{cor:plateau} predicts that, at fixed $K$ and $N_0$, refining the evaluation
grid $N$ drives the error to a nonzero constant floor rather than to zero.
\Cref{fig:exp5-conv} plots the rel.$L^2$ error of the baseline configuration
($K{=}10$, $N_0=17$) against $N\in\{17,65,129,257\}$. Both networks drop
sharply from $N{=}17$ to $N{=}65$ and then plateau: the FrFNO error changes from
$10.50\%$ at $129^2$ to $10.58\%$ at $257^2$ (increment $+0.08\%$), and the FNO
error from $20.34\%$ to $21.06\%$ (increment $+0.72\%$). The FrFNO floor
($10.6\%$) is less than half the FNO floor ($21.1\%$). This directly confirms the
plateau prediction of Corollary~\ref{cor:plateau}.

\begin{figure}[htbp]
\centering
\includegraphics[width=0.7\textwidth]{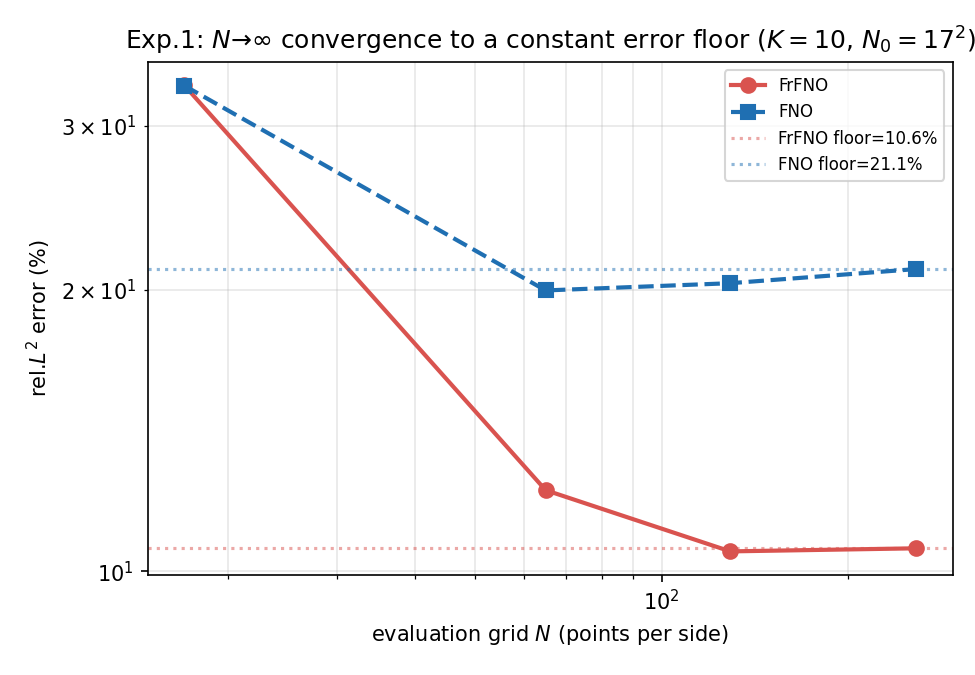}
\caption{Exp.~5a: error versus evaluation grid $N$ for the baseline ($K{=}10$,
$N_0=17$). Both networks plateau past $129^2$; the FrFNO floor ($10.6\%$) is
less than half the FNO floor ($21.1\%$).}
\label{fig:exp5-conv}
\end{figure}

\subsubsection{Exp.~5b: residual and full-field tails have the same Sobolev order}

Proposition~\ref{prop:reg} predicts that the residual $r=u-\Ebase$ and the full
solution $u$ carry the \emph{same} Sobolev order: a same-order principal-part
perturbation cancels the $2s$ derivatives supplied by the fractional heat
semigroup, so the tail ratio is a $K$-independent constant,
\begin{equation}\label{eq:tail-ratio}
\frac{\sigma_K(r)}{\sigma_K(u)}\le C(\eta,\beta,T),
\end{equation}
strictly below one only in the weakly perturbed regime. Here
$\sigma_K(f)^2=\sum_{\sqrt{i^2+j^2}\ge K}|\hat f_{ij}|^2$ is the
unnormalized DST tail energy, matching the theory. We measure both tails on the
true-solution residual at $N_0=65^2$ on a dense set of truncation modes
$K=2,4,\dots,40$ (\cref{fig:exp5-spec}).

The measurement confirms the same-order conclusion. At high wavenumbers
($K\ge12$) the ratio is essentially flat, $\sigma_K(r)/\sigma_K(u)\approx0.93\sim
1.04$, i.e.\ the two tails decay at the same rate; over the low band $K=2\sim 12$
the ratio rises from $0.32$ to $\approx1$
as $K$ crosses the dissipation cutoff, showing that the base field absorbs the
\emph{low}-mode linear evolution (where its energy is concentrated) but not the
high-mode tail. Two controls isolate the mechanism. With the nonlinear advection
switched off ($\beta=0$, variable $\nu$) the ratio is again asymptotically flat
(slope $\approx+0.5$ over the low band, saturating at high $K$): the variable
coefficient alone, being a same-order perturbation, supplies no extra
regularity. With constant coefficient and no nonlinearity ($\nu\equiv1,\beta=0$)
the residual is only numerical noise and the ratio collapses to $\approx0.03$,
verifying the implementation and the degenerate case $r\equiv0$ of the
proposition. The resolution advantage of FrFNO is therefore not a steeper
residual tail; as Proposition~\ref{prop:floor} states, it comes from the
resolution-independent table filling the out-of-band modes under
super-resolution, consistent with the much smaller degradation of FrFNO than the
matched FNO when the evaluation grid is refined (Exp.~5a and the main-text
Burgers comparison).

\begin{figure}[htbp]
\centering
\includegraphics[width=0.7\textwidth]{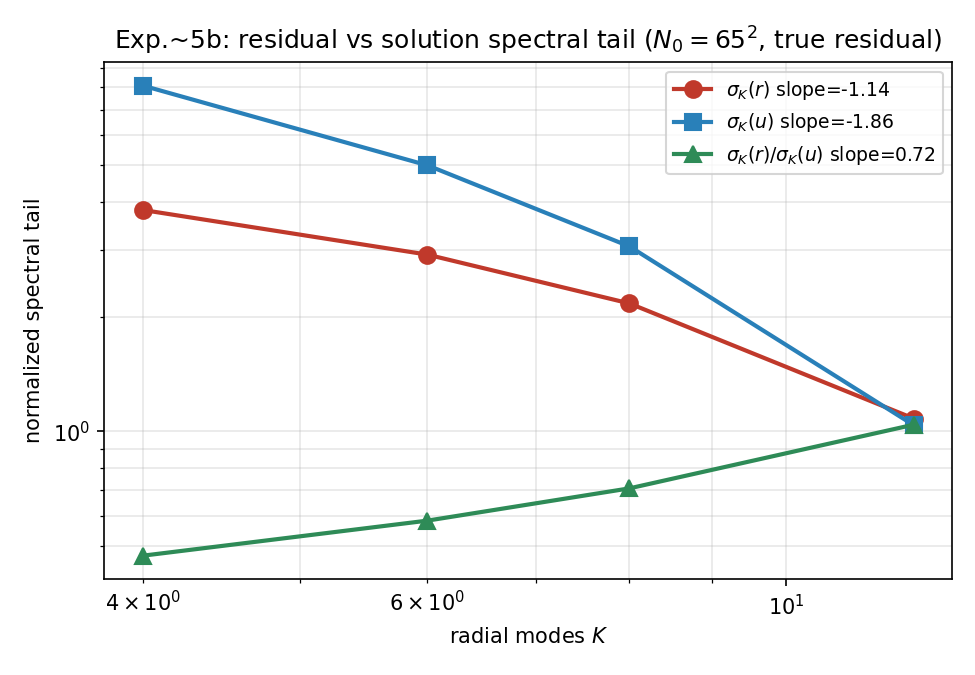}
\caption{Exp.~5b: unnormalized DST tail energies $\sigma_K(u)$ (full solution),
$\sigma_K(r)$ (true-solution residual), and their ratio versus the truncation
$K$ at $N_0=65$ (nonlinear Burgers). The two tails have the same high-$K$ slope
and the ratio saturates near one for $K\ge12$, confirming the same-order
prediction of Proposition~\ref{prop:reg}; the ratio is below one only on the low
modes that the base field represents exactly.}
\label{fig:exp5-spec}
\end{figure}

\subsubsection{Exp.~5c: $K$ as the bottleneck at $N_0=65$}

To verify the complementary half of the $\min(K,N_0)$ statement, we train FrFNO and
FNO at $N_0=65$ ($63$ interior spectral degrees of freedom) with
$K\in\{4,6,8,12\}$, so that $K<N_0$ and $K$ is the active bottleneck.
\Cref{tab:exp5-kbottleneck} reports the error floor at the $257^2$ evaluation tier,
and \cref{fig:exp5-kbottleneck} contrasts this regime with the $N_0=17$ regime
where $N_0$ is the bottleneck.

In the $K$-bottleneck regime (left panel of \cref{fig:exp5-kbottleneck}), the FrFNO
floor drops from $6.57\%$ at $K{=}4$ to $5.36\%$ at $K{=}6$ and then saturates
($5.35\%$ at $K{=}8$, $5.32\%$ at $K{=}12$). The initial drop confirms that widening
$K$ lowers the floor when $K$ is the bottleneck, exactly as Corollary~\ref{cor:plateau}
predicts. The saturation past $K{=}6$ reflects the limited spectral content of
the retained-band correction: once $K$ exceeds the number of modes on which the
residual network carries appreciable signal (the out-of-band modes being supplied
by the propagator table), further widening adds no useful signal and the floor
stabilizes. The FNO floor follows the same
qualitative pattern ($9.95\%\to8.89\%\to9.35\%\to10.70\%$) but remains uniformly
higher, and at large $K$ it rises again through overfitting.

In the $N_0$-bottleneck regime (right panel, fixed $17^2$ labels), widening $K$
\emph{raises} the floor for both networks (FrFNO $10.58\to18.52\%$; FNO
$21.06\to51.65\%$), because the extra modes have no genuine high-frequency signal to
fit and only increase estimation variance. The sharp contrast between the two panels
directly verifies the $\min(K,N_0)$ bottleneck mechanism of Corollary~\ref{cor:plateau}.

\begin{table}[htbp]
\centering\small
\caption{Exp.~5c: error floor (\%, rel.$L^2$ at $257^2$) at $N_0=65$ with
$K\in\{4,6,8,12\}$. Here $K<N_0$ so $K$ is the active bottleneck; the FrFNO floor
drops from $K{=}4$ to $K{=}6$ and then saturates.}
\label{tab:exp5-kbottleneck}
\setlength\tabcolsep{6pt}
\begin{tabular}{l|cccc}
\toprule
 & $K{=}4$ & $K{=}6$ & $K{=}8$ & $K{=}12$\\
\midrule
FrFNO & $6.57$ & $\mathbf{5.36}$ & $5.35$ & $5.32$\\
FNO   & $9.95$ & $8.89$ & $9.35$ & $10.70$\\
\bottomrule
\end{tabular}
\end{table}

\begin{figure}[htbp]
\centering
\includegraphics[width=0.95\textwidth]{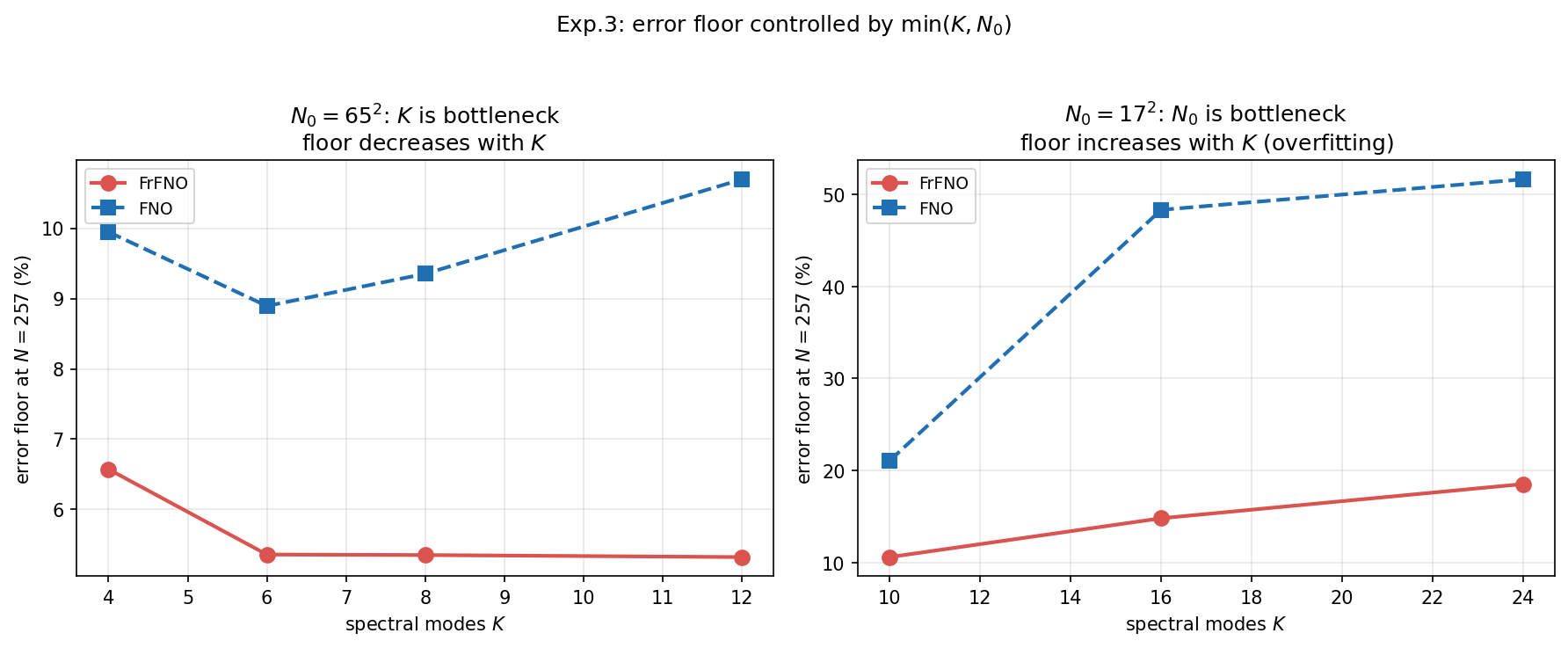}
\caption{Exp.~5c: error floor versus $K$ in two regimes. \emph{Left}: $N_0=65$,
where $K$ is the bottleneck---the FrFNO floor decreases with $K$ before saturating.
\emph{Right}: $N_0=17$, where $N_0$ is the bottleneck---widening $K$ raises the
floor through overfitting. The contrast directly verifies the $\min(K,N_0)$ mechanism.}
\label{fig:exp5-kbottleneck}
\end{figure}

\subsection{Ablation study}

\paragraph{Single-knockout ablation, short and long horizon}
\cref{tab:c1} removes each component in turn at both horizons. Removing the
analytic base is by far the dominant change ($9.49\%\to22.53\%$ at $T=0.015$ and
$22.69\%\to34.99\%$ at $T=0.06$). The secondary components are near-neutral on the
short horizon (DynFrac2d costs $0.41$ point and neighbor orders lie within seed
noise) but matter more at the four-times-longer horizon, where the frozen
propagator is less exact: DynFrac2d and neighbor orders
each cost $2.6\sim 3.7$ points at $129^2$. The neighbor fields leave the
training-resolution error almost unchanged but degrade zero-shot super-resolution,
so they carry cross-resolution order context and are retained. There is no
separate ``full-field output'' row, as that configuration is exactly the matched FNO;
and the DynFrac2d filter exponent ($\alpha$, $s$, or both) moves the result by less
than $0.1$ point.

\begin{table}[htbp]
\centering\small
\caption{Single-knockout ablation at the short horizon ($T=0.015$, common 3000-step
budget) and the long horizon ($T=0.06$, matched $1.88\times10^6$-parameter backbone, $128$
fields/order, $5000$ steps, $32$ test fields), \%. Absolute values sit above the
longer-budget main tables; the \emph{relative} attribution is what the
ablation establishes.}
\label{tab:c1}\label{tab:c2b}
\setlength\tabcolsep{2pt}
\begin{tabular}{p{0.24\textwidth}|ccc|ccc}
\toprule
 & \multicolumn{3}{c|}{short $T=0.015$} & \multicolumn{3}{c}{long $T=0.06$}\\
Variant & $17^2$ & $65^2$ & $129^2$ & $17^2$ & $65^2$ & $129^2$\\
\midrule
full FrFNO & 2.575 & 7.727 & 9.491 & 5.090 & 20.130 & 22.694\\
\ \ no analytic base $u_{\rm base}$ & 4.258 & 18.745 & 22.534 & 4.523 & 29.094 & 34.987\\
\ \ no DynFrac2d feature & 2.760 & 8.071 & 9.905 & 5.351 & 23.121 & 25.330\\
\ \ no neighbor orders & 2.664 & 7.259 & 9.320 & 5.203 & 23.913 & 26.434\\
\bottomrule
\end{tabular}
\end{table}

\paragraph{Seeds, data volume, model size, and hyperparameters}
Across three seeds FrFNO reaches $9.49/9.70/9.84\%$ at $129^2$ (worst $9.84\%$),
while the best FNO seed is at least $21.7\%$: the worst FrFNO seed still beats the
best FNO seed by a wide margin. Error decreases monotonically with training fields
per order ($11.29\%$ at $16$, $9.89\%$ at $64$, $9.49\%$ at $128$), and a compact
$701$K network ($K=6$) already attains $9.36\%$, on par with the $1881$K model: the
propagator table supplies the out-of-band modes and the residual network only
represents a small-amplitude correction inside the retained band, so a narrow
band suffices (Proposition~\ref{prop:floor}). The order grid $3\times3/5\times5/9\times9$
gives $10.26/9.89/9.49\%$ (saturated by $9\times9$); $K=6/8$ and learning rates
$10^{-3}/3\times10^{-3}$ move the result by less than one point. The method is thus
insensitive to these secondary hyperparameters once the propagator is present.
\cref{tab:rigor} summarizes the sweeps.

\begin{table}[htbp]
\centering
\caption{Summary of the rigor sweeps at $129^2$ (\%, common budget).}
\label{tab:rigor}
\begin{tabular}{lll}
\toprule
Sweep & Settings & $129^2$ error range\\
\midrule
Seeds & FrFNO 3 seeds & $9.49$--$9.84$\\
 & FNO 3 seeds & $21.71$--$22.67$\\
Data & $16/32/64/128$ fields per order & $11.29/10.06/9.89/9.49$\\
Model size & $701$K ($K=6$) / $1217$K ($K=8$) & $9.36/9.25$\\
Order grid & $3\times3$ / $5\times5$ / $9\times9$ & $10.26/9.89/9.49$\\
Learning rate & $10^{-3}$ / $3\times10^{-3}$ & $9.50/10.16$\\
\bottomrule
\end{tabular}
\end{table}

\subsection{Equation-class generality}

The equation-class generality results are reported in Section~5.9 of the main text: eight further equation settings share the same fractional dissipative principal part $-\nu(x)(-\Delta)^su$ and differ only in the zero-order or system structure, with FrFNO outperforming the matched FNO in every setting.

\paragraph{Equation definitions.} All settings use the Caputo time derivative ${}^C D_t^\alpha$ and the same training protocol (64 fields per order pair, 3000 steps, batch 64, width 32, depth 4, $K=7$ modes), unless noted otherwise. Initial conditions are drawn from a KL expansion of a Gaussian random field with correlation length $0.35$ and smoothness $10^{-4}$.

\begin{enumerate}
\item \textbf{Periodic Riesz Burgers.} On $\mathbb T^2=[0,1)^2$ with periodic boundary conditions,
\begin{equation}
{}^C D_t^\alpha u=-\nu(x)(-\Delta_R)^s u-\beta u\,\partial_x u,
\end{equation}
where $(-\Delta_R)^s$ is the periodic Riesz fractional Laplacian with FFT symbol $(k_x^2+k_y^2)^s$. Parameters: $\beta=3$, $T=0.015$, $\alpha\in\{0.7,0.85,1.0\}$, $s\in\{0.5,0.75,1.0\}$. Train at $16^2$, zero-shot to $32^2,64^2$.

\item \textbf{Linear fractional diffusion.} On $\Omega=(0,1)^2$ with homogeneous Dirichlet BC,
\begin{equation}
{}^C D_t^\alpha u=-\nu(x)(-\Delta)^s u.
\end{equation}
Parameters: $T=0.02$, $\alpha\in\{0.7,0.85,1.0\}$, $s\in\{0.5,0.75,1.0\}$. Train at internal $15^2$, zero-shot to $31^2,63^2$.

\item \textbf{Fisher--KPP.} On $\Omega=(0,1)^2$ with homogeneous Dirichlet BC,
\begin{equation}
{}^C D_t^\alpha u=-\nu(x)(-\Delta)^s u+\rho\,u(1-u),\qquad \rho=6.
\end{equation}
Parameters: $T=0.02$, same order grids as the linear case.

\item \textbf{Allen--Cahn.} On $\Omega=(0,1)^2$ with homogeneous Dirichlet BC,
\begin{equation}
{}^C D_t^\alpha u=-\nu(x)(-\Delta)^s u+\rho(u-u^3),\qquad \rho=4.
\end{equation}
Parameters: $T=0.02$, same order grids.

\item \textbf{Two-component vector Burgers.} On $\Omega=(0,1)^2$ with homogeneous Dirichlet BC,
\begin{equation}
{}^C D_t^\alpha u=-\nu_u(x)(-\Delta)^s u-(u\,\partial_x u+v\,\partial_y u),\quad
{}^C D_t^\alpha v=-\nu_v(x)(-\Delta)^s v-(u\,\partial_x v+v\,\partial_y v),
\end{equation}
with independent variable diffusivities $\nu_u,\nu_v$ per component. Parameters: $T=0.05$, frozen coefficient $c_0=0.25$ in the propagator, same order grids.

\item \textbf{Non-symmetric coupled system.} On $\Omega=(0,1)^2$ with homogeneous Dirichlet BC,
\begin{equation}
{}^C D_t^\alpha \mathbf U=-\nu(x)(-\Delta)^s\mathbf U-J\mathbf U+\mathcal N(\mathbf U),\quad
\mathbf U=(u,v)^\top,\quad
\mathcal N(\mathbf U)=-(\mathbf U\cdot\nabla)\mathbf U,
\end{equation}
with the constant non-symmetric coupling matrix
\begin{equation}
J=\begin{pmatrix}1.2&0.8\\0.3&0.6\end{pmatrix},
\end{equation}
whose eigenvalues are real and positive (uniformly dissipative). The full matrix propagator diagonalizes $J$ once and applies the scalar Mittag--Leffler multiplier in eigen-coordinates. Parameters: $T=0.05$, $c_0=1.0$, same order grids.

\item \textbf{2D incompressible fractional NS (vorticity).} On $\mathbb T^2$ with periodic BC,
\begin{equation}
{}^C D_t^\alpha \omega=-\nu(-\Delta_R)^s\omega-u\,\partial_x\omega-v\,\partial_y\omega,\quad
-\Delta\psi=\omega,\quad (u,v)=(\partial_y\psi,-\partial_x\psi),
\end{equation}
where the velocity is recovered from vorticity by the spectral Biot--Savart inversion. Parameters: $T=0.15$, $\nu\in\{0.10,0.15,0.20\}$, $\alpha\in\{0.7,0.85,1.0\}$, $s\in\{0.6,0.8,1.0\}$. Train at $16^2$, zero-shot to $32^2,64^2$.

\item \textbf{3D incompressible fractional NS (vector vorticity).} On $\mathbb T^3$ with periodic BC,
\begin{equation}
{}^C D_t^\alpha\boldsymbol\omega=-\nu(-\Delta_R)^s\boldsymbol\omega+(\boldsymbol\omega\cdot\nabla)\mathbf u-(\mathbf u\cdot\nabla)\boldsymbol\omega,\quad
\boldsymbol\omega=\nabla\times\mathbf u,\quad \nabla\cdot\mathbf u=0,
\end{equation}
where the velocity is recovered by the 3D spectral Biot--Savart inversion $\hat{\mathbf u}=i(\mathbf k\times\hat{\boldsymbol\omega})/|\mathbf k|^2$, and the new term $(\boldsymbol\omega\cdot\nabla)\mathbf u$ is vortex stretching. Parameters: $T=0.10$, $\nu=0.15$, $\alpha\in\{0.85,1.0\}$, $s\in\{0.75,1.0\}$. Deliberately small-scale: train at $12^3$, zero-shot to $16^3,24^3$, 1500 steps, width 24, depth 3, $K=5$.
\end{enumerate}

In every setting the injected propagator is the exact frozen-coefficient modal Mittag--Leffler / L1 multiplier (scalar for settings 1--5 and 7, eigen-rotated matrix for setting 6, componentwise scalar for setting 8); the residual network learns only the variable-coefficient, nonlinear, and nonlocal parts.

\paragraph{Reference solvers.} All ground-truth data are generated by a GPU spectral reference solver that shares the same L1 implicit-explicit (IMEX) skeleton: the fractional dissipative principal part is treated implicitly in the modal basis, while variable-coefficient corrections, reaction, and advection are treated explicitly, with two Picard inner iterations per time step. Variable coefficients are frozen at their spatial maximum $c=\max_x\nu(x)$ for the implicit solve, and the remainder $(\nu-c)(-\Delta)^s u$ is added explicitly. The time step count $N_t$ scales with resolution ($400$ at the training grid, $800$/$1600$ at the super-resolution grids for 2D settings; $200$/$300$/$460$ for the 3D setting). The spatial discretization and nonlinear treatment differ by boundary type:

\begin{itemize}
\item \textbf{Dirichlet settings} (2--6): DST-I sine transform with the five-point Laplacian symbol $\lambda_k=4\sin^2(k_x\pi/2N)+4\sin^2(k_y\pi/2N)$; advection uses a first-order upwind difference; reaction terms are evaluated pointwise explicitly.
\item \textbf{Periodic settings} (1, 7, 8): FFT with the Riesz fractional-Laplacian symbol $|\mathbf k|^{2s}$; advection uses periodic upwind; for the Navier--Stokes settings the velocity is recovered from vorticity at every step by the spectral Biot--Savart inversion ($\hat\psi=\hat\omega/|\mathbf k|^2$ in 2D; $\hat{\mathbf u}=i\mathbf k\times\hat{\boldsymbol\omega}/|\mathbf k|^2$ in 3D, which enforces $\nabla\cdot\mathbf u=0$).
\item \textbf{Coupled system} (6): the constant matrix $J$ is diagonalized once; in eigen-coordinates the linear part decouples into two scalar problems, each solved by the scalar L1 multiplier, and the result is rotated back.
\end{itemize}

All solvers are batched on GPU and use float32 for the time stepping with float64 for the L1 weight accumulation; the same solver produces both training and test data.

\section{Implementation details}\label{sm:impl}

This section collects all reproducibility details of the FrFNO implementation.
It covers data generation (\cref{app:data}), network architecture (\cref{app:arch}),
training protocol (\cref{app:train}), reference solver (\cref{app:ref}),
and the propagator table (\cref{app:table}).

\subsection{Data generation}\label{app:data}

\paragraph{Initial fields $u_0$}
Each initial condition is a multiscale random sine field on the unit square
with homogeneous Dirichlet boundary.
\begin{equation}\label{eq:gen-u0}
u_0(x,y)=\sum_{i=1}^{K_0}\sum_{j=1}^{K_0}
\frac{c_{ij}}{\sqrt{i^2+j^2}}\,\sin(i\pi x)\sin(j\pi y),
\qquad c_{ij}\stackrel{\text{i.i.d.}}{\sim}\mathcal N(0,1),
\end{equation}
with $K_0=8$ modes in each direction (64 modes total). The algebraic
factor $(i^2+j^2)^{-1/2}$ imposes a mild high-frequency decay. After
synthesis the field is rescaled so that the root-mean-square amplitude on
the interior grid points equals $0.5$. The boundary is pinned to zero
numerically. Training uses 128 independent fields (seeds $1000+i$) and the
held-out test set uses 48 fields (seeds drawn uniformly from $[10^5,10^6)$,
disjoint from the training seed range).

\paragraph{Test set generation}
The $N_{\rm te}=48$ held-out samples use a separate random stream (seed $2024$,
drawn after all training data). Each triple $(u_0,\nu,\alpha,s)$ is independent,
with $u_0$ from \cref{eq:gen-u0} on seeds in $[10^5,10^6)$, $\nu$ an independent
log-normal KL field with its own coefficient vector $\xi^{(\rm te)}$, and orders
drawn \emph{continuously}, $\alpha\sim U[0.55,0.95]$, $s\sim U[0.35,0.78]$, almost
all off the $9\times9$ training grid so as to probe FiLM interpolation directly.
The set is thus disjoint in all three input dimensions, and test references use
the same three resolution pairs, $N_t=400/800/1600$ for $N=16/64/128$.

\paragraph{Diffusion coefficient fields $\nu(x,y)$}
The coefficient is a smooth \emph{log-normal} random field generated by a
truncated Karhunen--Lo\`eve (KL) expansion. The first $K_{\rm KL}=20$
cosine eigenfunctions are selected from an $8\times8$ candidate pool, ordered
by the eigenvalue $1+m^2+n^2$.
\begin{equation}\label{eq:gen-nu-kl}
g(x,y)=\sigma_g\sum_{k=1}^{K_{\rm KL}}\sqrt{e_k}\,\xi_k\,\phi_k(x,y),
\qquad \xi_k\stackrel{\text{i.i.d.}}{\sim}\mathcal N(0,1),
\end{equation}
where $\phi_k(x,y)=\cos(m_k\pi x)\cos(n_k\pi y)$ (normalized to unit
$L^2$ norm), $e_k=(1+m_k^2+n_k^2)^{-\beta_{\rm KL}}$ with spectral decay
$\beta_{\rm KL}=1.0$ (normalized so that $\sum_k e_k=1$), and
$\sigma_g=0.35$ controls the log-amplitude. The diffusivity
is then
\begin{equation}\label{eq:gen-nu-exp}
\nu(x,y)=\frac{\exp(g(x,y))}{\frac{1}{|\Omega|}\int_\Omega \exp(g)\,dx},
\end{equation}
which guarantees strict positivity and unit spatial mean. The analytic base
field is therefore always evaluated at $\bar\nu=1$, and the network learns
only the spatially heterogeneous residual. The coefficient contrast is
$\eta=\|\delta\nu\|_{L_\infty}/\bar\nu\approx 0.35$ in practice.

\paragraph{Order grid}
The fractional orders are sampled on a uniform $9\times9$ grid.
$\alpha\in[0.55,0.95]$ and $s\in[0.35,0.78]$. Each training field is
paired with every order setting, giving $128\times81=10\,368$ reference
pairs at the training resolution. At each gradient step one order pair is
drawn uniformly from the grid and shared across the mini-batch.

\paragraph{Training-time pairing of $(u_0,\nu,u_{\rm base},\alpha,s)$}
The diffusivity is \emph{not} fixed per initial condition. Each
$(\alpha,s,u_0)$ triple receives an independent $\nu$. During dataset
construction a coefficient array $\Xi\in\mathbb R^{9\times9\times128\times
K_{\rm KL}}$ ($K_{\rm KL}=20$) is drawn i.i.d., and the diffusivity of triple
$(\alpha_i,s_j,u_0^{(k)})$ uses the slice $\Xi[i,j,k,:]$, so the same $u_0$ is
paired with different $\nu$ at different orders, yielding
$9\times9\times128=10\,368$ distinct combinations rather than $128$ fixed
$(u_0,\nu)$ pairs. At each gradient step one order pair is drawn uniformly
(shared across the batch), $64$ field indices without replacement, and the
per-triple diffusivity $\texttt{make\_nu}(\Xi[i,j,k_m,:])$, current and four
neighbor-order base fields, and normalized residual target are assembled on the
fly. The analytic bases are computed from the precomputed table, and only the
table and reference solutions are stored.

\subsection{Network architecture}\label{app:arch}

The residual network is the \texttt{DualNet} module. All main-table
experiments use \texttt{DualNet(7, seed=1)}, i.e.\ 7 explicit input
channels and all other parameters at their defaults.

\paragraph{Explicit input channels (7)}
The stacked input tensor has shape $(B,7,H,W)$ and contains, in order.
(1)~the initial field $u_0$, (2)~the analytic base field $\Ebase$ at the
current order $(\alpha,s)$, (3)~the diffusivity field $\nu$, and
(4)--(7)~the analytic base fields at the four nearest order neighbors
$(\alpha\pm\Delta\alpha,s\pm\Delta s)$.

\paragraph{Internally appended channels (6)}
Before the lifting layer, six additional channels are concatenated
internally.
\begin{itemize}
\item two coordinate channels $g_y,g_x$ (linearly spaced grids broadcast to
  the full spatial shape).
\item two DynFrac2d output channels (see below), obtained by fractional
  low-pass filtering of $g_y$ and $g_x$.
\item two constant channels carrying the scalar values of $\alpha$ and $s$
  broadcast over the spatial domain.
\end{itemize}
The tensor fed to the lifting layer therefore has
$7+2+2+2=13$ channels.

\paragraph{DynFrac2d fractional feature layer}
For each of the first \texttt{field\_dim} input channels (default 2, i.e.\
$g_y$ and $g_x$), the module computes
\begin{equation}\label{eq:dynfrac}
\mathcal F_\alpha[f]=\mathrm{iFFT}_{2\mathrm d}\!\Big[
\widehat f(\boldsymbol k)\,e^{-\alpha\log|\boldsymbol k|}\,
\mathbf 1_{\{|\boldsymbol k|_\infty<K\}}\Big],
\end{equation}
where $K=10$ is the same mode truncation used by the spectral convolutions.
The multiplier $e^{-\alpha\log|\boldsymbol k|}=|\boldsymbol k|^{-\alpha}$
is an $\alpha$-controlled algebraic low-pass filter. Each filtered field is
appended as an additional channel. Setting \texttt{field\_dim=0} disables
the layer. All spectral models in this paper (FrFNO, FNO and PINO) use
\texttt{field\_dim=2}.

\paragraph{Lifting, spectral blocks, FiLM, projection}
A pointwise two-layer perceptron lifts the 13 channels to the hidden width $w=48$
($13\!\to\!128\!\to\!48$, equivalent to $1\times1$ convolutions). Four spectral
blocks each add a $\mathrm{SpecConv}_{K=10}$ (learned complex weights
$w_1,w_2\in\mathbb C^{48\times48\times10\times10}$ on the lowest modes of both
$k_x$ corners) to a $1\times1$ convolution, then apply FiLM
$h\leftarrow(1+f_\gamma(e))\odot h+f_\beta(e)$, with
$e=\mathrm{MLP}_{2\to32\to32}(\alpha,s)$ (GELU) and per-block affine heads
$f_\gamma,f_\beta:\mathbb R^{32}\to\mathbb R^{48}$ (the ``$1+$'' initialization
preserves the identity map); LeakyReLU follows the first three blocks. Zero-padding
$p=4$ is applied before the blocks and removed after the last, and a pointwise
$48\!\to\!128\!\to\!1$ perceptron returns the single-channel residual.

\paragraph{Parameter count}
The full FrFNO (\texttt{DualNet(7)}, width 48, depth 4, modes 10,
\texttt{field\_dim=2}) has approximately $1.88\times10^6$ trainable
parameters. The matched FNO baseline uses the identical backbone with
\texttt{in\_dim=2}, \texttt{field\_dim=2}, and nearly the same parameter count.

\subsection{Baseline surrogate architectures and hyperparameters}\label{app:baselines}

All baselines solve the same conditional map
$(u_0,\nu,\alpha,s)\mapsto u(T)$ with one network covering the entire order
grid, are trained on the same $17^2$ labels, the same per-step random draw,
the same $8\,000$ steps and mini-batch $64$, and the same cosine schedule to
$5\times10^{-5}$ in float32. Widths and depths are chosen so that all models
lie in the same $0.38$--$2.1\times10^6$ parameter band. Following the
recommendation of each method's reference implementation, the CNO uses
AdamW with an $L^1$ loss. PINO uses Adam with the data MSE plus the $0.5$-weighted PDE residual
and the remaining three baselines (PDNO, DeepONet, U-Net) use Adam with the plain
MSE loss. Spatial models receive the two fields $(u_0,\nu)$ together with the
broadcast scalar channels $(\alpha,s)$. Spectral models use the channels-last
convention $(B,H,W,C)$.

\paragraph{PINO (physics-informed neural operator)}
PINO reuses the FNO spectral backbone, \texttt{DualNet(in\_dim=2, modes=10,
width=48, n\_layers=4, field\_dim=2, pad=4, emb=32, seed=7)}. The inputs $(u_0,\nu)$ are
joined by two coordinate grids, two $\alpha$-controlled fractional features
\cref{eq:dynfrac}, and two constant $(\alpha,s)$ channels (eight total), lifted
$8\!\to\!128\!\to\!48$. Four $\mathrm{SpecConv}_{K=10}+1\times1$ blocks with
per-block FiLM and a $48\!\to\!128\!\to\!1$ projection follow. It differs from FNO
only in the objective. To the data MSE it adds a weak PDE residual on a
bicubic-upsampled $33^2$ collocation grid (first-order upwind advection, DST
fractional Laplacian, normalized by $\sigma_{\rm train}^2$),
$L_{\rm data}+0.5\,L_{\rm PDE}$.

\paragraph{PDNO (learned-symbol pseudo-differential operator)}
\texttt{PDNO2d(width=72, n\_layers=6, hid=160, emb=32, seed=6)} lifts two
coordinate grids, $(u_0,\nu)$ and two order channels (six total) through
$6\!\to\!128\!\to\!72$. Six blocks evaluate a continuous complex symbol from the
five features $(\log(1+|\boldsymbol k|),k_x/(|\boldsymbol k|+\varepsilon),
k_y/(|\boldsymbol k|+\varepsilon),\alpha,s)$ via a $5\!\to\!160\!\to\!160\!\to\!2w$
GELU MLP at every integer \texttt{rfft2} mode (hence grid-independent), multiply
the spectrum channel-wise, and add a $1\times1$ bypass with per-block FiLM. A
$72\!\to\!128\!\to\!1$ projection closes the network.

\paragraph{CNO (convolutional neural operator)}
We use the vendored \emph{simplified} \texttt{CNO2d} (bicubic up/down-sampling with
antialiasing), \texttt{CNOWrap(size=17, ch=32, use\_bn=False, seed=3)}. Four input
channels lift to 16 and pass through three encoder/decoder levels of widths
$16/32/64$ with four residual blocks per level, four bottleneck blocks, $3\times3$
convolutions, and no batch norm. Inputs at other resolutions are bicubic-downsampled
to the fixed $17^2$, evaluated, and upsampled back. The official \emph{classic}
variant with Kaiser-windowed sinc filters was statistically tied and gave no
zero-shot transfer here, so the simpler variant is reported throughout.

\paragraph{DeepONet (branch--trunk)}
\texttt{DeepONet2d(p=128, sensor=17, w=256, seed=4)} bilinearly downsamples to a
fixed $17^2$ sensor grid and appends $(\alpha,s)$ (four channels). A branch CNN
(Conv--pool--Conv--pool--Conv--FC to $p=128$ coefficients, LeakyReLU, $2\times2$
average pooling) is combined at every query point with a trunk MLP on
$(x,y,\alpha,s)$ ($4\!\to\!256\!\to\!256\!\to\!128$, LayerNorm, LeakyReLU) via
$p^{-1/2}\sum_m b_m t_m(\boldsymbol x)$, so the output follows the query grid while
the branch sensor stays at $17^2$.

\paragraph{U-Net (pure-convolution baseline)}
\texttt{UNet2d(base=48, seed=5)} appends $(\alpha,s)$ (four channels) and uses a
three-level encoder--decoder with skip connections. DoubleConv $3\times3$ blocks,
encoder widths $48/96/192$, a $192$-channel bottleneck, $2\times2$ average-pool
down and bilinear up with channel concatenation, and a $1\times1$ output
$96\!\to\!1$. The result is bilinearly restored to the input size (arbitrary
resolutions, with no discretization guarantee).

\begin{table}[htbp]
\centering\small
\caption{Architectures and training settings of the five baseline surrogates
used in the main tables of Sections 5.2 and 5.3. Common protocol: identical training fields and
per-step random draw, $8\,000$ steps, mini-batch $64$, cosine annealing to
$5\times10^{-5}$, float32, normalized targets.}
\label{tab:baselines}
\setlength\tabcolsep{2pt}
\begin{tabular}{l p{1.6in} c l l}
\toprule
Model & Key configuration & Par.\ ($10^6$) & Optimizer & Loss\\
\midrule
PINO     & FNO spectral backbone; PDE residual on $33^2$, $\lambda=0.5$ & 1.880 & Adam $1.5\mathrm e{-3}$ & MSE$+0.5\,r$\\
PDNO     & width 72, 6 learned-symbol blocks, symbol MLP hidden 160 & 0.380 & Adam $1.5\mathrm e{-3}$ & MSE\\
CNO      & 3 levels, channels $16/32/64$, 4 residual blocks, no BN & 2.012 & AdamW $1\mathrm e{-3}$ & $L^1$\\
DeepONet & branch CNN + trunk MLP, $p=128$, $w=256$ & 2.073 & Adam $1.5\mathrm e{-3}$ & MSE\\
U-Net    & 3 levels, base width 48, skip connections & 1.828 & Adam $1.5\mathrm e{-3}$ & MSE\\
\bottomrule
\end{tabular}
\end{table}

\subsection{Training protocol}\label{app:train}

\begin{itemize}
\item \textbf{Optimizer:} Adam with learning rate $1.5\times10^{-3}$ and
  weight decay $10^{-5}$.
\item \textbf{Schedule:} cosine annealing to $5\times10^{-5}$ over the full
  training run.
\item \textbf{Steps:} $8\,000$ gradient steps for the main tables (Sections 5.2--5.4).
  The short- and long-horizon ablations use $3\,000$ steps ($T=0.015$) and
  $5\,000$ steps ($T=0.06$), the left and right blocks of \cref{tab:c1}.
\item \textbf{Mini-batch:} at each step one order pair $(\alpha,s)$ is drawn
  uniformly from the $9\times9$ grid, and 64 initial/coefficient field
  pairs are drawn without replacement from the 128 training fields. The same
  order pair is shared across all 64 fields in the mini-batch.
\item \textbf{Training resolution:} $N=16$ intervals per axis, i.e.\ a $17^2$
  grid ($15\times15$ interior points), with $N_t=400$ time steps for the
  reference solver.
\item \textbf{Loss:} mean-squared error on the residual
  $\|\Rres-\big(u(T)-\Ebase\big)\|_2^2$ for FrFNO. The FNO baseline
  minimizes MSE on the full field $\|\widetilde\Rres-u(T)\|_2^2$.
\item \textbf{Precision:} single precision (float32) for network training and
  inference. The propagator table is also stored in float32.
\end{itemize}

\subsection{Training and inference algorithms}

Algorithm~\ref{alg:offline} builds the cached propagator table once, on a
$1$-D grid of $z=\lambda^s$ independent of spatial resolution. It is invoked
before any training or inference. Algorithm~\ref{alg:train} then trains the
residual network at the low training resolution $N_0$. At each step the exact
analytic base field is computed from the table and detached from the
computational graph, and only the residual $r=u-\Ebase$ is learned.
Algorithm~\ref{alg:infer} performs zero-shot super-resolution. The same table
is queried at the new resolution $N$, and the trained network is evaluated
without retraining or fine-tuning.

\begin{algorithm}[htbp]
\caption{Offline resolution-independent propagator table}\label{alg:offline}
\begin{algorithmic}[1]
\Require orders $\{\alpha_i\}_{i=1}^{n_a}$, $\{s_j\}_{j=1}^{n_s}$; steps $N_t$; $\bar\nu$; 1-D grid $\{z_m\}_{m=1}^{M}$
\Ensure table $\mathcal T[i,j,m]$
\For{$i=1..n_a$, $j=1..n_s$}
  \State $b_\ell\gets\dfrac{(\ell+1)^{1-\alpha_i}-\ell^{1-\alpha_i}}{\Gamma(2-\alpha_i)},\ \ell=0,\dots,N_t-1$ \Comment{$L_1$ weight}
  \For{$m=1..M$}
    \State $\mu\gets \bar\nu\,z_m$;\quad $g^0\gets1$
    \For{$n=1..N_t$}
      \State $g^n\gets\dfrac{b_{n-1}g^0+
        \sum_{\ell=1}^{n-1}(b_{n-\ell-1}-b_{n-\ell})g^\ell}{b_0+\tau^{\alpha_i}\mu}$
    \EndFor
    \State $\mathcal T[i,j,m]\gets g^{N_t}$
  \EndFor
\EndFor
\State \Return $\mathcal T$
\end{algorithmic}
\end{algorithm}

\begin{algorithm}[htbp]
\caption{FrFNO training (low resolution $N_0$)}\label{alg:train}
\begin{algorithmic}[1]
\Require training fields $\{u_0^{(b)},\nu^{(b)}\}$; order grid $\mathcal A\times
\mathcal S$; horizon $T$; table $\mathcal T$; modes $K$; steps $S_{\rm tr}$
\State generate reference pairs $(u_0,\nu,\alpha,s,u(T))$ on the $N_0=17$ training grid ($17^2$ labels, i.e.\ $N=16$ intervals) via the IMEX scheme (Section~2.4 of the main paper)
\State init $\theta$; Adam$(1.5\times10^{-3}, w_d=10^{-5})$ with cosine decay to
$5\times10^{-5}$
\For{$\mathrm{step}=1..S_{\rm tr}$}
  \State draw a batch $(\alpha,s)\sim\mathcal A\times\mathcal S$ and fields
  \State $q^{(N_0)}\gets$ five-point eigenvalues;\quad
  $g\gets\operatorname{Interp}_{\mathcal T}(\alpha,s,(q^{(N_0)})^s)$
  \State $\Ebase\gets\mathrm{iDST}[g\cdot\mathrm{DST}(u_0)]$ \Comment{exact, detached}
  \State $r\gets\Rres\bigl(u_0,\Ebase,\nu,\{\Ebase^{(\alpha\pm\Delta\alpha,s\pm\Delta s)}\};\alpha,s\bigr)$
  \State $\theta\gets$ Adam step on $\|\Ebase+r-u(T)\|_2^2$
\EndFor
\State \Return $\theta$
\end{algorithmic}
\end{algorithm}

\begin{algorithm}[htbp]
\caption{Zero-shot inference at arbitrary resolution $N$}\label{alg:infer}
\begin{algorithmic}[1]
\Require query $(u_0^N,\nu^N,\alpha,s)$; trained $\theta$; table $\mathcal T$
\State $q^{(N)}\gets$ five-point eigenvalues on the $N$-grid
\State $g^{N_t}\gets\operatorname{Interp}_{\mathcal T}(\alpha,s,(q^{(N)})^s)$
\State $\Ebase^N\gets\mathrm{iDST}_N[g^{N_t}\cdot\mathrm{DST}_N(u_0^N)]$
\State $r^N\gets\Rres$ evaluated on the $N$-grid (fixed $K$, $1\times1$ path, FiLM)
\State \Return $\hat u^N=\Ebase^N+r^N$
\end{algorithmic}
\end{algorithm}

\subsection{Reference solver}\label{app:ref}

All training and test labels are generated by a GPU-accelerated
implicit--explicit (IMEX) scheme.

\paragraph{Time discretization}
The Caputo derivative is discretized by the standard $L_1$ scheme
(see Section~2 of the main paper), the fractional diffusion term is treated
implicitly (backward Euler in the fractional sense), and the nonlinear
advection term $\beta u\partial_x u$ is treated explicitly with a first-order
upwind flux. The resulting linear system at each time step is diagonalized
by the discrete sine transform and solved in $O(N^2\log N)$ per step.

\paragraph{Space discretization}
The fractional Laplacian uses the five-point finite-difference Laplacian
symbol $q_{\boldsymbol k}^{(N)}$ (Eq.~(2.11) of the main paper) raised to
the power $s$, applied via the DST. Reference solutions are computed at three
resolution pairs.
$(N,N_t)=(16,400)$, $(64,800)$, and $(128,1600)$.

\subsection{Propagator table}\label{app:table}

The analytic propagator is precomputed as a lookup table and queried by
trilinear interpolation.

\paragraph{Table construction}
The table is built in the smooth scalar coordinate $z=\lambda^s$, independent of
the evaluation resolution. On a $41\times41$ order grid
($\alpha\in[0.45,1.05]$, $s\in[0.25,0.90]$) and a log-spaced grid
$\{z_m\}_{m=1}^{M}$ of $M=1600$ points over $z\in[1,2\times10^5]$ ($M=400$ is the
fast setting), the $L_1$ recurrence (Eq.~(2.13) of the main paper) with $\mu=\bar\nu\,z_m$ is run
for $N_t=400$ steps at each $z_m$. Since $g$ depends only on $\alpha$ and $z$, one
stored curve serves all $s$. The $(41,41,M)$ table is built once and cached
globally, and rendering it at resolution $N$ is only an $O(N^2)$ evaluation of
$z=q_{\boldsymbol k}^{s}$ on the $(N-1)^2$ interior modes
$(\cref{alg:offline})$.

\paragraph{Interpolation and cost}
A query uses trilinear interpolation of $\mathcal T$ (bilinear in $(\alpha,s)$,
linear along $z=q^s$) from the cached table, which a resolution change never
rebuilds. Building it ($41\times41\times1600$, $N_t=400$) takes about $5$~s. One
query (plus DST/iDST on $129^2$) about $0.3$~ms, negligible against the network
forward pass.

\subsection{Reproducibility details}\label{sec:repro}

All experiments use PyTorch~2.4.1 (CUDA~12.4) on an RTX~3090 (24\,GB) with 64\,GB
host memory, single precision for training and double precision for the propagator
cross-check. Training uses batch size $64$, Adam with learning rate
$1.5\times10^{-3}$ and weight decay $10^{-5}$, cosine annealing to $5\times10^{-5}$,
and $8000$ iterations for the main tables (a uniform $3000$ iteration budget is
used in the controlled rigor sweeps). Random states for data generation, weight
initialization, and batch shuffling are fixed and recorded. The rigor study repeats
training over three seeds. CNO is trained with its official prescription
($\texttt{use\_bn=False}$, AdamW$(10^{-3},w_d=10^{-8})$, $L_1$ loss), since a
generic Adam/MSE recipe diverges on order grids containing $s=1$. Every experiment
writes its driver, log, text summary, metrics archive, final weights, and an
incremental checkpoint that supports restart. Code, weights, and seeds will be
released to reproduce the main table within $\pm0.2$ percentage points.

\bibliographystyle{plainnat}
\bibliography{FrFNO_manuscript}

\end{document}